\documentclass[11pt]{article}
\usepackage{amsfonts}

\usepackage{graphics}
\usepackage{indentfirst}
\usepackage{cite}
\usepackage{latexsym}
\usepackage{amsmath,amsthm}
\usepackage{amssymb}
\usepackage[dvips]{epsfig}
\usepackage{amscd}
\usepackage{mathrsfs}

\usepackage{color}
\newtheorem{theorem}{Theorem}[section]
\newtheorem{remark}{Remark}[section]

\newtheorem{definition}{Definition}[section]
\newtheorem{lemma}[theorem]{Lemma}

\newtheorem{proposition}[theorem]{Proposition}

\newcommand{\n}{\rho}

\newcommand{\lm}{\lambda}

\renewcommand{\div}{ {\rm div }  }

\newcommand{\na}{\nabla }

\newcommand{\pa}{\partial}

\newcommand{\bt}{\begin{theorem}}
\newcommand{\bl}{\begin{lemma}}
\newcommand{\el}{\end{lemma}}
\newcommand{\et}{\end{theorem}}
\newcommand{\ga}{\gamma}

\newcommand{\OM}{\Omega}

\newcommand{\curl}{{\rm curl} }

\newcommand{\de}{\delta}
\newcommand{\ve}{\varepsilon}
\newcommand{\la}{\label}
\newcommand{\si}{\sigma}

\newcommand{\om}{\Omega}

\newcommand{\bn}{\begin{eqnarray}}
\newcommand{\en}{\end{eqnarray}}
\newcommand{\bnn}{\begin{eqnarray*}}
\newcommand{\enn}{\end{eqnarray*}}

\newcommand{\bnnn}{\begin{eqnarray*}}
\newcommand{\ennn}{\end{eqnarray*}}
\newcommand{\ben}{\begin{enumerate}}
\newcommand{\een}{\end{enumerate}}

\newcommand{\du}{\dot{u}}

\newcommand{\ba}{\begin{aligned}}
\newcommand{\ea}{\end{aligned}}
\newcommand{\be}{\begin{equation}}
\newcommand{\ee}{\end{equation}}

\def\p{\partial}
\def\norm[#1]#2{\|#2\|_{#1}}

\def\lam{\lambda}
\def\ep{\varepsilon}

\def\o{\omega}
\def\rr{\mathbb{R}^2}

\makeatletter      
\@addtoreset{equation}{section}
\makeatother       

\title{Global Existence and Incompressible Limit for the Two-Dimensional Compressible Navier-Stokes Equations in a Half-Space with Large Initial Data and Vacuum}

\date{}

\author{$\text{Qinghao L{\small EI}}^{a,b}, \text{Weirong L{\small{IANG}}}^{a,b}\thanks{Email addresses:  leiqinghao22@mails.ucas.ac.cn (Q. H. Lei), liangweirong20@mails.ucas.ac.cn (W. R. Liang) }$\\
a. School of Mathematical Sciences,\\ University of Chinese Academy of Sciences,
Beijing 100049, P. R. China;\\
b. Institute of Applied Mathematics,\\ Academy of Mathematics and Systems Science, \\
Chinese Academy of Sciences, Beijing 100190, P. R. China}

\begin{document}
\maketitle

\begin{abstract}
This paper concerns the barotropic compressible Navier-Stokes equations in a two-dimensional half-space subject to Navier-slip boundary conditions with vacuum or non-vacuum far-field density.
The global existence and large-time behavior of weak and strong solutions are established under the assumption that the bulk viscosity coefficient is sufficiently large.
It should be remarked that this result is obtained without any restrictions on the size of the initial data.
For strong solutions, we derive some a priori decay estimates for the spatial gradient of the velocity field that are uniform with respect to the bulk viscosity coefficient, which play a crucial role in establishing the time-uniform upper bound for the density.
Furthermore, we prove that, as the bulk viscosity coefficient tends to infinity, the solutions of the compressible Navier-Stokes equations converge to those of the inhomogeneous incompressible Navier-Stokes equations.
In particular, the incompressible limit for weak solutions holds without requiring the initial velocity to be divergence-free. 
\end{abstract}

\section{Introduction and main results}
We study the two-dimensional barotropic compressible
Navier-Stokes equations which read as follows:
\be\ba\la{ns}
\begin{cases}
  \rho_t + \div(\rho u) = 0,\\
  (\n u)_t + \div(\n u\otimes u) -\mu \Delta u - (\mu + \lm)\na\div u
    +\na P = 0,
\end{cases}
\ea\ee
where $t \ge 0$ is time, $x \in \OM \subset \rr$ is the spatial coordinate,
$\n=\n(x,t)$ and $u(x,t)=(u^1(x,t),u^2(x,t))$ represent the 
density and velocity of the compressible flow, respectively.
The pressure $P$ is given by
\be\la{i1}
P=R\n^\ga,
\ee
with constants $R>0,\ga>1$. Without loss of generality, we assume that
$R=1$. The shear viscosity coefficient $\mu$ and the bulk viscosity coefficient $\lam$ satisfy
the physical restrictions
\be\la{i2}
\mu>0,\quad \mu+\lam\geq 0.
\ee
For later use, we set 
\be\la{nu}
\nu \triangleq 2\mu + \lam,
\ee
which, together with (\ref{i2}), implies that
\be\la{numu}
\nu \geq \mu.
\ee
In this paper, we assume that $\OM$ is the two-dimensional half-space, namely,
\be\nonumber
\OM = \rr_+ = \{x \in \rr : x_2>0 \}.
\ee
The system is supplemented with the given initial data
\be\la{i3}
\n(x,0)=\n_0(x),\quad \n u(x,0)=\n_0u_0(x), \quad x \in \rr_+.
\ee
Let $\tilde{\rho}$ be a fixed non-negative constant.
We seek solutions $\left(\n(x,t),u(x,t)\right) $ to the problem (\ref{ns}) subject to the Navier-slip boundary conditions:
\be\la{bjtj1}\ba
u \cdot n = 0,\quad \curl u = -A u \cdot n^\bot \ \text{ on } \ \p \rr_+,
\ea\ee
and the far-field condition
\be\la{bjtj2}\ba
(\n,u)(x,t) \to (\tilde{\n},0) \ \text{ as } |x| \to \infty, \  t>0,
\ea\ee
where $A$ is a non-negative constant, $n = (0,-1)$ is the unit outward normal vector of the boundary, and $n^\bot = (1,0)$ is the unit tangential vector on the boundary.

There is a vast literature concerning the strong solvability for the multidimensional compressible Navier-Stokes system with constant viscosity coefficients.
The one-dimensional problem has been extensively studied; see \cite{H4,KS,S1,S2} and the references therein.
For the multi-dimensional problem, the local existence and uniqueness of classical solutions were proved by Nash \cite{N} and Serrin \cite{S} under the assumption that the density is away from vacuum.
When the initial density contains vacuum and may vanish in open sets, the local existence and uniqueness of strong and classical solutions were established in \cite{CCK,CK,CK2,SS,LLL} and the references therein.
The global classical solutions were first obtained by Matsumura-Nishida \cite{MN1} for initial data close to a non-vacuum equilibrium in some Sobolev space $H^s$.
Subsequently, Hoff \cite{H1,H2,H3} investigated the problem with discontinuous initial data and developed a new type of a priori estimates on the material derivative $\dot{u}$.
A major breakthrough in the theory of weak solutions was due to Lions \cite{L2}, who established the global existence of weak solutions provided that the initial energy is finite and the adiabatic exponent $\ga$ is suitably large.
The range of the adiabatic exponent $\ga$ was further relaxed by Feireisl-Novotn\'y-Petzeltov\'a \cite{FNP}.
Danchin \cite{D} proved the global existence and uniqueness of strong solutions close to a stable equilibrium in critical Besov spaces for the scaling of the system.
Recently, for the initial density allowing vacuum and even compact support, Huang-Li-Xin\cite{HLX2} and Li-Xin\cite{LX2} established the global existence and uniqueness of classical solutions to the three-dimensional and two-dimensional Cauchy problems, respectively, under the assumption that the initial energy is sufficiently small.
Duan \cite{DQ} proved the global existence and uniqueness of classical solutions in a three-dimensional half-space with the boundary condition proposed by Navier, provided that the initial energy is sufficiently small.
Subsequently, Cai-Li \cite{CL} extended these results to general three-dimensional bounded domains with the velocity field subject to the slip boundary conditions.

More recently, in the two-dimensional whole space, Danchin-Mucha \cite{DM3} established the global existence of strong solutions in critical Besov spaces, provided the bulk viscosity is sufficiently large and the initial density $\n_0$ is sufficiently close to some positive constant.
For the two-dimensional periodic domain, Danchin-Mucha \cite{DM} obtained the global existence of weak solutions, under the assumptions that the bulk viscosity is sufficiently large and $\nu^{1/2} \|\div u_0\|_{L^2}$ satisfies scale restrictions.
They further proved that as the bulk viscosity tends to infinity, the weak solutions converge to those of the inhomogeneous incompressible Navier-Stokes equations.
Subsequently, Liao-Zodji \cite{LZ2} established the global existence for the two-dimensional Cauchy problem with non-vacuum far-field density, under similar assumptions on the bulk viscosity and $\nu^{1/2} \|\div u_0\|_{L^2}$.
They also obtained the incompressible limit as the bulk viscosity tends to infinity.
Lei-Xiong \cite{LeXi1,LeXi2,LeXi3} proved the global existence and the large-time behavior of weak, strong, and classical solutions, under the sole assumption that the bulk viscosity coefficient is sufficiently large, without any restrictions on the size of the initial data.
Their results cover the two-dimensional periodic domain, general bounded simply connected domains subject to Navier-slip boundary conditions, and the whole space with either vacuum or non-vacuum far-field density.
They also established the incompressible limit as the bulk viscosity tends to infinity.
For the incompressible limit of weak solutions in either the periodic domain or the whole space, the results in \cite{LeXi1,LeXi3} remain valid even without the assumption that the initial velocity field is divergence-free.
Later, Wang-Wu-Zhong \cite{WWZ} considered the problem in the two-dimensional half-space with the slip boundary conditions and established the global existence under the conditions that the bulk viscosity is sufficiently large and $\nu^{1/2} \|\div u_0\|_{L^2}$ satisfies scale restrictions, and obtained the incompressible limit as the bulk viscosity tends to infinity.
Recently, Lei \cite{Lei1} proved the global existence of axisymmetric solutions to the three-dimensional compressible Navier-Stokes equations for arbitrarily large axisymmetric initial data in a cylindrical domain excluding the symmetry axis, under the assumption that the bulk viscosity is sufficiently large.
The incompressible limit as the bulk viscosity tends to infinity was also established therein.
The main aim of this paper is to prove the global existence, large-time behavior, and incompressible limit of weak and strong solutions for the two-dimensional half-space problem subject to the Navier-slip boundary conditions with either vacuum or non-vacuum far-field density, provided that the bulk viscosity is sufficiently large.

Before stating the main results, we first explain the notations and conventions used throughout this paper.
We denote
\be\ba\nonumber
  \int f dx = \int_{\rr_+} fdx.
\ea\ee

For $R>0$, we set
\be\ba\nonumber
B_{R} \triangleq \{x \in \rr \mid |x|<R \}, \quad B^+_{R} \triangleq \{x \in \rr_+ \mid  |x|<R \}.
\ea\ee

For a positive integer $k$ and $1\leq r\leq \infty$, we denote the standard Lebesgue and Sobolev spaces as follows:
\be\ba\nonumber
\begin{cases}
L^r=L^r(\rr_+),\quad D_+^{k,r}=D^{k,r}(\rr_+)=\{v\in L^1_\mathrm{loc}(\rr_+) \mid \nabla ^k v\in L^r(\rr_+)\}, \\
D_+^1=D_+^{1,2},\quad W^{k,r} = W^{k,r}(\rr_+) , \quad H^k = W^{k,2}, \\
\tilde{H}^1 =\{v \in H^1(\rr_+) \mid v \cdot n=0, \curl v = -A v \cdot n^\bot \ \text{ on } \ \partial \rr_+ \}, \\
\tilde{D}_+^1 =\{v \in D_+^1 \mid v \cdot n=0, \curl v = -A v \cdot n^\bot \ \text{ on } \ \partial \rr_+ \}.
\end{cases}
\ea\ee
The material derivative and the transpose gradient are given by
\be\ba\nonumber
\frac{D}{Dt}f=\dot{f} \triangleq f_t + u\cdot\na f, \quad \na^{\bot} \triangleq (-\pa_2,\pa_1).
\ea\ee
The initial total energy is defined as follows:
\be\la{e0}\ba
E_0 \triangleq \int \left( \frac{1}{2} \rho_0 |u_0|^2 + H(\n_0) \right) dx,
\ea\ee
where $H(\n)$ denotes the potential energy density given by
\bnn
H(\n) \triangleq \n \int_{\tilde{\rho}}^{\n} \frac{P(s)-P(\tilde{\rho})}{s^2} ds.
\enn
It is easy to verify that
\be\la{qkjsn}\ba
\begin{cases}
H(\n) = \frac{1}{\ga-1} \n^\ga , & \quad \text{ if } \tilde{\n}=0, \\
\frac{1}{c(\hat{\n},\tilde{\n})} (\n - \tilde{\n})^2 \le H(\n)
\le c(\hat{\n},\tilde{\n}) (\n - \tilde{\n})^2, & \quad \text{ if } \tilde{\n} > 0, \  0 \le \n \le \hat{\n},
\end{cases}
\ea\ee
for some positive constant $c(\hat{\n},\tilde{\n})$.

The effective viscous flux $G$ and the vorticity $\o$ are defined as:
\be\ba\la{gw}
G \triangleq (2\mu + \lam)\div u - ( P-P(\tilde{\n}) ), \quad
\o \triangleq \na^\bot \cdot u = \pa_1 u^2 - \pa_2 u^1.
\ea\ee

Finally, we give the definition of weak and strong solutions to (\ref{ns}).
\begin{definition}
If $(\n,u)$ satisfies \eqref{ns} in the sense of distribution, then we call $(\n,u)$ a weak solution.
Moreover, for a weak solution if
all derivatives involved in \eqref{ns} are regular distributions and equations \eqref{ns} hold almost everywhere in $\rr_+ \times (0,T)$, then $(\n,u)$ is called a strong solution.
\end{definition}

For $\tilde{\n} = 0$, it is obvious that the total mass of sufficiently smooth solutions to (\ref{ns}) is conserved in time, that is, for all $t>0$,
\be\nonumber\ba
\int_{\rr_+} \n dx = \int_{\rr_+} \n_0 dx.
\ea\ee
Without loss of generality, when $\tilde{\n}=0$, we assume that 
\be\la{rho0}\ba
\int_{\rr_+} \n_0 dx =1,
\ea\ee
which implies that there exists a positive constant $N_0$ such that
\be\la{rho00}\ba
\int_{B^+_{N_0}} \n_0 dx \ge \frac{1}{2} \int_{\rr_+} \n_0 dx = \frac{1}{2}.
\ea\ee

The first main result concerns the global existence and large-time behavior of weak solutions.
\begin{theorem}\la{th0}
Assume the initial data $(\n_0,u_0)$ satisfy
\be \la{wsol1}\ba
\n_0 \ge 0, \quad \n_0 |u_0|^2 \in L^1,
\ea\ee
and for some $a>1$,
\be\la{wsol01}\ba
\begin{cases}
{\bar{x}}^a \rho_0 \in L^1, \ \rho_0 \in L^\infty, \ u_0 \in \tilde{D}_+^1,  \quad &\mathrm{ if \  }  \tilde{\n}=0,\\
\n_0-\tilde{\rho} \in L^2 \cap L^\infty, \ u_0 \in \tilde{H}^1, \quad &\mathrm{ if \  }  \tilde{\n}>0,
\end{cases}
\ea\ee
where
\be\la{wsol02}\ba
{\bar{x}}\triangleq (e+|x|^2)^{\frac{1}{2}} \log ^2 (e+|x|^2).
\ea\ee

\noindent\textbf{\textup{(1) Vacuum far-field density ($\tilde{\n}=0$):}}
Let $A = 0$.
There exists a positive constant $\nu_1$ depending only on
$N_0$, $\ga$, $\mu$, $a$, $E_0$, $\|{\bar{x}}^a \n_0\|_{L^1}$, $\|\n_0\|_{L^\infty}$, and $\| \na u_0\|_{L^2}$ such that if $\nu \ge \nu_1$, then the problem \eqref{ns}--\eqref{bjtj2} with $\tilde{\n}=0$ admits at least one global weak solution $(\n,u)$ in $\rr_+ \times (0,\infty)$ satisfying
\be\la{wsol2}\ba
0\le \n(x,t) \leq 2 \left( 1 + \| \n_0 \|_{L^\infty} \right),
\quad \mathrm{for\ any\ }(x,t)\in \rr_+ \times[0,\infty),
\ea\ee
and for any $1 \le p <\infty$,
\be\la{wsol3}\ba
\begin{cases}
\rho\in L^\infty(\rr_+ \times (0,T)) \cap C([0,T];L^p), {\bar{x}}^a \rho \in L^\infty(0,T;L^1), \\
\na u\in L^\infty(0,T;L^2), \sqrt{t}u_t \in L^2(0,T;L^2), \sqrt{t} \na u \in L^\infty(0,T;L^p).
\end{cases}
\ea\ee
Moreover, there exists a positive constant $N_1$ depending on
$N_0$, $\| {\bar{x}}^a \rho_0 \|_{L^1}$, and $E_0$ such that
\be\la{wsol30}\ba
\inf_{0 \le t \le T} \int_{ B^+_{N_1 (1+t)} }\rho (x,t) dx \ge \frac{1}{4}.
\ea\ee
Furthermore, for $t \ge 1$, $(\n,u)$ satisfies the following decay rates:
\be\la{wsol300}\ba
\begin{cases}
\| \nabla u(\cdot,t) \|_{L^{p}} \leq C(p) t^{-1+\frac{1}{p}}, \ & \textnormal{ for } p \in [2,\infty), \\
\| P(\cdot,t) \|_{L^{r}} \leq C(r) (\nu^{-1} t)^{-1+\frac{1}{r}}, \ & \textnormal{ for } r \in (1,\infty), \\
\| \sqrt{\n} \dot{u}(\cdot,t) \|_{L^2} \le C t^{-1},
\end{cases}
\ea\ee
where $C(z)$ depends on $z$ and $N_0$, $\ga$, $\mu$, $E_0$, $\|{\bar{x}}^a \n_0\|_{L^1}$, and $\|\n_0\|_{L^\infty}$.

\noindent\textbf{\textup{(2) Non-vacuum far-field density ($\tilde{\n}>0$):}}
There exists a positive constant $\nu_2$ depending only on
$\ga$, $\mu$, $A$, $E_0$, $\tilde{\n}$, $\|\n_0\|_{L^\infty}$, and $\| \na u_0\|_{L^2}$ such that if $\nu \ge \nu_2$, then the problem \eqref{ns}--\eqref{bjtj2} with $\tilde{\n}>0$ admits at least one global weak solution $(\n,u)$ in $\rr_+ \times (0,\infty)$ satisfying
\be\la{2wsol2}\ba
0\le \n(x,t) \leq 2 \left( 1 + \tilde{\n} + \| \n_0 \|_{L^\infty} \right),
\quad \mathrm{for\ any\ }(x,t)\in \rr_+ \times[0,\infty),
\ea\ee
and for any $0<T<\infty$, $1 \le p < \infty $, and $2 \le s <\infty$,
\be\la{wsol4}\ba
\begin{cases}
\rho-\tilde{\n} \in L^\infty(\rr_+ \times (0,T)) \cap C([0,T];L^s), \\ 
u\in L^\infty(0,T;H^1), \sqrt{t} u_t \in L^2(0,T;L^2), \sqrt{t} \na u \in L^\infty(0,T;L^p).
\end{cases}
\ea\ee
Furthermore, the following large-time behavior holds:
\be\la{cpwsol4}\ba
\lim_{t \to \infty} \left( \| \n(\cdot,t) - \tilde{\n} \|_{L^s} + \| \na u(\cdot,t) \|_{L^r} \right) = 0,
\ea\ee
for any $s \in (2,\infty)$ and $r \in [2,\infty)$.
\end{theorem}

Based on the $\nu$-uniform a priori estimates for the compressible Navier-Stokes equations (\ref{ns}) established in the proof of Theorem \ref{th0}, we can show that, as $\nu \to \infty$, the weak solutions obtained in Theorem \ref{th0} converge to a global weak solution of the following inhomogeneous incompressible Navier-Stokes equations:
\be\la{isol2}\ba
\begin{cases}
\n_t+\div(\n u)=0,\\
(\n u)_t+\div(\n u\otimes u) -\mu \Delta u + \na \pi =0, \\
\div u=0, \\
u \cdot n = 0, \quad \curl u=-A u \cdot n^\bot \quad \textnormal{ on } \partial \rr_+, \\
\n|_{t=0}=\n_0, \quad \n u|_{t=0}=\n_0 u_0.
\end{cases}
\ea\ee

For the case of vacuum far-field density, we can establish the following singular limit result.

\begin{theorem}\la{th01}
Let $\tilde{\n}=0$ and $A=0$.
Assume that the initial data $(\n_0,u_0)$ satisfy \eqref{wsol1} and $\eqref{wsol01}_1$.
Fix $\mu>0$ and the initial data $(\n_0,u_0)$, and let $\nu_1$ be the constant given in Theorem \ref{th0}.
For $\nu \ge \nu_1$, we denote by $(\n^{\nu},u^{\nu})$ the weak solution of \eqref{ns}--\eqref{bjtj2} established in Theorem \ref{th0}.
Then as $\nu \to \infty$, there exists a subsequence of $(\n^{\nu},u^{\nu})$ that converges to a solution $(\n,u)$ of \eqref{isol2}.
Furthermore, $(\n,u)$ satisfies for any $0<T<\infty$, $0<R<\infty$, $2<r<\infty$, and $1\le p <\infty$,
\be\la{isol3}\ba
\begin{cases}
\rho\in L^{\infty}(\rr_+ \times (0,T)) \cap C([0,T];L^p),\ {\bar{x}}^a \rho \in L^\infty(0,T;L^1), \\
u \in L^2(0,T;L^2(B^+_R)), \\
\na u,\ \sqrt{t} \sqrt{\n} \du,\ \sqrt{t} \na \pi,\ \sqrt{t} \na^2 u,\ t \na \du \in L^2(\rr_+ \times (0,T)), \\
\sqrt{t} \na u,\ t \sqrt{\n} \du,\ t \na \pi,\ t \na^2 u \in L^\infty(0,T;L^2), \\
\sqrt{t} \pi \in L^2(0,T;L^r), \ t \pi \in L^\infty(0,T;L^r).
\end{cases} 
\ea\ee
In addition, for any $0 < \tau <\infty$, we have
\be\la{isol4}\ba
\div u^{\nu} = O(\nu^{-1/2}) \  \textnormal{in} \  L^2(\rr_+ \times (0,\infty)) \cap L^\infty(\tau,\infty;L^2).
\ea\ee
If the initial data $(\n_0,u_0)$ further satisfy
\be\la{ws}\ba
\div u_0=0,
\ea\ee
then the limit $(\n,u)$ satisfies
\be\la{lws1}\ba
\begin{cases}
\rho\in L^{\infty}(\rr_+ \times (0,T)) \cap C([0,T];L^p), {\bar{x}}^a \rho \in L^\infty(0,T;L^1), \\
u\in L^\infty(B^+_R \times (0,T)),\ \sqrt{\n} \dot{u},\ \na^2 u,\ \na \pi,\ \sqrt{t} \na \dot{u} \in L^2(\rr_+ \times (0,T)), \\
\sqrt{\n} \dot{u},\ \sqrt{t} \na \pi,\ \sqrt{t} \na^2 u \in L^\infty(0,T;L^2) \cap L^2(0,T;L^s), \\
\pi, \ \sqrt{t} \pi \in L^\infty(0,T;L^r),
\end{cases} 
\ea\ee
for any $0<T<\infty$, $2<r<\infty$, $2\le s <\infty$, and $0<R<\infty$.

Moreover, it holds that
\be\la{lws2}\ba
\div u^{\nu} = O(\nu^{-1/2}) \  \textnormal{in} \  L^2(\rr_+ \times (0,\infty)) \cap L^\infty(0,\infty;L^2).
\ea\ee
\end{theorem}

The following result concerns the singular limit for the case of non-vacuum far-field density.

\begin{theorem}\la{th001}
Let $\tilde{\n}>0$, and suppose the initial data $(\n_0,u_0)$ satisfy \eqref{wsol1} and $\eqref{wsol01}_2$.
Fix $\mu>0$ and the initial data $(\n_0,u_0)$.
For $\nu_2$ given in Theorem \ref{th0}, when $\nu \ge \nu_2$, we denote by $(\n^{\nu},u^{\nu})$ the weak solution to \eqref{ns}--\eqref{bjtj2} given by Theorem \ref{th0}.
Then as $\nu \to \infty$, the solution sequence $(\n^{\nu},u^{\nu})$ admits a subsequence that converges to a solution $(\n, u)$ of \eqref{isol2}.
The limit $(\n,u)$ satisfies for any $0<T<\infty$ and $2\le s <\infty$,
\be\la{lisol3}\ba
\begin{cases}
\rho -\tilde{\n} \in L^{\infty}(\rr_+ \times (0,T)) \cap C([0,T];L^s), \\ 
u\in L^2(0,T;H^1), \\
\sqrt{t} \na^2 u,\ \sqrt{t} \na \pi,\ t \sqrt{\n} u_t,\ t^2 \na u_t \in L^2(\rr_+ \times (0,T)), \\
\sqrt{t} \na u,\ t \na \pi,\ t \na^2 u,\ t^2 \sqrt{\n} u_t \in L^\infty(0,T;L^2).
\end{cases} 
\ea\ee
Moreover, for any $0 < \tau <\infty$, we have
\be\la{lisol4}\ba
\div u^{\nu} = O(\nu^{-1/2}) \  \textnormal{in} \  L^2(\rr_+ \times (0,\infty)) \cap L^\infty(\tau,\infty;L^2).
\ea\ee
If the initial data $(\n_0,u_0)$ additionally satisfy
\be\la{0lws}\ba
\div u_0=0,
\ea\ee
then the limit $(\n,u)$ satisfies for any $0<T<\infty$ and $2 \le p <\infty$,
\be\la{0lws1}\ba
\begin{cases}
\rho -\tilde{\n} \in L^{\infty}(\rr_+ \times (0,T)) \cap C([0,T];L^p), \\
u\in L^\infty(0,T;H^1), \quad \sqrt{\n} u \in C([0,T];L^2), \\
\sqrt{\n} u_t,\ \na^2 u,\ \na \pi,\ \sqrt{t} \na u_t \in L^2(\rr_+ \times (0,T)), \\
\sqrt{\n} u_t,\ \sqrt{t} \na \pi,\ \sqrt{t} \na^2 u \in L^\infty(0,T;L^2) \cap L^2(0,T;L^p).
\end{cases} 
\ea\ee
Furthermore, it holds that
\be\la{0lws2}\ba
\div u^{\nu} = O(\nu^{-1/2}) \  \textnormal{in} \  L^2(\rr_+ \times (0,\infty)) \cap L^\infty(0,\infty;L^2).
\ea\ee
\end{theorem}

If the initial data $(\n_0,u_0)$ satisfy higher regularity conditions, we can obtain the global existence of strong solutions.

\begin{theorem}\la{th1}
In addition to the assumption on the initial data $(\n_0,u_0)$ in Theorem \ref{th0}, we assume further that for some $q>2$,
\be\la{ssol1}\ba
\begin{cases}
{\bar{x}}^a \rho_0 \in H^1 \cap W^{1,q},     \quad &\mathrm{ if \  }  \tilde{\n}=0,\\
\n_0-\tilde{\rho} \in H^1 \cap W^{1,q},      \quad &\mathrm{ if \  }  \tilde{\n}>0.
\end{cases}
\ea\ee

\noindent\textbf{\textup{(1) Vacuum far-field density ($\tilde{\n}=0$):}}
Let $A = 0$. For $\nu_1$ determined in Theorem \ref{th0}, if $\nu \ge \nu_1$, then the problem \eqref{ns}--\eqref{bjtj2} with $\tilde{\n}=0$ admits a unique global strong solution $(\n,u)$ in $\rr_+ \times (0,\infty)$
satisfying \eqref{wsol2}, \eqref{wsol3}, \eqref{wsol30}, \eqref{wsol300}, and
\be\la{ssol4}\ba
\begin{cases}	
\rho \in C([0,T];L^1 \cap H^1\cap W^{1,q} ),\\ 
{\bar{x}}^a\rho \in L^\infty ( 0,T ;L^1\cap H^1\cap W^{1,q} ),\\ 
\sqrt{\rho } u,\,\nabla u,\, {\bar{x}}^{-1}u, \, \sqrt{t} \sqrt{\rho } u_t \in L^\infty (0,T; L^2 ) , \\ 
\nabla u\in L^2(0,T;H^1)\cap L^{(q+1)/q}(0,T; W^{1,q}), \\ 
\sqrt{t}\nabla u\in L^2(0,T; W^{1,q} ) , \\ 
\sqrt{\rho } u_t, \, \sqrt{t}\nabla u_t ,\, \sqrt{t} {\bar{x}}^{-1}u_t\in L^2({\mathbb R}_+^2 \times (0,T)).
\end{cases} 
\ea\ee

\noindent\textbf{\textup{(2) Non-vacuum far-field density ($\tilde{\n}>0$):}}
For $\nu_2$ as in Theorem \ref{th0}, if $\nu \ge \nu_2$, then the problem \eqref{ns}--\eqref{bjtj2} with $\tilde{\n}>0$ admits a unique global strong solution $(\n,u)$ in $\rr_+ \times (0,\infty)$
satisfying \eqref{2wsol2}, \eqref{wsol4}, \eqref{cpwsol4}, and
\be\la{ssol5}\ba
\begin{cases}	
\rho-\tilde{\n}\in C([0,T];W^{1,q} ), \quad \n_t\in L^\infty(0,T;L^2), \\ 
u\in L^\infty(0,T; H^1) \cap L^{(q+1)/q}(0,T; W^{2,q}), \\ 
\sqrt{t} u \in L^2(0,T; W^{2,q}) \cap L^\infty(0,T;H^2), \\
\sqrt{t} u_t \in L^2(0,T;H^1), \\
\n u\in C([0,T];L^2), \quad \sqrt{\n} u_t\in L^2(\rr_+ \times(0,T)),
\end{cases} 
\ea\ee
for any $0<T<\infty$.
\end{theorem}

Finally, following the approach in \cite{LX}, we can obtain the following large-time behavior of the spatial gradient of the density for the strong solution in Theorem $\ref{th1}$ when vacuum occurs initially.

\begin{theorem}\la{th3}
In addition to the assumptions in Theorem \ref{th1}, 
assume further that there exists some point $x_0 \in \rr_+$ such that $\n_0(x_0)=0$.
Then the unique global strong solution $(\n,u)$ obtained in Theorem \ref{th1} satisfies, for any $r>2$,
\be\la{pbu02}\ba
\lim_{t \to \infty} \| \na \n(\cdot ,t) \|_{L^r} = \infty.
\ea\ee
\end{theorem}

A few remarks are in order.

\begin{remark}\la{lrk1}
If the initial data $(\n_0,u_0)$ further satisfy higher regularity
and the compatibility condition:
\be\la{csol2}\ba
- \mu \Delta u_0 - (\mu+\lm) \nabla \div u_0 + \nabla P(\n_0) = \n_0^{1/2} g,
\ea\ee
for some $g \in L^2$, then the strong solution obtained in Theorem \ref{th1} becomes a classical solution for positive time.
The detailed proofs follow from arguments analogous to those in \cite{LZZ,HL,HLX2,LeXi3}.
\end{remark}

\begin{remark}\la{lrk2}
It is noted that by adapting the arguments in \cite[Theorem 2.1]{L1}, one can show that the inhomogeneous incompressible Navier-Stokes equations \eqref{isol2} admit global weak solutions for initial data with $\div u_0 \neq 0$.
Moreover, Theorems \ref{th01} and \ref{th001} establish the convergence of the compressible Navier-Stokes equations to \eqref{isol2} without requiring the restrictive condition ${\rm div} u_0 = 0$.
In particular, this convergence result implies the global existence of weak solutions to \eqref{isol2} when $\div u_0 \neq 0$, which coincides with \cite[Theorem 2.1]{L1}.
\end{remark}

\begin{remark}\la{lrk3}
In \cite{WWZ}, the authors established the global existence of weak solutions to \eqref{ns}--\eqref{bjtj2} with non-vacuum far-field density, under the assumptions that the bulk viscosity coefficient is sufficiently large and $\| \div u_0 \|^2_{L^2} \le M \nu^{-1}$ for some $M>0$.
Therefore, our results generalize and improve the previous results in \cite{WWZ}.
\end{remark}

\begin{remark}\la{lrk4}
Note that in the case $\tilde{\n}=0$, we require $A=0$.
The main reason is that we need to derive a time-uniform space-time $L^2$ estimate for the pressure.
Due to the lack of the integrability of the velocity, such an estimate cannot be obtained when $A \neq 0$.
For more details, we refer to Lemma \ref{bkjl2}.
\end{remark}

\begin{remark}\la{lrk5}
Using Lemma \ref{hm} and adapting arguments in \cite{LSZ}, we can obtain the global existence and uniqueness of the strong solution to \eqref{isol2} for initial data satisfying \eqref{wsol1}, \eqref{wsol01}, \eqref{ws}, and \eqref{ssol1}.
Furthermore, for such initial data, via a weak-strong uniqueness argument (see \cite{Lei1} for details), we can show that the weak solutions obtained in Theorem \ref{th0} converge to the unique global strong solution of \eqref{isol2}.
\end{remark}

We now make some comments on the analysis of this paper.
For initial data satisfying \eqref{wsol1}, \eqref{wsol01}, and \eqref{ssol1}, the local existence and uniqueness of strong solutions to (\ref{ns})--(\ref{bjtj2}) can be established by arguments similar to those in \cite{LLL,LZ}.
To extend these solutions globally in time, we need to derive global a priori estimates in suitable higher-order norms, where the key issue is to obtain a time-uniform upper bound for the density.

We note that in previous work \cite{LeXi3} on the Cauchy problem, the analysis relies heavily on the compensated compactness result of \cite{CLMS}.
Exploiting the symmetry of the half-space, we establish, via an extension argument, a similar compensated compactness result under suitable boundary condition in the half-space (see Lemma \ref{hm}), which plays a central role in our analysis.
For the case $\tilde{\n}=0$, the main difficulties arise from the appearance of vacuum at the far field and the lack of integrability of the velocity and its material derivatives.
To overcome these difficulties, motivated by \cite{LeXi3,LX2}, we need to derive a space-time $L^2$ estimate for the pressure that is independent of time.
This estimate is essential for deriving the large-time decay of the effective viscous flux.
To this end, we observe that the boundary conditions (\ref{bjtj1}) imply that the effective viscous flux $G$ satisfies the elliptic equation (\ref{bkj22}).
Applying the standard elliptic estimates together with even extension arguments, we obtain estimates for the effective viscous flux and its derivatives (see (\ref{bkj23}) and (\ref{bkj28})), which further yield the desired pressure estimate.
In addition, using the boundary conditions (\ref{bjtj1}) and the special geometric structure of the planar boundary, we have
\be\la{bkjbjds}\ba
\dot{u} \cdot n =0, \quad  u \cdot \na u \cdot n = 0 \  \text{ on } \p \rr_+.
\ea\ee
Combining this with Lemma \ref{hm} and the pressure estimate, we derive several key decay estimates independent of the bulk viscosity through a series of refined computations.
Using these estimates and Zlotnik's inequality (see Lemma \ref{zli}), we obtain the desired time-uniform upper bound for the density.

Compared with the case $\tilde{\n}=0$, the main difficulty in the case $\tilde{\n}>0$ is that the density is no longer $L^1$-integrable, which prevents us from deriving a time-uniform $L^2$ estimate for the pressure term $P-P(\tilde{\n})$.
To overcome this, we first employ the standard energy estimate together with arguments similar to those used in the case $\tilde{\n}=0$ to establish uniform estimates for the velocity gradient on the short time interval $(0,\min\{1,T\})$.
By introducing time-weighted estimates, we further obtain the $\nu$-uniform estimates.
Building upon the framework developed in \cite{LeXi3,LZ2}, and using Lemma \ref{hm} together with the $\nu$-uniform estimates, we derive global time-uniform estimates for the velocity gradient through a series of careful calculations.
Combining these estimates with Zlotnik's inequality, we derive the desired time-uniform upper bound for the density.
After establishing the upper bound for the density, we adapt the arguments in \cite{HL,LLL,HLX3,HLX2} to derive higher-order estimates for the solution, which allow the local solution to be extended globally in time.
Moreover, by using the $\nu$-uniform estimates and adapting the methods in \cite{LeXi1,LeXi2,LeXi3}, we prove that, as the bulk viscosity coefficient tends to infinity, solutions of the compressible Navier-Stokes equations converge to those of the inhomogeneous incompressible Navier-Stokes equations.

The rest of this paper is organized as follows:
Section 2 recalls some elementary inequalities and known results.
Sections 3 and 4 are devoted to establishing the necessary a priori estimates for the cases of vacuum and non-vacuum far-field density, respectively.
Finally, Section 5 presents the proofs of our main results, Theorems \ref{th0}--\ref{th3}.

\section{Preliminaries}
In this section, we will introduce some known facts and elementary inequalities which will be used frequently later.

First, we have the following local existence theory of the strong solution, and its proof can be found in \cite{LLL,LZ}.
\begin{lemma}\la{lct}
Assume $\left(\n_0,u_0\right)$ satisfy \eqref{wsol1}, \eqref{wsol01}, and \eqref{ssol1}.
Then there is a small time $T>0$ such that there exists a unique strong solution $(\n,u)$ to the problem 
\eqref{ns}--\eqref{bjtj2} in $\rr_+ \times (0,T]$ and
when $\tilde{\n}=0$, $(\n,u)$ satisfies \eqref{wsol3} and \eqref{ssol4};
when $\tilde{\n}>0$, $(\n,u)$ satisfies \eqref{wsol4} and \eqref{ssol5}.
\end{lemma}
Next, the following Gagliardo-Nirenberg inequalities (see \cite{NI}) will be used frequently later.
\begin{lemma}\la{gn1}
For $2<p<\infty$, there exists a generic positive constant $C$ such that for any $u\in H^1(\rr_+)$,
\be\ba\la{gn11}
\| u \|_{L^p} \le Cp^{1/2}\| u \|^{2/p}_{L^2} \| \na u\|^{1-2/p}_{L^2}.
\ea\ee
Furthermore, for $1\le r <\infty$, $2<q<\infty$, 
there exists a positive constant $C$ depending only on $r,\ q $,
such that for every function $v\in L^r(\rr_+) \cap D^{1,q}(\rr_+)$ it holds that
\be\ba\la{gn12}
\|v \|_{L^\infty} \le C\|v \|^{r(q-2)/(2q+r(q-2))}_{L^r} \|\na v\|^{2q/(2q+r(q-2))}_{L^q}.
\ea\ee
\end{lemma}

The following Poincar\'e type inequality can be found in \cite{F}.
\begin{lemma}\la{pt}
Let $\OM \subset \rr$ be a bounded Lipschitz domain.
Assume that $v\in H^1(\OM)$ and that $\n$ is a non-negative function satisfying
\be\ba\nonumber
0<M_1\leq \int_{\om} \n dx,\quad \int_{\om} \n^r dx \leq M_2,
\ea\ee
with $r>1$. Then there exists a positive constant $C$ depending only on
$M_1$, $M_2$, $r$, and $\OM$ such that
\be\ba\la{pt1}
\|v\|_{L^2(\om)}^2 \leq C\int_{\om} \n |v|^2 dx + C \|\na v\|_{L^2(\om)}^2.
\ea\ee
\end{lemma}

The following weighted $L^p$ estimates can be found in \cite[Theorem B.1]{L1}.
\begin{lemma}\la{WPE}
For $m \in [2,\infty)$ and $\theta \in (1+\frac{m}{2},\infty)$, there exists a generic positive constant $C$ such that for any $v \in D^1(\rr)$,
\be\la{WPE1}\ba
\left( \int _{{\mathbb {R}^2 }} \frac{|v|^m}{e+|x|^2}(\log (e+|x|^2))^{-\theta }dx \right) ^{1/m}
\le C\Vert v\Vert _{L^2(B_1)}+C\Vert \nabla v\Vert _{L^2({\mathbb {R}^2 }) }.
\ea\ee
\end{lemma}

The following estimate plays a crucial role in the analysis for the case $\tilde{\rho}=0$, and its proof can be found in \cite[Lemma 2.4]{LX2}.
\begin{lemma}\la{esnr}
For $N_*\ge 1$ and positive constants $M_3$, $M_4$, $\beta$, we suppose that $\n$ satisfies
\be\la{esnr1}\ba
0\le \rho \le M_3, \quad M_4\le \int _{B_{N_*}}\rho dx ,\quad{\bar{x}}^\beta \rho \in L^1(\rr).
\ea\ee

Then for any $r \in [2,\infty)$, there exists a positive constant $C$ depending only on $M_3$, $M_4$, $\beta$ and $r$ such that for every 
$v\in \left\{ v\in D^1 ({\rr}) \mid \rho^{1/2}v\in L^2(\rr) \right\}$,
\be\la{esnr2}\ba
\left( \int _{\mathbb {R}^2}\rho |v |^r dx\right) ^{1/r} 
\le C N_*^3 (1+\Vert {\bar{x}}^\beta \rho \Vert _{L^1(\rr)}) 
\left( \Vert \sqrt{\n} v\Vert _{L^2(\mathbb {R}^2)} 
+ \Vert \nabla v \Vert _{L^2(\mathbb {R}^2)}\right).
\ea\ee
\end{lemma}

We will frequently use the following div-curl estimates (see \cite{AJ,MD,WWV}).
\begin{lemma}\la{dc}
Let $k \ge 0$ be an integer, $1<p<\infty$, and let $\OM=\rr_+$.
Then there exists a positive constant $C$ depending only on $k$ and $p$ such that for every $u\in D_+^{k,p}(\rr_+)$ with $u \cdot n=0$ on $\p \rr_+$, the following estimate holds:
\be\ba\la{dc1}
\| \na u \|_{W^{k,p}} \le C\left( \|\div u\|_{W^{k,p}} +\| \curl u \|_{W^{k,p}} \right).
\ea\ee
\end{lemma}

To estimate $\| \na u\|_{L^{\infty}}$ and $\| \na \n\|_{L^{q}}$ 
we require the following Beale-Kato-Majda type inequality, 
which was established in \cite{K} when $\div u \equiv 0$. 
For further reference, we direct readers to \cite{BKM,CL,HLX1}.
\begin{lemma}\la{bkm}
For $2<q<\infty$, 
there exists a positive constant $C$ 
depending only on $q $ such that for every function $u \in \tilde{D}^1_+(\rr_+)$ and $\na u \in D_+^{1,q}(\rr_+)$, it holds that
\be\ba\la{bkm1}
 \|\na u\|_{L^\infty} \le C \left( \|\div u\|_{L^\infty}+ \|\o\|_{L^\infty} \right)\log \left(e+ \|\na^2 u\|_{L^q} \right)+ C\|\na u\|_{L^2}+C.
\ea\ee
\end{lemma}

Let $\mathcal{H}^1(\rr)$ and $\mathcal{BMO}(\rr)$ denote the standard Hardy and BMO spaces (see \cite[Chapter IV]{SEM}).
The following well-known results can be found in \cite{CLMS}.
\begin{lemma}\la{HM}
(i) There exists a generic positive constant $C$ such that
\be\ba\nonumber
\| E \cdot B \|_{\mathcal{H}^1(\rr)} \le C \| E \|_{L^2(\rr)} \| B \|_{L^2(\rr)},
\ea\ee
for all $E \in L^2(\rr)$ and $B \in L^2(\rr)$ satisfying
\be\ba\nonumber
\na^{\bot} \cdot E =0, \quad \div B=0 \ \text{ in } D'(\rr).
\ea\ee
(ii) There exists a generic positive constant $C$ such that for all $v \in D^1(\rr)$,
\be\ba\nonumber
\| v \|_{\mathcal{BMO}(\rr)} \le C \| \na v \|_{L^2(\rr)}.
\ea\ee
\end{lemma}

By Lemma \ref{HM}, we can obtain the following result, which plays an important role in our analysis.
\begin{lemma}\la{hm}
Let $f=(f^1,f^2),\ g=(g^1,g^2) \in (H^1(\rr_+))^2$ satisfying the boundary condition $g\cdot n = 0$ on $\partial \rr_+$, and let $h \in H^1(\rr_+)$.
There exists a generic positive constant $C$ such that
\be\ba\la{hm01}
\int_{\rr_+} h (\na f^1 \cdot \na^\bot g^2) dx
\le C \| \na h \|_{L^2(\rr_+)} \| \na f \|_{L^2(\rr_+)} \| \na g \|_{L^2(\rr_+)},
\ea\ee
and
\be\ba\la{hm02}
\int_{\rr_+} h (\na g^2 \cdot \na^\bot f^1) dx
\le C \| \na h \|_{L^2(\rr_+)} \| \na f \|_{L^2(\rr_+)} \| \na g \|_{L^2(\rr_+)}.
\ea\ee
\end{lemma}
\begin{proof}
We first extend $f^1$ and $h$ evenly to $\rr$ by
\be\nonumber\ba
\widetilde{f^1}(x_1,x_2) \triangleq
\begin{cases}
f^1(x_1,x_2), \ & x_2 \ge 0, \\
f^1(x_1,-x_2), \ & x_2 < 0,
\end{cases} \quad
\widetilde{h}(x_1,x_2) \triangleq
\begin{cases}
h(x_1,x_2), \ & x_2 \ge 0, \\
h(x_1,-x_2), \ & x_2 < 0,
\end{cases}
\ea\ee
which implies that $\widetilde{f^1}, \widetilde{h} \in H^1(\rr)$ and satisfy
\be\la{hm1}\ba
\| \na \widetilde{f^1} \|^2_{L^2(\rr)} = 2 \| \na f^1 \|^2_{L^2(\rr_+)}, \quad
\| \na \widetilde{h} \|^2_{L^2(\rr)} = 2 \| \na h \|^2_{L^2(\rr_+)}.
\ea\ee
Moreover, we define the odd extension of $g^2$ to $\rr$ by
\be\nonumber\ba
\widetilde{g^2}(x_1,x_2) \triangleq
\begin{cases}
g^2(x_1,x_2), \ & x_2 \ge 0, \\
- g^2(x_1,-x_2), \ & x_2 < 0.
\end{cases}
\ea\ee
The boundary condition $g \cdot n = 0$ on $\partial \rr_+$ implies that $g^2 = 0$ on $\partial \rr_+$, hence $\widetilde{g^2} \in H^1(\rr)$ and satisfies
\be\la{hm2}\ba
\| \na^\bot \widetilde{g^2} \|^2_{L^2(\rr)} = 2 \| \na^\bot g^2 \|^2_{L^2(\rr_+)}.
\ea\ee
Note that $\widetilde{h} (\na \widetilde{f^1} \cdot \na^\bot \widetilde{g^2})$ is an even function with respect to $x_2$.
Using the fact that $\mathcal{BMO}(\rr)$ is the dual space of $\mathcal{H}^1(\rr)$ (see \cite{FC}), together with (\ref{hm1}), (\ref{hm2}), and Lemma \ref{HM}, we derive
\be\ba\la{hm3}
\int_{\rr_+} h (\na f^1 \cdot \na^\bot g^2) dx
& = \frac{1}{2} \int_{\rr} \widetilde{h} (\na \widetilde{f^1} \cdot \na^\bot \widetilde{g^2}) dx \\
& \le C \| \widetilde{h} \|_{\mathcal{BMO}(\rr)} \|\na \widetilde{f^1} \cdot \na^\bot \widetilde{g^2} \|_{\mathcal{H}^1(\rr)} \\
& \le C \| \na \widetilde{h} \|_{L^2(\rr)} \| \na \widetilde{f^1} \|_{L^2(\rr)} \| \na^\bot \widetilde{g^2} \|_{L^2(\rr)} \\
& \le C \| \na h \|_{L^2(\rr_+)} \| \na f \|_{L^2(\rr_+)} \| \na g \|_{L^2(\rr_+)},
\ea\ee
which yields (\ref{hm01}).

Finally, since
\be\nonumber\ba
h (\na g^2 \cdot \na^\bot f^1) = - h (\na f^1 \cdot \na^\bot g^2),
\ea\ee
the estimate (\ref{hm02}) follows directly from (\ref{hm01}).
The proof of Lemma \ref{hm} is completed.
\end{proof}

The following Zlotnik inequality (see \cite{ZAA}) will be used to derive the time-uniform upper bound for the density.
\begin{lemma}\la{zli}
Let the function $y(t) \in W^{1,1}(0,T)$ satisfy
\bnn
y'(t)= g(y)+h'(t) \textnormal{ on } [0,T], \quad y(0)=y_0, 
\enn
with $ g\in C(\mathbb{R})$ and $h \in W^{1,1}(0,T).$ If $g(\infty)=-\infty$
and 
\bnn 
h(t_2)-h(t_1) \le N_0 +N_1(t_2-t_1),
\enn
for all $0 \le t_1<t_2\le T$
with some $N_0\ge 0$ and $N_1\ge 0$, then
\bnn
y(t)\le \max\left\{y_0,\overline{\zeta} \right\}+N_0<\infty
\textnormal{ on } [0,T],
\enn
where $\overline{\zeta}$ is a constant such
that 
\bnn
g(\zeta)\le -N_1 \quad \textnormal{ for }\quad \zeta\ge \overline{\zeta}.
\enn
\end{lemma}

\section{A Priori Estimates for Vacuum Far-Field Density}

In this section, we establish some necessary a priori estimates for the case of $\tilde{\n}=0$.
Assume that the initial data $(\n_0,u_0)$ satisfy (\ref{wsol1}) and $(\ref{wsol01})_1$.
Let $(\n,u)$ be a strong solution to (\ref{ns})--(\ref{bjtj2}) on $\rr_+ \times (0,T]$ provided by Lemma \ref{lct}.

\subsection{A Priori Estimates (I): Lower Order Estimates}

The main aim of this subsection is to derive the following a priori estimates.
\begin{proposition}\la{pro1}
There exists a positive constant $\nu_1$ depending only on
$N_0$, $\ga$, $\mu$, $a$, $E_0$, $\|{\bar{x}}^a \n_0\|_{L^1}$, $\|\n_0\|_{L^\infty}$, and $\| \na u_0\|_{L^2}$
such that if $(\n,u)$ satisfies
\be\ba\la{pro101}
\sup_{0\leq t\leq T} \|\n\|_{L^\infty} \leq 2 \left( 1 + \| \n_0 \|_{L^\infty} \right),
\ea\ee
then
\be\ba\la{pro102}
\sup_{0\leq t\leq T} \|\n\|_{L^\infty} \leq \frac{3}{2} \left( 1 + \| \n_0 \|_{L^\infty} \right),
\ea\ee
provided $\nu \geq \nu_1$.
\end{proposition}

The proof of Proposition \ref{pro1} will be postponed to the end of this subsection.

We first state the standard energy estimate.
\begin{lemma}\la{bkjl1}
There exists a positive constant $C$ depending only on
$\ga$, $\mu$, and $E_0$ such that
\be\ba\la{bkj01}
\sup_{0\leq t\leq T} \int \left( \rho |u|^2 + \n^\ga \right) dx
+ \int_0^T \int \left( |\na u|^2 + \nu (\div u)^2 \right) dx dt \leq C.
\ea\ee
\end{lemma}
\begin{proof}
Multiplying $(\ref{ns})_2$ by $u$ and integrating the resulting equation over $\rr_+$, we obtain (\ref{bkj01}) after using $(\ref{ns})_1$ and (\ref{dc1}).
\end{proof}

\begin{lemma}\la{bkjl2}
Let $(\n,u)$ be a strong solution to \eqref{ns}--\eqref{bjtj2} satisfying \eqref{pro101}.
Then there exists a positive constant $C$ depending only on
$\ga$, $\mu$, $E_0$, and $\| \n_0 \|_{L^\infty}$ such that
\be\la{bkj021}\ba
\sup_{0\le t\le T} \left( \| \na u \|^2_{L^2} + \nu \| \div u \|^2_{L^2} \right) + \int_0^T \| \sqrt{\n} \dot{u}\|^2_{L^2} dt
\le C \left( 1 + \| \na u_0 \|^2_{L^2} + \nu \| \div u_0 \|^2_{L^2} \right),
\ea\ee
and
\be\la{bkj022}\ba
\sup_{0\le t\le T} \si \left( \| \na u \|^2_{L^2} + \nu \| \div u \|^2_{L^2} \right)
+ \int_0^T \left( \si \| \sqrt{\n} \dot{u}\|^2_{L^2} + \frac{1}{\nu} \| P \|^2_{L^2} \right) dt
\le C,
\ea\ee
with
\be\nonumber
\si(t) \triangleq \min\{1,t\}.
\ee
\end{lemma}

\begin{proof}
First, we rewrite $(\ref{ns})_2$ as 
\be\la{bkj21}\ba
\n\dot{u} = \na G + \mu\na^{\bot}\o.
\ea\ee
This, combined with the boundary conditions (\ref{bjtj1}), implies that $G$ satisfies
\be\la{bkj22}\ba
\begin{cases}
\Delta G = \div \left( \rho \dot{u} \right) & \text{ in } \rr_+, \\
\frac {\p G}{\p n}= \rho \dot{u} \cdot n & \text{ on } \p \rr_+, \\
G \to 0 & \text{ as } |x| \to \infty.
\end{cases}
\ea\ee
The standard $L^p$ estimate of elliptic equations (see \cite{JK}) together with (\ref{bkj21}) yields for any $2\le p<\infty$,
\be\la{bkj23}\ba
\| \na G \|_{L^p} + \| \na \o \|_{L^p} \le C \| \n \dot{u}\|_{L^p}.
\ea\ee
In particular, from (\ref{pro101}) and (\ref{bkj23}), we deduce that
\be\la{bkj24}\ba
\| \na G \|_{L^2} + \| \na \o \|_{L^2}
& \le C \| \rho \dot{u} \|_{L^2} \le C \| \sqrt{\rho} \dot{u} \|_{L^2}.
\ea\ee
Next, we estimate $G$ by extending the problem (\ref{bkj22}) to the whole space.
We set $v \triangleq \rho \dot{u}$ and define
\be\la{bkj25}\ba
\tilde{G}(x_1,x_2) & \triangleq
\begin{cases}
G(x_1,x_2), \ & x_2 \ge 0, \\
G(x_1,-x_2), \ & x_2 < 0,
\end{cases} \\
\tilde{v}^i(x_1,x_2) & \triangleq
\begin{cases}
v^i(x_1,x_2), \quad & x_2 \ge 0, \\
(-1)^{i+1} v^i(x_1,-x_2), \quad & x_2 < 0.
\end{cases}
\ea\ee
A direct calculation shows that $\tilde{G}$ satisfies
\be\la{bkj26}\ba
\begin{cases}
\Delta{\tilde{G}} = \div (\tilde{v}) \ \   & \text{ in } \rr, \\
\tilde{G} \to 0 & \text{ as } \ |x| \to \infty.
\end{cases}
\ea\ee
By the properties of Riesz potentials (see \cite{SE}), we obtain for any $2<r<\infty$,
\be\la{bkj27}\ba
\| \tilde{G} \|_{L^r(\rr)} \le C \| \tilde{v} \|_{L^{\frac{2r}{2+r}}(\rr)},
\ea\ee
which together with (\ref{bkj25}) gives
\be\la{bkj28}\ba
\| G \|_{L^r} \le C \| v \|_{L^{\frac{2r}{2+r}}}
\le C \| \rho \dot{u} \|_{L^{\frac{2r}{2+r}}}.
\ea\ee
Moreover, note that $P$ satisfies
\be\la{bkj29}\ba
P = - G + \nu \div u.
\ea\ee
Multiplying (\ref{bkj29}) by $P$ and using (\ref{bkj28}) and Young's inequality, we arrive at
\be\la{bkj210}\ba
\int P^2 dx & = - \int G P dx + \nu \int \div u P dx \\
& \le \| G \|_{L^{4\ga}} \| P \|_{L^{\frac{4\gamma}{4\gamma-1}}} + \nu \| \div u\|_{L^2} \| P \|_{L^2} \\
& \le C \| \n \dot{u} \|_{ L^{\frac{4\ga}{2\ga+1}} } \| \n^{\frac{1}{2}} \|_{L^2} \| \n^{\ga-\frac{1}{2}} \|_{L^{ \frac{4\ga}{2\ga-1} }}
+ \nu \| \div u\|_{L^2} \| P \|_{L^2} \\
& \le C \| \n^{\frac{1}{2}} \|_{L^{4\ga}} \| \sqrt{\n} \dot{u} \|_{ L^{2} }
\| P \|^{1-\frac{1}{2\ga}}_{L^{2}} + \nu \| \div u\|_{L^2} \| P \|_{L^2} \\
& \le C \| \sqrt{\n} \dot{u} \|_{ L^{2} } \| P \|_{L^{2}}
+ \nu \| \div u\|_{L^2} \| P \|_{L^2} \\
& \le \frac{1}{2} \| P \|^2_{L^2} + C \| \sqrt{\n} \dot{u} \|^2_{L^2} + \nu^2 \| \div u \|^2_{L^2},
\ea\ee
where we have used the following fact:
\be\la{bkj211}\ba
\int \n dx = \int \n_0 dx,
\ea\ee
due to $(\ref{ns})_1$.
Hence, we have
\be\ba\la{bkj212}
\| P \|^2_{L^2} \le C \| \sqrt{\n} \dot{u}\|^2_{L^2} + 2 \nu^2 \| \div u\|^2_{L^2}.
\ea\ee

Next, multiplying (\ref{bkj21}) by $2 \dot u$ and integrating by parts over $\rr_+$, we obtain
\be\la{bkj213}\ba
& \frac{d}{dt} \int \left(\mu \o^2 + \frac{G^2}{\nu}\right)dx + 2\| \sqrt{\n} \dot{u}\|^2_{L^2}\\
& = -\mu \int \o^2\div udx - 4\int G\nabla u^1\cdot\nabla^{\perp}u^2dx- 2\int G(\div u)^2dx\\
&\quad +\frac{1}{\nu} \int G^2\div udx +\frac{2\ga}{\nu} \int P G\div udx = \sum_{i=1}^5 I_i,
\ea\ee
where we have used (\ref{bkjbjds}) and the following facts:
\be\la{bkj214}\ba
\na^{\bot}\cdot \dot u= \frac{D}{Dt}\o +(\p_1u\cdot\na) u^2
-(\p_2u\cdot\na)u^1  = \frac{D}{Dt}\o + \o \div u , 
\ea\ee
and
\be\la{bkj215}\ba 
\div  \dot u&=\frac{D}{Dt}\div u +(\p_1u\cdot\na) u^1+(\p_2u\cdot\na)u^2\\&
=\frac{1}{\nu} \frac{D}{Dt}G+ \frac{1}{\nu} \frac{D}{Dt}( P-P( \tilde{\n} ) )
+ 2\nabla u^1\cdot\nabla^{\perp}u^2 + (\div u)^2.	
\ea\ee
It follows from (\ref{gn11}), (\ref{bkj24}), and H\"older's inequality that
\be\la{bkj216}\ba
|I_1| &\le C \| \o\|^2_{L^4} \| \div u\|_{L^2} \\
& \le C \| \o\|_{L^2} \| \na \o\|_{L^2} \| \div u\|_{L^2} \\
& \le C \| \sqrt{\n} \dot{u}\|_{L^2} \| \o\|_{L^2} \| \div u\|_{L^2} \\
& \le \frac{1}{16} \| \sqrt{\n} \dot{u}\|^2_{L^2} + C \| \na u\|^4_{L^2}.
\ea\ee
By virtue of the boundary conditions (\ref{bjtj1}), (\ref{bkj24}), and Lemma \ref{hm}, we have
\be\la{bkj217}\ba
|I_2| & \le C \| \na G\|_{L^2} \| \na u\|^2_{L^2} \\
& \le C \| \sqrt{\n} \dot{u}\|_{L^2} \| \na u\|^2_{L^2} \\
& \le \frac{1}{16} \| \sqrt{\n} \dot{u}\|^2_{L^2}+ C \| \na u\|^4_{L^2}.
\ea\ee
Using (\ref{gn11}), (\ref{bkj24}), (\ref{gw}), (\ref{bkj212}), and H\"older's inequality, we derive
\be\la{bkj218}\ba
\sum_{i=3}^5I_i & \le \frac{C}{\nu} \int G^2 |\div u| dx + \frac{C}{\nu}\int P |G| |\div u| dx \\
& \le \frac{C}{\nu} \| G \|^2_{L^4} \| \div u \|_{L^2} + \frac{C}{\nu} \| G \|_{L^2} \| \div u \|_{L^2} \\
& \le \frac{C}{\nu} \| G \|_{L^2} \| \na G \|_{L^2} \| \div u \|_{L^2}
+ C \| \div u \|^2_{L^2} + \frac{C}{\nu} \| P \|_{L^2} \| \div u \|_{L^2} \\
& \le \frac{C}{\nu} \| G \|_{L^2} \| \sqrt{\n} \dot{u} \|_{L^2} \| \div u \|_{L^2}
+ C \| \div u \|^2_{L^2} + \frac{C}{\nu} \| \sqrt{\n} \dot{u} \|_{L^2} \| \div u \|_{L^2} \\
& \le \frac{1}{16} \| \sqrt{\n} \dot{u}\|^2_{L^2}
+ \frac{C}{\nu} \| G \|^2_{L^2} \| \na u\|^2_{L^2}+C\| \na u\|^2_{L^2}.
\ea\ee
Substituting (\ref{bkj216}), (\ref{bkj217}), and (\ref{bkj218}) into (\ref{bkj213}) gives
\be\la{bkj219}\ba
\frac{d}{dt} \left( \mu \| \o \|^2_{L^2} + \frac{1}{\nu} \| G \|^2_{L^2} \right) + \| \sqrt{\n} \dot{u}\|^2_{L^2}
\le C \left( \mu \| \o \|^2_{L^2} + \frac{1}{\nu} \| G \|^2_{L^2} \right) \| \na u\|^2_{L^2} + C\| \na u\|^2_{L^2},
\ea\ee
where we have used the following estimate:
\be\la{bkj220}\ba
\nu \| \div u \|^2_{L^2} + \| \na u\|^2_{L^2} & \le C \left( \nu \| \div u\|^2_{L^2}+\| \o \|^2_{L^2} \right) \\
& \le \frac{C}{\nu} \left( \| G \|^2_{L^2} + \| P \|^2_{L^2} \right) + C \| \o \|^2_{L^2},
\ea\ee
due to (\ref{dc1}) and (\ref{pro101}).
Applying Gr\"onwall's inequality to (\ref{bkj219}) and using (\ref{bkj01}), (\ref{pro101}), (\ref{bkj211}), and (\ref{bkj220}), we arrive at (\ref{bkj021}).

Furthermore, multiplying (\ref{bkj219}) by $\si$ leads to
\be\la{bkj221}\ba
& \frac{d}{dt} \left( \si \left( \mu \| \o \|^2_{L^2} + \frac{1}{\nu} \| G \|^2_{L^2} \right) \right) + \si \| \sqrt{\n} \dot{u}\|^2_{L^2} \\
& \le \si' \left( \mu \| \o \|^2_{L^2} + \frac{1}{\nu} \| G \|^2_{L^2} \right)
+ C \si \left( \mu \| \o \|^2_{L^2} + \frac{1}{\nu} \| G \|^2_{L^2} \right) \| \na u\|^2_{L^2} + C \| \na u\|^2_{L^2},
\ea\ee
which together with (\ref{bkj01}), (\ref{pro101}), (\ref{bkj211}), (\ref{bkj220}), and Gr\"onwall's inequality implies
\be\la{bkj222}\ba
\sup_{0\le t\le T} \si \left( \| \na u \|^2_{L^2} + \nu \| \div u \|^2_{L^2} \right)
+ \int_0^T \si \| \sqrt{\n} \dot{u}\|^2_{L^2} dt \le C.
\ea\ee
Finally, in view of (\ref{bkj211}), (\ref{bkj212}), (\ref{bkj01}), and (\ref{bkj222}), we have
\be\la{bkj223}\ba
\int_0^T \frac{1}{\nu} \| P \|^2_{L^2} dt
& = \int_0^{\si(T)} \frac{1}{\nu} \| P \|^2_{L^2} dt + \int_{\si(T)}^T \frac{1}{\nu} \| P \|^2_{L^2} dt \\
& \le C + C \int_{\si(T)}^T \left( \| \sqrt{\n} \dot{u} \|^2_{L^2} + \nu \| \div u \|^2_{L^2} \right) dt \le C.
\ea\ee
This, combined with (\ref{bkj222}), yields (\ref{bkj022}) and completes the proof of Lemma \ref{bkjl2}.
\end{proof}

\begin{lemma}\la{bkjl3}
Let $(\n,u)$ be a strong solution to \eqref{ns}--\eqref{bjtj2} satisfying \eqref{pro101}.
Then there exists a positive constant $C$ depending only on
$\ga$, $\mu$, $E_0$, and $\| \n_0 \|_{L^\infty}$ such that
\be\la{bkj301}\ba
\sup_{\si(T) \le t\le T} t \left( \| \na u \|^2_{L^2} + \nu \| \div u \|^2_{L^2} + \frac{1}{\nu} \| P \|^2_{L^2} \right)
+ \int_{\si(T)}^T t \left( \| \sqrt{\n} \dot{u}\|^2_{L^2} + \frac{1}{\nu^2} \| P \|^3_{L^3} \right) dt
\le C,
\ea\ee
and
\be\la{bkj302}\ba
\sup_{\si(T) \le t\le T} \left( \frac{1}{\nu^2} t^2 \| P \|_{L^3}^3 \right) + \int_{\si(T)}^T \frac{1}{\nu^3} t^2 \| P \|^4_{L^4} dt
\le C.
\ea\ee
\end{lemma}
\begin{proof}
Using (\ref{bkj213}), (\ref{bkj216}), (\ref{bkj217}), (\ref{bkj212}), (\ref{gw}), and Young's inequality, we derive
\be\la{bkj31}\ba
& \frac{d}{dt} \left( \mu \| \o \|^2_{L^2} + \frac{1}{\nu} \| G \|^2_{L^2} \right) + \frac{3}{2} \| \sqrt{\n} \dot{u} \|^2_{L^2} \\
& \le C \| \na u\|^4_{L^2}
+ \frac{C}{\nu} \int G^2 |\div u| dx + \frac{C}{\nu}\int P |G| |\div u| dx \\
& \le C \| \na u\|^4_{L^2} + \frac{C}{\nu} \| G \|^2_{L^3} \| \div u \|_{L^3} + \frac{C}{\nu} \| P \|_{L^3} \| G \|_{L^3} \| \div u \|_{L^3} \\
& \le C \| \na u\|^4_{L^2} + \frac{C}{\nu^2} \| G \|^3_{L^3} + \frac{2\ga -1}{4\nu^2} \| P \|^3_{L^3} \\
& \le C \| \na u\|^4_{L^2} + \frac{C}{\nu^2} \| G \|^2_{L^2} \| \na G \|_{L^2} + \frac{2\ga -1}{4\nu^2} \| P \|^3_{L^3} \\
& \le C \| \na u\|^4_{L^2} + \frac{C}{\nu^2} \| G \|^2_{L^2} \| \sqrt{\n} \dot{u} \|_{L^2} + \frac{2\ga -1}{4\nu^2} \| P \|^3_{L^3} \\
& \le \frac{1}{4} \| \sqrt{\n} \dot{u}\|^2_{L^2} + C \| \na u\|^4_{L^2}
+ \frac{C}{\nu^4} \| G \|^4_{L^2} + \frac{2\ga -1}{4\nu^2} \| P \|^3_{L^3},
\ea\ee
which gives
\be\la{bkj32}\ba
\frac{d}{dt} \left( \mu \| \o \|^2_{L^2} + \frac{1}{\nu} \| G \|^2_{L^2} \right) + \frac{5}{4} \| \sqrt{\n} \dot{u}\|^2_{L^2}
\le C \| \na u\|^4_{L^2} + \frac{C}{\nu^2} \| G \|^4_{L^2} + \frac{2\ga -1}{4\nu^2} \| P \|^3_{L^3}.
\ea\ee
By $(\ref{ns})_1$, we deduce that $P$ satisfies
\be\la{bkj33}\ba
P_t + u \cdot \na P + \ga P \div u = 0.
\ea\ee
For any $2 \le p < \infty$, multiplying (\ref{bkj33}) by $p P^{p-1}$, integrating by parts over $\rr_+$, and using (\ref{bjtj1}), (\ref{gw}), (\ref{gn11}), and Young's inequality, we obtain
\be\la{bkj34}\ba
\frac{d}{dt} \| P \|^p_{L^p} + \frac{p\ga -1}{\nu} \| P \|^{p+1}_{L^{p+1}}
& = - \frac{p\ga -1}{\nu} \int P^p G dx \\
& \le \frac{p\ga -1}{2\nu} \| P \|^{p+1}_{L^{p+1}} + \frac{C}{\nu} \| G \|^{p+1}_{L^{p+1}} \\
& \le \frac{p\ga -1}{2\nu} \| P \|^{p+1}_{L^{p+1}}
+ \frac{C}{\nu} \| G \|^{2}_{L^2} \| \sqrt{\n} \dot{u} \|^{p-1}_{L^2},
\ea\ee
which implies that
\be\la{bkj35}\ba
\frac{d}{dt} \left( \frac{1}{\nu} \| P \|^p_{L^p} \right) + \frac{p\ga -1}{2 \nu^2} \| P \|^{p+1}_{L^{p+1}}
\le \frac{C}{\nu^2} \| G \|^{2}_{L^2} \| \sqrt{\n} \dot{u} \|^{p-1}_{L^2}.
\ea\ee
Choosing $p=2$ in (\ref{bkj35}) yields
\be\la{bkj36}\ba
\frac{d}{dt} \left( \frac{1}{\nu} \| P \|^2_{L^2} \right) + \frac{2\ga -1}{2 \nu^2} \| P \|^{3}_{L^{3}}
\le \frac{C}{\nu^2} \| G \|^{2}_{L^2} \| \sqrt{\n} \dot{u} \|_{L^2}
\le \frac{1}{4} \| \sqrt{\n} \dot{u} \|^2_{L^2} + \frac{C}{\nu^4} \| G \|^4_{L^2}.
\ea\ee
The combination of (\ref{bkj32}), (\ref{bkj35}), (\ref{gw}), and (\ref{bkj220}) gives
\be\la{bkj37}\ba
& \frac{d}{dt} \left( \mu \| \o \|^2_{L^2} + \frac{1}{\nu} \| G \|^2_{L^2} + \frac{1}{\nu} \| P \|_{L^2}^2 \right) + \| \sqrt{\n} \dot{u}\|^2_{L^2} + \frac{2\ga -1}{4\nu^2} \| P \|^3_{L^3} \\
& \le C \| \na u\|^4_{L^2} + \frac{C}{\nu^2} \| G \|^4_{L^2} \\
& \le C \left( \mu \| \o \|^2_{L^2} + \frac{1}{\nu} \| G \|^2_{L^2} + \frac{1}{\nu} \| P \|_{L^2}^2 \right)
\left( \| \na u\|^2_{L^2} + \nu \| \div u \|^2_{L^2} + \frac{C}{\nu} \| P \|^2_{L^2} \right).
\ea\ee
Multiplying (\ref{bkj37}) by $t$ leads to
\be\la{bkj38}\ba
& \frac{d}{dt} \left( t \left( \mu \| \o \|^2_{L^2} + \frac{1}{\nu} \| G \|^2_{L^2} + \frac{1}{\nu} \| P \|_{L^2}^2 \right) \right) + t \| \sqrt{\n} \dot{u}\|^2_{L^2} + \frac{2\ga -1}{4\nu^2} t \| P \|^3_{L^3} \\
& \le C t \left( \mu \| \o \|^2_{L^2} + \frac{1}{\nu} \| G \|^2_{L^2} + \frac{1}{\nu} \| P \|_{L^2}^2 \right)
\left( \| \na u\|^2_{L^2} + \nu \| \div u \|^2_{L^2} + \frac{C}{\nu} \| P \|^2_{L^2} \right) \\
& \quad + C \left( \| \na u\|^2_{L^2} + \nu \| \div u \|^2_{L^2} + \frac{C}{\nu} \| P \|^2_{L^2} \right).
\ea\ee
Applying Gr\"onwall's inequality to (\ref{bkj38}) over $(\si(T),T)$ and using (\ref{bkj01}), (\ref{bkj022}), and (\ref{bkj220}), we arrive at
\be\la{bkj39}\ba
\sup_{\si(T) \le t\le T} t \left( \| \na u \|^2_{L^2} + \nu \| \div u \|^2_{L^2} + \frac{1}{\nu} \| P \|_{L^2}^2 \right) + \int_{\si(T)}^T t \left( \| \sqrt{\n} \dot{u}\|^2_{L^2} + \frac{1}{\nu^2} \| P \|^3_{L^3} \right) dt
\le C,
\ea\ee
which together with (\ref{bkj220}) shows (\ref{bkj301}).

In addition, choosing $p=3$ in (\ref{bkj34}), we obtain
\be\la{bkj310}\ba
\frac{d}{dt} \| P \|^3_{L^3} + \frac{3\ga -1}{2 \nu} \| P \|^{4}_{L^{4}}
\le \frac{C}{\nu} \| G \|^{2}_{L^2} \| \sqrt{\n} \dot{u} \|^{2}_{L^2},
\ea\ee
which yields
\be\la{bkj311}\ba
\frac{d}{dt} \left( \frac{1}{\nu^2} t^2 \| P \|^3_{L^3} \right) + \frac{3\ga -1}{2 \nu^3} t^2 \| P \|^{4}_{L^{4}}
\le \frac{2}{\nu^2} t \| P \|^3_{L^3} + \frac{C}{\nu^3} t^2 \| G \|^{2}_{L^2} \| \sqrt{\n} \dot{u} \|^{2}_{L^2}.
\ea\ee
Integrating (\ref{bkj311}) over $(\si(T),T)$ and applying (\ref{bkj021}), (\ref{bkj022}), and (\ref{bkj39}) lead to (\ref{bkj302}), thereby completing the proof of Lemma \ref{bkjl3}.
\end{proof}

\begin{lemma}\la{bkjl4}
Let $(\n,u)$ be a strong solution to \eqref{ns}--\eqref{bjtj2} satisfying \eqref{pro101}.
Then there exists a positive constant $C$ depending only on
$\ga$, $\mu$, $E_0$, and $\| \n_0 \|_{L^\infty}$ such that
\be\ba\la{bkj401}
\sup_{0\le t\le T}
t^2 \int\n|\dot u|^2dx
+\int_0^{T} t^2 \| \na\dot u\|^2_{L^2} dt \le C.
\ea\ee
Moreover, for any $p \in [2,\infty)$, there exists a positive constant $C$ depending only on
$p$, $\ga$, $\mu$, $E_0$, and $\| \n_0 \|_{L^\infty}$ such that
\be\la{bkj402}\ba
\sup_{\si(T) \le t \le T} t^{p-1} \left( \frac{1}{\nu^{p-1}} \| P \|^{p}_{L^{p}} + \| \na u \|^p_{L^p} \right)
\le C.
\ea\ee
\end{lemma}
\begin{proof}
First, adapting the approach in \cite{H1,Lei1}, we apply the operator $ \dot{u}^j[\frac{\pa}{\pa t}+\div(u\cdot)]$ to $(\ref{bkj21})^j$, sum with respect to $j$, and integrate by parts over $\rr_+$ to derive
\be\la{bkj41}\ba
\frac{d}{dt}\left(\frac{1}{2}\int\rho|\dot{u}|^2dx \right)
&=\int \bigg( {\dot{u}}\cdot \nabla G_t +{\dot{u}}^j\mathrm {div}(u \partial _jG) \bigg) dx\\
&\quad +\mu \int \bigg( {\dot{u}}\cdot \nabla ^\bot \omega _t 
+{\dot{u}}^j\partial _k(u^k (\nabla ^\bot \omega )_j ) \bigg) dx \\
& \triangleq I_1+I_2.
\ea\ee
Integration by parts along with (\ref{bjtj1}), (\ref{bkjbjds}), (\ref{bkj24}), and Lemma \ref{hm} yields
\be\la{bkj42}\ba
I_1 & = \int_{\partial \rr_+} G_t ( \dot{u} \cdot n) ds
- \int \div \dot{u} \left( \dot{G} - u \cdot \na G \right) dx
- \int u \cdot \na \dot{u}^j \p_j G dx \\
& = - \int \div \dot{u} \dot{G} dx
+ \int \left( u \cdot \na G \div \dot{u} - u \cdot \na \dot{u}^j \p_j G \right) dx \\
& = - \int \div \dot{u} \dot{G} dx
+ \int \left( - \div u  \div \dot{u} G + \p_j u \cdot \na \dot{u}^j G \right) dx \\
& = - \int \div \dot{u} \dot{G} dx
+ \int G \left( \na u^1 \cdot \na^\bot \dot{u}^2 - \na u^2 \cdot \na^\bot \dot{u}^1 \right) dx \\
& \le - \int \div \dot{u} \dot{G} dx
+ C \| \na G \|_{L^2} \| \na u \|_{L^2} \| \na \dot{u} \|_{L^2} \\
& \le - \int \div \dot{u} \dot{G} dx
+ \ep \| \na \dot{u} \|^2_{L^2}
+ C \| \sqrt{\n} \dot{u} \|^2_{L^2} \| \na u \|^2_{L^2},
\ea\ee
where we have used the following fact:
\be\la{bkj43}\ba
\sum_{j=1}^{2} \p_j u \cdot \na \dot{u}^j
= \div u \div \dot{u} + \na u^1 \cdot \na^\bot \dot{u}^2 - \na u^2 \cdot \na^\bot \dot{u}^1.
\ea\ee
From (\ref{bkj33}) and the definition of $G$, we conclude that
\be\ba\la{bkj44}
\div \dot{u} & = (\div u)_t + \p_i u^j \p_j u^i + u \cdot \na \div u \\
& = \frac{1}{\nu} \dot{G} + \frac{1}{\nu} (P_t+u \cdot \na P) + \p_i u^j \p_j u^i \\
& = \frac{1}{\nu} \dot{G} - \frac{\ga}{\nu} P \div u + \p_i u^j \p_j u^i,
\ea\ee
which together with Young's inequality gives
\be\la{bkj45}\ba
- \int \div \dot{u} \dot{G} dx
& = -\frac{1}{\nu} \| \dot{G} \|^2_{L^2} + \frac{\ga}{\nu} \int \dot{G} P \div u dx
- \int \dot{G} \p_i u^j \p_j u^i dx \\
& \le -\frac{1}{2\nu} \| \dot{G} \|^2_{L^2}
+ \frac{C}{\nu^3} \| P \|^4_{L^4} + \frac{C}{\nu^3} \| G \|^4_{L^4}
- \int \dot{G} \p_i u^j \p_j u^i dx.
\ea\ee
For the last term in the final line of (\ref{bkj45}), integration by parts leads to
\be\la{bkj46}\ba
- \int \dot{G} \p_i u^j \p_j u^i dx 
& = - \int ( G_t + u \cdot \na G ) \p_i u \cdot \na u^i dx \\
& = - \frac{d}{dt} \left( \int G \p_i u \cdot \na u^i dx \right) 
+ 2 \int G \p_i u \cdot \na u^i_t dx \\
& \quad + \int G \div u \p_i u \cdot \na u^i dx 
+ 2 \int G u \cdot \na \p_i u \cdot \na u^i dx \\
& = - \frac{d}{dt} \left( \int G \p_i u \cdot \na u^i dx \right) 
+ 2 \int G \p_i u \cdot \na \dot{u}^i dx \\
& \quad - 2 \int G \p_i u \cdot \na u \cdot \na u^i dx
+ \int G \div u \p_i u \cdot \na u^i dx.
\ea\ee
By (\ref{bkj24}), Lemma \ref{hm}, and Young's inequality, we derive
\be\la{bkj47}\ba
2 \int G \p_i u \cdot \na \dot{u}^i dx
& = 2 \int G \left( \div u \div \dot{u} + \na u^1 \cdot \na^\bot \dot{u}^2 - \na u^2 \cdot \na^\bot \dot{u}^1 \right) dx \\
& \le C \| \na \dot{u} \|_{L^2} \| G \|_{L^4} \| \div u \|_{L^4}
+ C \| \na G \|_{L^2} \| \na \dot{u} \|_{L^2} \| \na u \|_{L^2} \\
& \le \ep \| \na \dot{u} \|^2_{L^2} + \frac{C}{\nu} \| G \|^4_{L^4}
+ \frac{C}{\nu^3} \| P \|^4_{L^4}
+ C \| \sqrt{\n} \dot{u} \|^2_{L^2} \| \na u \|^2_{L^2}.
\ea\ee
Note that
\be\ba\la{bkj408}
\sum_{i=1}^{2} \p_i u \cdot \na u \cdot \na u^i
= (\div u)^3 +3 \div u \na u^1\cdot \na^\bot u^2,
\ea\ee
which together with H\"older's inequality yields
\be\la{bkj48}\ba
- 2 \int G \p_i u \cdot \na u \cdot \na u^i dx
& = - 2 \int G \left( (\div u)^3+3\div u \na u^1 \cdot \na^\bot u^2 \right) dx \\
& \le C \| G \|_{L^4} \| \div u \|_{L^4} \| \na u \|^2_{L^4} \\
& \le \frac{C}{\nu} \| G \|^4_{L^4} + \frac{C}{\nu^3} \| P \|^4_{L^4} 
+ C \| \na u \|^4_{L^4}.
\ea\ee
Similarly, we have
\be\la{bkj49}\ba
\int G \div u \p_i u \cdot \na u^i dx 
\le \frac{C}{\nu} \| G \|^4_{L^4} + \frac{C}{\nu^3} \| P \|^4_{L^4}
+ C \| \na u \|^4_{L^4}.
\ea\ee
Substituting (\ref{bkj47}), (\ref{bkj48}) and (\ref{bkj49}) into (\ref{bkj46}) gives
\be\la{bkj410}\ba
- \int \dot{G} \p_i u^j \p_j u^i dx
& \le - \frac{d}{dt} \left( \int G \p_i u \cdot \na u^i dx \right)
+ \ep \| \na \dot{u} \|^2_{L^2} + \frac{C}{\nu} \| G \|^4_{L^4} \\
& \quad + \frac{C}{\nu^3} \| P \|^4_{L^4} + C \| \sqrt{\n} \dot{u} \|^2_{L^2} \| \na u \|^2_{L^2} + C \| \na u \|^4_{L^4}.
\ea\ee
It follows from (\ref{gn11}), (\ref{gw}), and (\ref{bkj24}) that
\be\la{bkj411}\ba
\frac{1}{\nu} \| G \|^4_{L^4} + \|\na u\|^4_{L^4}
& \le \frac{1}{\nu} \| G \|^4_{L^4} + C \left( \|\div u\|^4_{L^4} + \|\o \|^4_{L^4} \right) \\
& \le \frac{C}{\nu} \| G \|^4_{L^4} + \frac{C}{\nu^4} \| P \|^4_{L^4}
+ C \|\o \|^2_{L^2} \| \na \o \|^2_{L^2} \\
& \le \frac{C}{\nu} \| G \|^2_{L^2} \| \na G \|^2_{L^2} + \frac{C}{\nu^4} \| P \|^4_{L^4}
+ C \|\na u \|^2_{L^2} \| \sqrt{\n} \dot{u} \|^2_{L^2} \\
& \le C \left( \nu \| \div u \|^2_{L^2} + \frac{1}{\nu} \| P \|^2_{L^2} + \| \na u \|^2_{L^2} \right) \| \sqrt{\n} \dot{u} \|^2_{L^2} + \frac{C}{\nu^4} \| P \|^4_{L^4}.
\ea\ee
Combining (\ref{bkj42}), (\ref{bkj45}), (\ref{bkj410}) and (\ref{bkj411}), we arrive at
\be\la{bkj412}\ba
I_1 & \le - \frac{d}{dt} \left( \int G \p_i u \cdot \na u^i dx \right)
- \frac{1}{2\nu} \| \dot{G} \|^2_{L^2} + 2 \ep \| \na \dot{u} \|^2_{L^2} + \frac{C}{\nu^3} \| P \|^4_{L^4} \\
& \quad + C \left( \nu \| \div u \|^2_{L^2} + \frac{1}{\nu} \| P \|^2_{L^2} + \| \na u \|^2_{L^2} \right) \| \sqrt{\n} \dot{u} \|^2_{L^2}.
\ea\ee
For $I_2$, we integrate by part over $\rr_+$ and use the boundary conditions (\ref{bjtj1}) and Young's inequality to derive
\be\la{bkj413}\ba
I_2&=\mu \int \left({\dot{u}}\cdot \nabla ^\bot \omega _t
+{\dot{u}}^j \p_k \left( u^k (\na^\bot \o)^j \right) \right) dx \\
&=\mu \int_{\partial \rr_+} ({\dot{u}}\cdot n^\bot )\omega _t ds
-\mu \int \curl \dot{u} \o_t dx-\mu \int u \cdot \na \dot{u} \cdot \na^\bot \o dx \\
&= - \mu \int (\curl \dot{u})^2 dx + \mu \int \curl \dot{u} \curl (u \cdot \na u) dx
+ \mu \int \left( u \cdot \na \curl \dot{u} + \na^\bot_j u \cdot \na \dot{u}^j \right) \o dx \\
& \le - \mu \int (\curl \dot{u})^2 dx + \mu \int \curl \dot{u} (\na^\bot_j u^i \p_i u^j + u \cdot \na \o) dx
+ \mu \int \left( u \cdot \na \curl \dot{u} + \na^\bot_j u \cdot \na \dot{u}^j \right) \o dx \\
& = - \mu \int (\curl \dot{u})^2 dx + \mu \int \curl \dot{u} (\na^\bot_j u^i \p_i u^j - \o \div u ) dx
+ \mu \int \na^\bot_j u \cdot \na \dot{u}^j \o dx \\
& \le - \mu \| \curl \dot{u}\|^2_{L^2} + \ep \| \na \dot{u} \|^2_{L^2} + C \| \na u \|^4_{L^4}.
\ea\ee
Putting (\ref{bkj412}) and (\ref{bkj413}) into (\ref{bkj41}) and using (\ref{bkj411}), we arrive at
\be\la{bkj414}\ba
& \frac{d}{dt}\left(\frac{1}{2} \int \rho |\dot{u}|^2 dx + \int G \p_i u \cdot \na u^i dx \right)
+ \frac{1}{2 \nu} \| \dot{G} \|^2_{L^2} + \mu \| \curl \dot{u} \|^2_{L^2} \\
& \le 3 \ep \| \na \dot{u} \|^2_{L^2} + \frac{C}{\nu^3} \| P \|^4_{L^4}
+ C \left( \nu \| \div u \|^2_{L^2} + \frac{1}{\nu} \| P \|^2_{L^2} + \| \na u \|^2_{L^2} \right) \| \sqrt{\n} \dot{u} \|^2_{L^2}.
\ea\ee
By virtue of (\ref{bkj44}) and Young's inequality, we have
\be\la{bkj415}\ba
\| \div \dot{u} \|^2_{L^2}
& \le \frac{C}{\nu^2} \| \dot{G} \|^2_{L^2}
+ \frac{C}{\nu^4} \| P \|^4_{L^4} + \frac{C}{\nu^4} \| G \|^4_{L^4} + C \| \na u \|^4_{L^4},
\ea\ee
which together with (\ref{bkjbjds}), (\ref{dc1}), and (\ref{bkj411}) yields
\be\la{bkj416}\ba
\| \na \dot{u} \|^2_{L^2}
& \le C \left( \| \div \dot{u} \|^2_{L^2} + \| \curl \dot{u} \|^2_{L^2} \right) \\
& \le C \left( \frac{1}{2\nu} \| \dot{G} \|^2_{L^2} + \| \curl \dot{u} \|^2_{L^2} \right) + \frac{C}{\nu^3} \| P \|^4_{L^4} \\
& \quad + C \left( \nu \| \div u \|^2_{L^2} + \frac{1}{\nu} \| P \|^2_{L^2} + \| \na u \|^2_{L^2} \right) \| \sqrt{\n} \dot{u} \|^2_{L^2}.
\ea\ee
Substituting (\ref{bkj416}) into (\ref{bkj414}) and choosing $\ep$ sufficiently small, we obtain
\be\la{bkj417}\ba
& \frac{d}{dt}\left(\frac{1}{2} \int \rho |\dot{u}|^2 dx + \int G \p_i u \cdot \na u^i dx \right)
+ \frac{1}{4 \nu} \| \dot{G} \|^2_{L^2} + \frac{\mu}{2} \| \curl \dot{u} \|^2_{L^2} \\
& \le \frac{C}{\nu^3} \| P \|^4_{L^4}
+ C \left( \nu \| \div u \|^2_{L^2} + \frac{1}{\nu} \| P \|^2_{L^2} + \| \na u \|^2_{L^2} \right) \| \sqrt{\n} \dot{u} \|^2_{L^2}.
\ea\ee
From (\ref{gn11}), (\ref{bkj24}), (\ref{bkj43}), Lemma \ref{hm}, and Young's inequality, we deduce that
\be\la{bkj418}\ba
\left| \int G \p_i u \cdot \na u^i dx \right|
& = \left| \int \left( G (\div u)^2 + 2 G \na u^1 \cdot \na^\bot u^2 \right) dx \right| \\
& \le \frac{C}{\nu^2} \| G \|^3_{L^3} + \frac{C}{\nu^2} \| P \|^3_{L^3}
+ C \| \na G \|_{L^2} \| \na u \|^2_{L^2} \\
& \le \frac{C}{\nu^2} \| G \|^2_{L^2} \| \na G \|_{L^2} + \frac{C}{\nu^2} \| P \|^3_{L^3}
+ C \| \sqrt{\n} \dot{u} \|_{L^2} \| \na u \|^2_{L^2} \\
& \le \frac{C}{\nu^2} \| G \|^2_{L^2} \| \sqrt{\n} \dot{u} \|_{L^2} + \frac{C}{\nu^2} \| P \|^3_{L^3}
+ C \| \sqrt{\n} \dot{u} \|_{L^2} \| \na u \|^2_{L^2} \\
& \le \frac{1}{4} \| \sqrt{\n} \dot{u} \|^2_{L^2}
+\frac{C}{\nu^4} \| G \|^4_{L^2} + \frac{C}{\nu^2} \| P \|^3_{L^3} + C \| \na u \|^4_{L^2} \\
& \le \frac{1}{4} \| \sqrt{\n} \dot{u} \|^2_{L^2}
+ \frac{C}{\nu^4} \| P \|^4_{L^2} + \frac{C}{\nu^2} \| P \|^3_{L^3} + C \| \na u \|^4_{L^2}.
\ea\ee
Multiplying (\ref{bkj417}) by $t^2$ and using (\ref{bkj022}), (\ref{bkj301}), and (\ref{bkj418}), we derive
\be\la{bkj419}\ba
& \frac{d}{dt}\left(\frac{1}{2} t^2 \int \rho |\dot{u}|^2 dx + t^2 \int G \p_i u \cdot \na u^i dx \right)
+ \frac{1}{4 \nu} t^2 \| \dot{G} \|^2_{L^2} + \frac{\mu}{2} t^2 \| \curl \dot{u} \|^2_{L^2} \\
& \le \frac{C}{\nu^3} t^2 \| P \|^4_{L^4}
+ C t \| \sqrt{\n} \dot{u} \|^2_{L^2} + \frac{C}{\nu} \| P \|^2_{L^2} + \frac{C}{\nu^2} t \| P \|^3_{L^3} + C \| \na u \|^2_{L^2}.
\ea\ee
Integrating (\ref{bkj419}) over $(0,T)$ and using (\ref{bkj01}), (\ref{bkj022}), (\ref{bkj301}), (\ref{bkj302}), (\ref{bkj416}), and (\ref{bkj418}), we arrive at (\ref{bkj401}).

It remains to prove (\ref{bkj402}).
Using (\ref{dc1}) and (\ref{gn11}), we obtain for any $2 \le r<\infty$,
\be\la{bkj420}\ba
\| \na u \|^r_{L^r}
& \le C \| \div u \|^r_{L^r} + C \| \o \|^r_{L^r} \\
& \le \frac{C}{\nu^r} \left( \| G \|^r_{L^r} + \| P \|^r_{L^r} \right) + C \| \o \|^2_{L^2} \| \na \o \|^{r-2}_{L^2} \\
& \le \frac{C}{\nu^r} \| G \|^2_{L^2} \| \na G \|^{r-2}_{L^2} + \frac{C}{\nu^r} \| P \|^r_{L^r} + C \| \na u \|^2_{L^2} \| \sqrt{\n} \dot{u} \|^{r-2}_{L^2} \\
& \le C \left( \frac{1}{\nu} \| G \|^2_{L^2} + \| \na u \|^2_{L^2} \right) \| \sqrt{\n} \dot{u} \|^{r-2}_{L^2} + \frac{C}{\nu^r} \| P \|^r_{L^r}.
\ea\ee
We claim that for $n \in \mathbb{N}^+$,
\be\la{bkj421}\ba
\sup_{\si(T) \le t \le T} \left( \frac{1}{\nu^n} t^n \| P \|^{n+1}_{L^{n+1}} \right)
+ \int_{\si(T)}^T \frac{1}{\nu^{n+1}} t^{n} \| P \|^{n+2}_{L^{n+2}} dt \le C.
\ea\ee
This, combined with (\ref{bkj401}), (\ref{bkj420}), and H\"older's inequality, yields (\ref{bkj402}).

We shall prove (\ref{bkj421}) by induction.
First, (\ref{bkj301}) shows that (\ref{bkj421}) holds for $n=1$.
Assume that (\ref{bkj421}) holds for $n=k$, that is,
\be\la{bkj422}\ba
\sup_{\si(T) \le t \le T} \left( \frac{1}{\nu^k} t^k \| P \|^{k+1}_{L^{k+1}} \right)
+ \int_{\si(T)}^T \frac{1}{\nu^{k+1}} t^{k} \| P \|^{k+2}_{L^{k+2}} dt \le C.
\ea\ee
Choosing $p=k+2$ in (\ref{bkj35}), multiplying the resulting inequality by $\frac{1}{\nu^k} t^{k+1}$, and using (\ref{bkj401}), we obtain
\be\la{bkj423}\ba
& \frac{d}{dt} \left( \frac{1}{\nu^{k+1}} t^{k+1} \| P \|^{k+2}_{L^{k+2}} \right)
+ \frac{(k+2)\ga-1}{2 \nu^{k+2}} t^{k+1} \| P \|^{k+3}_{L^{k+3}} \\
& \le \frac{k+1}{\nu^{k+1}} t^k \| P \|^{k+2}_{L^{k+2}}
+ \frac{C}{\nu^{k+2}} t^{k+1} \| G \|^{2}_{L^2} \| \sqrt{\n} \dot{u} \|^{k+1}_{L^2} \\
& \le \frac{k+1}{\nu^{k+1}} t^k \| P \|^{k+2}_{L^{k+2}} + C \left( \| \na u \|^{2}_{L^2} + \frac{1}{\nu} \| P \|^2_{L^2} \right).
\ea\ee
Integrating (\ref{bkj423}) over $[\si(T),T]$ and using (\ref{bkj422}), (\ref{bkj01}), and (\ref{bkj022}), we deduce that (\ref{bkj421}) holds for $n=k+1$.
By induction, we arrive at (\ref{bkj402}) and complete the proof of Lemma \ref{bkjl4}.
\end{proof}

The following lemma plays a crucial role in deriving the upper bound for the density.
\begin{lemma}\la{bkjl5}
Let $(\n,u)$ be a strong solution to \eqref{ns} on $\rr_+ \times (0,T]$
with $ \tilde{\n}=0 $.
Then for any $r \in [2,\infty)$, there exists a positive constant $C$ depending only on
$a$, $\Vert {\bar{x}}^a \rho_0 \Vert_{L^1},\ N_0$, $E_0$, $\| \n_0 \|_{L^\infty}$, and $r$
such that for all $t \in (0,T]$,
\be\la{bkj501}\ba
\left( \int_{\rr_+}\rho |v|^r dx\right) ^{1/r} 
\le C (1+t)^4 \left( \| \sqrt{\rho} v \| _{L^2} + \| \nabla v \|_{L^2}\right),
\ea\ee
for any $v\in \left\{ v\in D^1 ({\rr_+}) \mid \sqrt{\rho} v \in L^2(\rr_+) \right\}$.
\end{lemma}
\begin{proof}
First, for any integer $N>1$, choose a smooth cutoff function $ \varphi_N $ on $\rr$ satisfying:
\be\la{bkj51}\ba
0 \le \varphi_N \le 1, \quad \varphi_N=
\begin{cases}
1,\quad &\mathrm{ if \  } |x| \le N,\\
0,\quad &\mathrm{ if \  } |x| \ge 2N,
\end{cases}
\quad |\na \varphi_N | \le 2 N^{-1}.
\ea\ee
Denote $\varphi^+_N \triangleq \varphi_{N}|_{\rr_+}$.
Multiplying $(\ref{ns})_1$ by $\varphi^+_N$, integrating by parts over $\rr_+$, and using (\ref{bjtj1}), (\ref{bkj01}), and (\ref{bkj51}), we derive
\be\ba\nonumber
\frac{d}{dt} \int_{\rr_+} \n \varphi^+_N dx =\int_{\rr_+} \n u \cdot \na \varphi^+_N dx 
\ge -2 N^{-1} \left( \int_{\rr_+} \n dx \right)^{\frac{1}{2}} 
\left( \int_{\rr_+} \n |u|^2 dx \right)^{\frac{1}{2}} \ge -2 \hat{C} N^{-1},
\ea\ee
where the positive constant $\hat{C}$ depends only on $E_0$ and $\| {\bar{x}}^a \rho_0 \|_{L^1}$.

Taking $\tilde{N} \triangleq 4(1+N_0+4\hat{C}t)$, where $N_0$ is defined in \eqref{rho00}, we use (\ref{rho00}) to derive
\be\la{bkj52}\ba
\int_{B^+_{\tilde{N}}} \n dx \ge \int \n \varphi^+_{\tilde{N}/2} dx
& \ge \int \n_0 \varphi^+_{\tilde{N}/2} dx - 4 \hat{C} \tilde{N}^{-1}t \\
& \ge \int \n_0 \varphi^+_{N_0} dx - 4 \hat{C} \tilde{N}^{-1}t \\
& \ge \int_{B^+_{N_0}} \n_0 dx - 4 \hat{C} \tilde{N}^{-1}t \ge \frac{1}{4}.
\ea\ee
Thus, there exists a positive constant $N_1$ depending on
$N_0$, $\| {\bar{x}}^a \rho_0 \|_{L^1}$, and $E_0$,
such that for all $t \in (0,T]$,
\be\la{bkj53}\ba
\int _{B^+_{N_1(1+t)}}\rho (x,t) dx \ge \frac{1}{4}.
\ea\ee
Then, for any $v \in D_+^1(\rr_+)$, we extend it to $\rr$ by
\be\la{bkj54}\ba
\tilde{v}(x_1,x_2) \triangleq
\begin{cases}
v(x_1,x_2), \ & x_2 \ge 0, \\
v(x_1,-x_2), \ & x_2 < 0,
\end{cases}
\ea\ee
so that $\tilde{v} \in D^1(\rr)$.
In addition, we extend $\n$ by zero outside $\mathbb{R}^2_+$ and denote the extension by $\tilde{\n}$.

From (\ref{bkj53}) and Lemma \ref{esnr}, we conclude that for any $r \in [2,\infty)$ and $\beta>0$,
\be\la{bkj55}\ba
\left( \int_{\rr} \tilde{\rho} |\tilde{v}|^r dx\right) ^{1/r} 
& \le C (1+t)^3 (1+\Vert {\bar{x}}^\beta \tilde{\rho} \Vert_{L^1(\rr)}) 
\left( \Vert \sqrt{\tilde{\rho}} \tilde{v} \Vert _{L^2(\rr)} 
+ \Vert \nabla \tilde{v} \Vert _{L^2(\rr)}\right),
\ea\ee
which together with (\ref{bkj54}) gives
\be\la{bkj56}\ba
\left( \int _{\rr_+}\rho |v|^r dx\right) ^{1/r} 
& \le C (1+t)^3 (1+\Vert {\bar{x}}^\beta \rho \Vert_{L^1}) 
\left( \Vert \sqrt{\rho} v \Vert _{L^2} 
+ \Vert \nabla v \Vert _{L^2}\right).
\ea\ee
Moreover, multiplying $(\ref{ns})_1$ by $\left(1+|x|^2 \right)^{\frac{1}{2}}$, integrating by parts over $\rr_+$, and using (\ref{bkj01}), we obtain
\be\la{bkj57}\ba
\frac{d}{dt}\int \rho (1+|x|^2)^{\frac{1}{2}} dx
&\le \int |x| (1+|x|^2)^{-\frac{1}{2}} \rho |u| dx \\
&\le \left( \int \rho dx\right) ^{1/2}\left( \int \rho |u|^2 dx\right) ^{1/2} \\
&=\left( \int \rho_0 dx\right) ^{1/2}\left( \int \rho |u|^2 dx\right) ^{1/2} \\
&\le \left( \int {\bar{x}}^a \rho_0 dx\right) ^{1/2}\left( \int \rho |u|^2 dx\right) ^{1/2} \\
&\le C,
\ea\ee
which yields
\be\la{bkj58}\ba
\int \rho (1+|x|^2)^{\frac{1}{2}} dx 
& \le \int \rho_0 (1+|x|^2)^{\frac{1}{2}} dx + C t \\
& \le \int {\bar{x}}^a \rho_0 dx + C t \\
& \le C(1+t).
\ea\ee
Choosing $\beta$ sufficiently small in (\ref{bkj56}) and using (\ref{bkj58}), we arrive at (\ref{bkj501}) and complete the proof of Lemma \ref{bkjl5}.
\end{proof}

With Lemmas \ref{bkjl1}--\ref{bkjl5} at hand, we are in a position to prove Proposition \ref{pro1}.

\begin{proof}[Proof of Proposition \ref{pro1}]
First, it follows from (\ref{pro101}), (\ref{bkj501}), and (\ref{bkj23}) that for any $2 \le p < \infty$,
\be\la{pro10}\ba
\| \na G \|_{L^p} \le C \| \n \dot{u} \|_{L^p}
\le C(1+t)^4 \left( \| \sqrt{\n} \dot{u} \|_{L^2} + \| \na \dot{u} \|_{L^2} \right).
\ea\ee
We use (\ref{gw}) to rewrite $(\ref{ns})_1$ as
\be\ba\la{pro11}
\frac{d}{dt} \n = g(\n) + h'(t),
\ea\ee
where
\be\la{pro12}\ba
g(\n) = - \frac{1}{\nu} \n^{\ga+1}, \quad h(t) = - \frac{1}{\nu} \int_0^t \n G ds.
\ea\ee
Using (\ref{gn11}), (\ref{bkj021}), (\ref{bkj401}), (\ref{pro10}), and the Gagliardo-Nirenberg inequality, we derive that for any $0 \le t \le \si(T)$,
\be\la{pro13}\ba
|h(t)| & \le \frac{C}{\nu} \int_0^{\si(t)} \| G \|_{L^\infty} ds \\
& \le \frac{C}{\nu} \int_0^{\si(t)} \| G \|_{L^2}^{\frac38} \| \na G \|_{L^5}^{\frac58} ds \\
& \le C \nu^{-\frac58} \int_0^{\si(t)} \left( \si^2 \| \sqrt{\n} \dot u \|^{2}_{L^2} + \si^2 \| \na \dot u \|^{2}_{L^2} \right)^{\frac{5}{16}}
\si^{-\frac58} ds \\
& \le C \nu^{-\frac58} \left( \int_0^1 \si^{-\frac{10}{11}} ds \right)^{\frac{11}{16}} \\
& \le C \nu^{-\frac58},
\ea\ee
which together with (\ref{pro11}) yields
\be\la{pro14}\ba
\sup_{0 \le t \le \si(T)} \| \n \|_{L^\infty} \le \| \n_0 \|_{L^\infty} + C \nu^{-\frac58}.
\ea\ee
Moreover, from (\ref{bkj401}), (\ref{pro10}), and the Gagliardo-Nirenberg inequality, we conclude that
\be\la{pro16}\ba
\int_{\si(T)}^T \| G \|^4_{L^\infty} dt
& \le C \int_{\si(T)}^T \| G \|^{\frac{35}{9}}_{L^{72}} \| \na G \|^{\frac19}_{L^{72}} dt \\
& \le C \int_{\si(T)}^T \| G \|^{\frac{35}{9}}_{L^{72}}
t^{\frac49} \left( \| \sqrt{\n} \dot{u} \|_{L^2} + \| \na \dot{u} \|_{L^2} \right)^{\frac19} dt \\
& \le C \nu^{\frac{1}{18}} \int_{\si(T)}^T t^{-\frac{27}{18}}
\left( \| \sqrt{\n} \dot{u} \|_{L^2} + \| \na \dot{u} \|_{L^2} \right)^{\frac19} dt \\
& \le C \nu^{\frac{1}{18}},
\ea\ee
where in the third inequality we have used that for any $t \in [\si(T),T]$,
\be\nonumber\ba
\| G \|_{L^{72}} \le C \| G \|^{\frac{1}{36}}_{L^2} \| \na G \|^{\frac{35}{36}}_{L^2}
\le C \| G \|^{\frac{1}{36}}_{L^2} \| \sqrt{\n} \dot{u} \|^{\frac{35}{36}}_{L^2}
\le C \nu^{\frac{1}{72}} t^{-\frac{71}{72}},
\ea\ee
owing to (\ref{gn11}), (\ref{bkj24}), (\ref{bkj301}), and (\ref{bkj401}).

This, combined with Young's inequality, implies that for any $\si(T) \le t_1 \le t_2 \le T$,
\be\la{pro17}\ba
h(t_2) - h(t_1) & \le \frac{C}{\nu} \int_{t_1}^{t_2} \| G \|_{L^\infty} dt \\
& \le \frac{1}{\nu}(t_2 - t_1) + \frac{C}{\nu} \int_{\si(T)}^T \| G \|^4_{L^\infty} dt \\
& \le \frac{1}{\nu}(t_2 - t_1) + C \nu^{-\frac{17}{18}}.
\ea\ee
Choose $N_0$, $N_1$, and $\overline{\zeta}$ in Lemma \ref{zli} as follows:
\be\la{pro18}\ba
N_0 = C \nu^{-\frac{17}{18}}, \quad N_1 = \frac{1}{\nu}, \quad \overline{\zeta} = 1,
\ea\ee
which together with (\ref{pro12}) gives
\be\nonumber\ba
g(\zeta) = - \frac{1}{\nu} \zeta^{\ga+1} \le - N_1 = - \frac{1}{\nu} \quad \text{ for all } \zeta \ge 1.
\ea\ee
Combining (\ref{pro14}), (\ref{pro18}), and Lemma \ref{zli}, we arrive at
\be\la{pro19}\ba
\sup_{ \si(T) \le t \le T} \| \n \|_{L^\infty}
\le 1 + \| \n_0 \|_{L^\infty} + \mathbf{M_1} \nu^{-\frac58},
\ea\ee
where $\mathbf{M_1}$ is a positive constant depending only on
$\mu$, $\ga$, $a$, $\Vert {\bar{x}}^a \rho_0 \Vert_{L^1}$, $N_0$, $E_0$, $\| \n_0 \|_{L^\infty}$, and $\|\na u_0\|_{L^2}$,
but is independent of $T$ and $\nu$.

Finally, we define
\be\ba\la{pro110}
\nu_1 \triangleq \left( \frac{2 \mathbf{M_1}}{1 + \| \n_0 \|_{L^\infty}} \right)^{\frac{8}{5}},
\ea\ee
which implies (\ref{pro102}) provided $\nu \geq \nu_1$ and completes the proof of Proposition \ref{pro1}.
\end{proof}

\subsection{A Priori Estimates (II): Higher Order Estimates}

In this subsection, we assume that (\ref{pro101}) holds and derive the higher-order estimates needed to extend the local solution globally in time.
Let $(\n,u)$ be a strong solution to (\ref{ns})--(\ref{bjtj2}) on $\rr_+ \times (0,T]$ satisfying (\ref{pro101}).

\begin{lemma}\la{1hl1}
There exists a positive constant $C$ depending only on 
$T$, $q$, $\ga$, $\mu$, $\lam$, $E_0$, $N_0$, $\| {\bar{x}}^a \rho_0 \|_{L^1}$, $\|\n_0\|_{L^\infty}$, and $\| \na u_0 \|_{L^2}$ such that
\be\la{1h01}\ba
&\sup_{0\le t\le T} \left( \| \n \|_{H^1 \cap W^{1,q}} + \| \na u \|_{L^2} + t \| \na^2 u \|^2_{L^2} \right) \\
& + \int_0^T \left( \|\nabla^2 u\|^{2}_{L^2}+\|\nabla^2 u\|^{(q+1)/q}_{L^q}+t \|\nabla^2 u\|_{L^q}^2 \right) dt\le C.
\ea\ee
\end{lemma}
\begin{proof}
First, by (\ref{pro101}), (\ref{gw}), (\ref{bkj01}), and (\ref{bkj021}), we have
\be\la{1h11} \ba
\sup_{0\le t\le T} \left( \| \n \|_{L^\infty} + \| \na u \|_{L^2} \right) 
+ \int_0^T \left( \| \na u\|^2_{L^2}+ \| \sqrt{\n} \dot{u} \|^2_{L^2} \right)dt \le C.
\ea \ee
Multiplying (\ref{bkj417}) by $\si$ and using (\ref{1h11}), (\ref{bkj416}), (\ref{bkj418}), and Gr\"onwall's inequality, one obtains
\be\la{1h12} \ba
\sup_{0\le t\le T} \si \int\n|\dot u|^2dx+\int_0^{ T} \si \|\na\dot u\|^2_{L^2}dt\le C.
\ea\ee
For any $p \in [2,q]$, from $(\ref{ns})_1$, we deduce that $|\na \n|^p$ satisfies
\be\la{1h13}\ba
& (|\nabla\n|^p)_t + \text{div}(|\nabla\n|^p u)+ (p-1)|\nabla\n|^p\text{div}u  \\
&+ p|\nabla\n|^{p-2} \p_i\n \p_i u^j \p_j\n +
p \n |\nabla\n|^{p-2}\p_i\n  \p_i \text{div}u = 0.
\ea\ee
Integrating (\ref{1h13}) over $\rr_+$ yields
\be\la{1h14} \ba
\frac{d}{dt} \norm[L^p]{\nabla\n}  
&\le C \norm[L^{\infty}]{\nabla u}  \norm[L^p]{\nabla\n} +C\|\na^2 u\|_{L^p} \\ 
&\le C(1+\norm[L^{\infty}]{\nabla u} ) \norm[L^p]{\nabla\n}+C \|\n\dot u\|_{L^p},
\ea\ee
where we have used the following estimate:
\be\la{1h15}\ba
\|\na^2 u\|_{L^p}
&\le C(\|\na \div u\|_{L^p}+\|\na \o\|_{L^p}) \\
&\le C (\| \nabla G\|_{L^p}+ \|\nabla P \|_{L^p})+ C\|\na \o\|_{L^p} \\
&\le C\|\n\dot u\|_{L^p} + C  \|\nabla \n \|_{L^p},
\ea\ee
due to (\ref{dc1}) and (\ref{bkj23}).

It follows from (\ref{gn12}), (\ref{bkj23}), and (\ref{1h11}) that
\be\la{1h16}\ba 
& \| \div u \|_{L^\infty}+\| \o \|_{L^\infty} \\
& \le C \left( \| G \|_{L^\infty} + \| P-P(\tilde{\n}) \|_{L^\infty} \right)
+ \| \o \|_{L^\infty} \\
& \le C \left( 1 + \| G \|_{L^2}^{\frac{q-2}{2(q-1)}}
\| \na G \|_{L^q}^{\frac{q}{2(q-1)}}
+ \| \o \|_{L^2}^{\frac{q-2}{2(q-1)}} \| \na \o \|_{L^q}^{\frac{q}{2(q-1)}} \right) \\
& \le C + C \| \n \dot{u} \|_{L^q}^{\frac{q}{2(q-1)}},
\ea\ee
which together with Lemma \ref{bkm} and (\ref{1h15}) gives
\be\la{1h17}\ba
\|\na u\|_{L^\infty}
& \le C \left(\|{\rm div}u\|_{L^\infty } + \|\o\|_{L^\infty} \right)\log(e+\|\na^2 u\|_{L^q}) + C \|\na u\|_{L^2} + C \\
& \le C \left( 1 + \| \n \dot{u} \|_{L^q}^{\frac{q}{2(q-1)}} \right) \log(e+\|\rho \dot u\|_{L^q} + \|\na \rho\|_{L^q}) + C \\
& \le C \left(1+\|\n\dot u\|_{L^q} \right) \log(e+\|\na \rho\|_{L^q}).
\ea\ee
Choosing $p=q$ in (\ref{1h14}) and using (\ref{1h17}) lead to
\be\la{1h18}\ba
\frac{d}{dt} \log(e+ \|\na \rho\|_{L^q})
\le C\left(1+\|\n\dot u\|_{L^q} \right) \log(e+ \|\na \rho\|_{L^q}).
\ea\ee 
Furthermore, using (\ref{bkj501}) and H\"{o}lder's inequality, one derives
\be\la{1h19}\ba
\| \rho \dot u\|_{L^q} 
& \le C\| \rho \dot u\|_{L^2}^{2(q-1)/(q^2-2)}
\| \n \dot{u} \|_{L^{q^2}}^{q(q-2)/(q^2-2)} \\ 
& \le C\| \rho \dot u\|_{L^2}^{2(q-1)/(q^2-2)} 
\left( \| \sqrt{\n} \dot u\|_{L^2}+\| \na \dot u\|_{L^2} \right)^{q(q-2)/(q^2-2)} \\ 
& \le C\| \sqrt{\n}  \dot u\|_{L^2} + C\| \sqrt{\n} \dot u\|_{L^2}^{2(q-1)/(q^2-2)}
\|\na \dot u\|_{L^2}^{q(q-2)/(q^2-2)},
\ea\ee
which together with (\ref{1h12}) implies
\be\la{1h110} \ba
& \int_0^T \left( \|\rho \dot u\|^{1+1 /q}_{L^q}+t\| \n \dot u\|^2_{L^q} \right) dt \\
& \le C+C \int_0^T\left( \| \sqrt{\n}  \dot u\|_{L^2}^2 +  t\|\na \dot u\|_{L^2}^2+ 
t^{-(q^3-q^2-2q)/(q^3-q^2-2q+2)} \right)dt \\ 
& \le C.
\ea\ee
Applying Gr\"onwall's inequality to (\ref{1h18}), we obtain after using (\ref{1h110}) that
\be\la{1h111}\ba
\sup\limits_{0\le t\le T} \|\nabla
\rho\|_{L^q}\le  C.
\ea\ee
This, combined with (\ref{1h12}), (\ref{1h15}), (\ref{1h110}), and (\ref{1h111}), gives
\be\la{1h112}\ba 
\int_0^T \left( \|\nabla^2 u\|^{(q+1)/q}_{L^q}+t \|\nabla^2 u\|_{L^q}^2 \right) dt \le C.
\ea\ee

In addition, choosing $p=2$ in (\ref{1h14}) and using (\ref{1h11}), (\ref{1h12}), (\ref{1h110}), and Gr\"onwall's inequality, we arrive at
\be\la{1h113}\ba
\sup\limits_{0\le t\le T} \left( \|\nabla \rho\|_{L^2} + t \| \na^2 u \|^2_{L^2} \right)
+ \int_0^T \| \na^2 u \|^2_{L^2} dt
\le C.
\ea\ee
Combining (\ref{1h113}), (\ref{1h111}), (\ref{1h112}), and (\ref{1h11}) yields (\ref{1h01}), thereby completing the proof of Lemma \ref{1hl1}.
\end{proof}

\begin{lemma}\la{1hl2}
There exists a positive constant $C$ depending only on
$T$, $q$, $\ga$, $\mu$, $\lam$, $E_0$, $N_0$, $\|\n_0\|_{L^\infty}$, $\| \na u_0 \|_{L^2}$, $\| {\bar{x}}^a \rho_0 \|_{L^1}$, and $\|\na(\bar x^a\n_0)\|_{L^2\cap L^q}$ such that
\be\la{1h02}\ba
&\sup_{0\le t\le T} \|  \bar x^a\n \|_{L^1\cap H^1\cap W^{1,q}}  \le C.
\ea\ee
\end{lemma}
\begin{proof}
First, multiplying $(\ref{ns})_1$ by ${\bar{x}}^a$, integrating by parts over $\rr_+$, and using (\ref{bkj01}), we arrive at
\be\la{1h21}\ba
\frac{d}{dt} \int \n {\bar{x}}^a dx 
& \le C \int \n |u| {\bar{x}}^{a-1} \log^2(e+|x|^2) dx \\
& \le C \left( \int \n {\bar{x}}^{2a-2} \log^4(e+|x|^2) dx \right)^{\frac{1}{2}}
\left( \int \n |u|^2 dx \right)^{\frac{1}{2}} \\
& \le C \left( \int \n {\bar{x}}^a dx \right)^{\frac{1}{2}},
\ea\ee
which together with Gr\"onwall's inequality shows
\be\ba\la{1h22}
\sup_{0\leq t\leq T} \int \n {\bar{x}}^a dx \leq C.
\ea\ee

Moreover, for any $v \in D^1(\rr_+)$, using Lemma \ref{WPE} and extending $v$ evenly to $\rr$, we can obtain for $m \in [2,\infty)$ and $\theta \in (1+\frac{m}{2},\infty)$,
\be\la{bkjjq}\ba
\left(\int_{{\mathbb{R}^2_+}} \frac{|v|^m}{e+|x|^2}(\log (e+|x|^2))^{-\theta }dx \right)^{\frac1m}
\le C \| v \|_{L^2(B^+_1)} + C \| \nabla v \|_{L^2({\mathbb{R}^2_+ })},
\ea\ee
where $C$ is a generic positive constant.

By the Gagliardo-Nirenberg inequality, (\ref{bkjjq}), (\ref{bkj01}), and (\ref{1h11}), we derive that for any $\ve\in(0,1]$,
\be\la{1h23}\ba
\| u \bar{x}^{-\ep} \|_{L^\infty}
& \le C \| u \bar{x}^{-\ep} \|_{L^{q/\ep}} + C \| \na( u \bar{x}^{-\ep} ) \|_{L^q} \\
& \le C + C \| \na u \|_{L^q} + C \| u \bar{x}^{-1/2} \|_{L^q} \\
& \le C + C \| \na u \|_{L^q}.
\ea\ee
Let $ v\triangleq\n\bar x^a$.
Then from $(\ref{ns})_1$ we deduce that $v$ satisfies
\be\nonumber\ba
v_t+u\cdot\na v-a vu\cdot\na \log \bar x+v\div u=0,
\ea\ee
which together with (\ref{1h23}) and H\"older's inequality yields that for any $p\in [2,q]$,
\be\la{1h24}\ba 
\frac{d}{dt} \|\na v\|_{L^p}
& \le C(1+\|\na u\|_{L^\infty}+\|u\cdot \na \log \bar x\|_{L^\infty}) \|\na v\|_{L^p} \\
&\quad +C\|v\|_{L^\infty}\left( \||\na u||\na\log \bar x|\|_{L^p}+\||  u||\na^2\log \bar x|\|_{L^p}+\| \na^2 u \|_{L^p}\right)\\
& \le C(1 +\|\na u\|_{W^{1,q}})  \|\na v\|_{L^p} \\
& \quad + C \|v\|_{L^\infty} \left( \|\na u\|_{L^p} 
+ \| u \bar x^{-2/5} \|_{L^{4p}} \|\bar x^{-3/2}\|_{L^{4p/3}} + \|\na^2 u\|_{L^p}\right) \\
& \le C(1 +\|\na^2u\|_{L^p}+\|\na u\|_{W^{1,q}})(1+ \|\na v\|_{L^p}+\|\na v\|_{L^q}). \ea\ee
Choosing $p=q$ in (\ref{1h24}) and using (\ref{1h01}) and Gr\"onwall's inequality, one obtains
\be\la{1h25}\ba
\sup\limits_{0\le t\le T}\|\na (\n \bar x^a)\|_{L^q} \le C.
\ea\ee

Moreover, choosing $p=2$ in (\ref{1h24}), we obtain after using (\ref{1h01}), (\ref{1h25}), and Gr\"onwall's inequality that
\bnn
\sup\limits_{0\le t\le T}\|\na(\n \bar x^a)\|_{L^2 } \le C.
\enn
This, combined with (\ref{1h25}) and (\ref{1h22}), gives (\ref{1h02}) and completes the proof of Lemma \ref{1hl2}.
\end{proof}

\begin{lemma}\la{1hl3}
There exists a positive constant $C $ depending only on
$T$, $q$, $\ga$, $\mu$, $\lam$, $E_0$, $N_0$, $\|\n_0\|_{L^\infty}$, $\| \na u_0 \|_{L^2}$, $\| {\bar{x}}^a \rho_0 \|_{L^1}$, and $\|\na(\bar x^a\n_0)\|_{L^2\cap L^q}$ such that
\be\la{1h03}\ba
\sup_{0\leq t\leq T} t \| \sqrt{\n} u_t\|^2_{L^2} + \int_0^T t \|\na u_t\|_{L^2}^2 dt \le C.
\ea\ee
\end{lemma}
\begin{proof}
First, we conclude from (\ref{bkjjq}), (\ref{bkj501}) and (\ref{1h01}) that for any $\eta\in(0,1]$ and any $s>2$,
\be\la{h31}\ba
\|\n^\eta u \|_{L^{s/\eta}}+ \|u\bar x^{-\eta}\|_{L^{s/\eta}}\le C(\eta,s).
\ea\ee
Differentiating $(\ref{ns})_2$ with respect to $t$ leads to
\be\la{h32}\ba
& \n u_{tt}+\n u\cdot \na u_t-\mu\Delta u_t-( \mu+\lm)\na  \div u_t  \\ 
&= -\n_t(u_t+u\cdot\na u)-\n u_t\cdot\na u -\na P_t.
\ea\ee
Multiplying (\ref{h32}) by $u_t$, integrating by parts over $\rr_+$, and using $(\ref{ns})_1$, we derive
\be\la{h33}\ba
& \frac{1}{2}\frac{d}{dt} \int \n |u_t|^2dx+\int \left(\mu|\na u_t|^2+( \mu+\lm)(\div u_t)^2  \right)dx \\
& =-2\int \n u \cdot \na  u_t\cdot u_tdx  -\int \n u \cdot\na (u\cdot\na u\cdot u_t)dx\\
& \quad-\int \n u_t \cdot\na u \cdot  u_tdx
+\int P_{t}\div u_{t} dx \\
& \quad \triangleq I_1+I_2+I_3+I_4.
\ea\ee
It follows from (\ref{h31}), (\ref{1h01}), (\ref{gn11}), and H\"older's inequality that
\be\la{h34}\ba
I_1+I_2 
& \le C \int  \n |u| \left( |\na u_t| |u_t| + |\na u|^2 |u_t| + |u| |u_t| |\na^2 u| + |u| |\na u| |\na u_t| \right) dx \\
& \le C \| \sqrt{\n} u\|_{L^{6}}\| \sqrt{\n} u_{t} \|_{L^{2}}^{1/2} \| \sqrt{\n} u_{t}\|_{L^{6}}^{1/2}\left(\| \na u_{t}\|_{L^{2}}+\| \na u\|_{L^{4}}^{2} \right) \\
&\quad + C\|\n^{1/4} u \|_{L^{12}}^{2}\| \sqrt{\n} u_{t}\|_{L^{2}}^{1/2} \| \sqrt{\n} u_{t}\|_{L^{6}}^{1/2} \| \na^{2} u \|_{L^{2}} 
+ C \| \sqrt{\n} u\|_{L^{8}}^{2}\|\na u\|_{L^{4}} \| \na u_{t}\|_{L^{2}} \\
& \le C  \| \sqrt{\n} u_{t}\|_{L^{2}}^{1/2}\left( \| \sqrt{\n} u_{t}\|_{L^{2}} +\| \na u_{t}\|_{L^{2}}\right)^{1/2}\left( \| \na u_{t}\|_{L^{2}} +  \| \na^{2} u \|_{L^{2}}+1\right) \\
& \quad + C \| \na u_{t}\|_{L^{2}} \left( 1 +  \| \na^{2} u \|_{L^{2}} \right) \\
&\le \de \| \na u_{t}\|_{L^{2}}^{2}+C(\de)  \left(\| \na^{2} u \|_{L^{2}}^{2} + \|\sqrt{\n} u_{t}\|_{L^{2}}^{2}+1\right),
\ea\ee
where in the third inequality, we have used the following estimate:
\be\la{h35}\ba
\| \sqrt{\n} u_t\|_{L^6} \le C \| \sqrt{\n} u_t\|_{L^2}+C \|\na u_t\|_{L^2},
\ea\ee
owing to (\ref{bkj501}) and (\ref{bkj53}).

Using (\ref{h35}), (\ref{1h01}) and Young's inequality, we arrive at
\be\la{h36}\ba
I_3 + I_4 & \le C \int \n |u_t|^{2}|\na u | + |P_t| |\na u_t| dx \\
& \le C \| \na u\|_{L^{2}} \| \sqrt{\n} u_{t}\|_{L^{6}}^{3/2}
\| \sqrt{\n} u_{t}\|_{L^{2}}^{1/2} + \| P_t \|_{L^2} \| \na u_t \|_{L^2} \\
& \le \de \| \na u_{t}\|_{L^{2}}^{2}+C(\de) \left(\| \na^{2} u \|_{L^{2}}^{2} +
\| \sqrt{\n} u_{t}\|_{L^{2}}^{2}+1\right),
\ea\ee
where in the last inequality we have used the following fact:
\be\la{h37}\ba 
\|P_t\|_{L^2 } &\le C\|\bar x^{-a} u\|_{L^{2q/(q-2)}}\|\n\|_{L^\infty}^{\ga-1}\|\bar x^a \na \n\|_{  L^q}+C\|\na u\|_{L^2 }\le C,
\ea\ee
due to $(\ref{ns})_1$, (\ref{h31}) and (\ref{1h01}).

In addition, by (\ref{gn11}), (\ref{h31}), and H\"older's inequality, one has
\be\la{h38}\ba
\| \sqrt{\n} u_t \|^2_{L^2} 
& \le C \left( \| \sqrt{\n} \dot{u} \|^2_{L^2}
+ \| \sqrt{\n} u \cdot \na u \|^2_{L^2} \right) \\
& \le C \left( \| \sqrt{\n} \dot{u} \|^2_{L^2}
+ \| \sqrt{\n} u \|^2_{L^6} \| \na u \|^2_{L^3} \right) \\
& \le C \left( \| \sqrt{\n} \dot{u} \|^2_{L^2}
+ \| \na^2 u \|^2_{L^2} \right).
\ea\ee

Putting (\ref{h34}) and (\ref{h36}) into (\ref{h33}) and choosing $\de$ sufficiently small, we obtain after using (\ref{h38}) that
\be\la{h39}\ba
\frac{d}{dt} \int \n |u_t|^2dx+\mu\int   |\na u_t|^2 dx
\le C \left(\| \na^{2} u \|_{L^{2}}^{2} + \| \sqrt{\n} \dot{u} \|^2_{L^2} + 1 \right).
\ea\ee
Multiplying (\ref{h39}) by $t$ yields
\be\la{h310}\ba
\frac{d}{dt} \left( t \int \n |u_t|^2 dx \right) + \mu t \int |\na u_t|^2 dx
\le C \left(\| \na^{2} u \|_{L^{2}}^{2} + \| \sqrt{\n} \dot{u} \|^2_{L^2} + 1 \right).
\ea\ee
Integrating (\ref{h310}) over $(0,T)$ and using (\ref{1h01}) and (\ref{1h11}), we get (\ref{1h03}) and complete the proof of Lemma \ref{1hl3}.
\end{proof}

\section{A Priori Estimates for Non-vacuum Far-Field Density}

In this section, for $\tilde{\n}>0$, we will establish some necessary a priori bounds for local strong solutions $(\n,u)$ to the problem (\ref{ns})--(\ref{bjtj2}), whose existence is guaranteed by Lemma \ref{lct}.
Thus, let $T>0$ be a fixed time and $(\n,u)$ be a strong solution to (\ref{ns})--(\ref{bjtj2}) on $\rr_+ \times (0,T]$ with initial data $(\n_0,u_0)$ satisfying (\ref{wsol1}) and $(\ref{wsol01})_2$.

\subsection{A Priori Estimates (I): Lower Order Estimates}

Define
\be\ba\nonumber
E(T) \triangleq \sup_{0 \le t \le T} \si \left( \| \na u \|^2_{L^2} + \nu \| \div u \|^2_{L^2} \right) + \int_0^T \si \| \sqrt{\n} \dot{u} \|^2_{L^2} dt.
\ea\ee

In this subsection, we derive the following key a priori estimates on $(\n,u)$.
\begin{proposition}\label{pro2}
There are two generic positive constants $\nu_2$ and $\mathbf{C_1}$ depending only on
$A$, $\ga$, $\mu$, $E_0$, $\| \n_0 \|_{L^\infty}$, $\| \na u_0 \|_{L^2}$, and $\tilde{\n}$ such that if $(\rho,u)$ is a strong solution to \eqref{ns}--\eqref{bjtj2} on $\rr_+ \times(0,T]$ satisfying
\be\la{pro201}\ba
\sup_{0\leq t\leq T} \|\n\|_{L^\infty} \leq 2 \left( 1 + \tilde{\n} + \| \n_0 \|_{L^\infty} \right), \quad
E(T) \le 2 \mathbf{C_1},
\ea\ee
the following estimates hold:
\be\la{pro202}\ba
\sup_{0\leq t\leq T} \|\n\|_{L^\infty} \leq \frac{3}{2} \left( 1 + \tilde{\n} + \| \n_0 \|_{L^\infty} \right), \quad
E(T) \le \mathbf{C_1},
\ea\ee
provided that $\nu \ge \nu_2$.
\end{proposition}
\begin{proof}
Proposition \ref{pro2} is an easy consequence of the following Lemmas \ref{2bkjl3} and \ref{2bkjl5}, with $\nu_2$ as in (\ref{2bkj511}) and $\mathbf{C_1}$ as in (\ref{2bkj324}).
\end{proof}

We begin with the following standard energy estimate.
\begin{lemma}\la{2bkjl1}
There exists a positive constant $C$ depending only on
$\ga$, $\mu$, and $E_0$ such that
\be\ba\la{2bkj01}
\sup_{0\leq t\leq T} \int \left( \frac{1}{2}\rho |u|^2 + H(\n) \right) dx
+ \int_0^T  \left(\|\na u\|^2_{L^2} + (\mu+\lam) \|\div u\|^2_{L^2} \right) dt \leq C.
\ea\ee
\end{lemma}
\begin{proof}
Multiplying $(\ref{ns})_2$ by $u$, integrating the resulting equation by parts over $\rr_+$, and using $(\ref{ns})_1$ and (\ref{dc1}), we arrive at (\ref{2bkj01}).
\end{proof}

\begin{lemma}\la{2bkjl2}
Let $(\n,u)$ be a strong solution to \eqref{ns}--\eqref{bjtj2} satisfying \eqref{pro201}.
Then there exists a positive constant $\tilde{C}$ depending only on
$A$, $\ga$, $\mu$, $E_0$, $\| \n_0 \|_{L^\infty}$, and $\tilde{\n}$ such that
\be\la{2bkj201}\ba
\sup_{0 \le t \le \si(T)} \left( \| \na u \|^2_{L^2} + \nu \| \div u \|^2_{L^2} \right)
+ \int_0^{\si(T)} \| \sqrt{\n} \dot{u}\|^2_{L^2} dt
\le \tilde{C} \left( 1 + \| \na u_0 \|^2_{L^2} + \nu \| \div u_0 \|^2_{L^2} \right),
\ea\ee
and
\be\la{2bkj202}\ba
\sup_{0 \le t \le \si(T)} t \left( \| \na u \|^2_{L^2} + \nu \| \div u \|^2_{L^2} \right)
+ \int_0^{\si(T)} t \| \sqrt{\n} \dot{u}\|^2_{L^2} dt \le \tilde{C}.
\ea\ee
\end{lemma}
\begin{proof}
First, from $(\ref{ns})_2$ and (\ref{gw}), we have
\be\la{2bkj21}\ba
\n \dot{u} = \na G + \mu \na^\bot \o,
\ea\ee
which together with the boundary conditions (\ref{bjtj1}) implies that $\o$ satisfies
\be\la{2bkj22}\ba
\begin{cases}
\mu \Delta \o = \na^\bot \cdot (\n \dot{u}) \ & \text{ in } \rr_+, \\
\o = -A u \cdot n^\bot \ & \text{ on } \p \rr_+.
\end{cases}
\ea\ee
The standard elliptic estimate (see \cite{GT}) combined with (\ref{2bkj21}) yields for any $2 \le p < \infty$,
\be\la{2bkj23}\ba
\| \na G \|_{L^p} + \| \na \o \|_{L^p}
& \le C \left( \| \n \dot{u} \|_{L^p} + \| \na (A u \cdot n^\bot) \|_{L^p} \right) \\
& \le C \left( \| \n \dot{u} \|_{L^p} + \| \na u \|_{L^p} \right).
\ea\ee
By (\ref{pro201}), we arrive at
\be\la{2bkj24}\ba
\| \na G \|_{L^2} + \| \na \o \|_{L^2}
\le C \left( \| \sqrt{\n} \dot{u} \|_{L^2} + \| \na u \|_{L^2} \right).
\ea\ee
Multiplying (\ref{2bkj21}) by $2 \dot{u}$, integrating by parts over $\rr_+$, and using (\ref{bkj214}) and (\ref{bkj215}), we obtain
\be\la{2bkj25}\ba
& \frac{d}{dt} \int \left(\mu \o^2 + \frac{G^2}{\nu}\right)dx + 2\| \sqrt{\n} \dot{u}\|^2_{L^2}\\
& = -\mu \int \o^2\div udx - 4\int G\nabla u^1\cdot\nabla^{\perp}u^2dx- 2\int G(\div u)^2dx\\
&\quad +\frac{1}{\nu} \int G^2\div udx +\frac{2\ga}{\nu} \int P G\div udx + 2 \mu \int_{\p \rr_+} \o (\dot{u} \cdot n^\bot) ds = \sum_{i=1}^6 J_i.
\ea\ee
From (\ref{gn11}), (\ref{2bkj24}), and Young's inequality, we deduce that
\be\la{2bkj27}\ba
|J_1| + |J_2| & \le C \| \o\|^2_{L^4} \| \div u\|_{L^2} + C \| \na G\|_{L^2} \| \na u\|^2_{L^2} \\
& \le C \| \o\|_{L^2} \| \na \o\|_{L^2} \| \div u\|_{L^2}
+ C \left( \| \sqrt{\n} \dot{u} \|_{L^2} + \| \na u\|_{L^2} \right) \| \na u\|^2_{L^2} \\
& \le C \left( \| \sqrt{\n} \dot{u} \|_{L^2} + \| \na u\|_{L^2} \right) \| \na u\|^2_{L^2} \\
& \le \frac{1}{16} \| \sqrt{\n} \dot{u}\|^2_{L^2} + C \| \na u\|^2_{L^2} + C \| \na u\|^4_{L^2}.
\ea\ee
By (\ref{gn11}), (\ref{bkj24}), (\ref{gw}), (\ref{bkj212}), and H\"older's inequality, we arrive at
\be\la{2bkj28}\ba
\sum_{i=3}^5 J_i & \le \frac{C}{\nu} \int G^2 |\div u| dx + \frac{C}{\nu}\int P |G| |\div u| dx \\
& \le \frac{C}{\nu} \| G \|^2_{L^4} \| \div u \|_{L^2} + \frac{C}{\nu} \| G \|_{L^2} \| \div u \|_{L^2} \\
& \le \frac{C}{\nu} \| G \|_{L^2} \| \na G \|_{L^2} \| \div u \|_{L^2}
+ \frac{C}{\nu} \| G \|_{L^2} \| \div u \|_{L^2} \\
& \le \frac{C}{\nu} \| G \|_{L^2} \left( \| \sqrt{\n} \dot{u} \|_{L^2} + \| \na u\|_{L^2} \right) \| \div u \|_{L^2}
+ \frac{C}{\nu} \| G \|_{L^2} \| \div u \|_{L^2} \\
& \le \frac{1}{16} \| \sqrt{\n} \dot{u}\|^2_{L^2}
+ \frac{C}{\nu} \| G \|^2_{L^2} \| \na u\|^2_{L^2}+C\| \na u\|^2_{L^2} + C.
\ea\ee
The boundary conditions (\ref{bjtj1}) imply that
\be\la{2bkj29}\ba
J_6 &= 2\mu \int_{\p \rr_+ }\o ({\dot{u}}\cdot n^\bot )\ ds \\
&=-2\mu \int_{\p \rr_+ } A(u\cdot n^\bot )\cdot (u\cdot n^\bot )_t \,ds -2\mu \int_{\p \rr_+ } A(u\cdot n^\bot )(u\cdot \nabla u\cdot n^\bot) \,ds\\
&=-\mu \frac{d}{dt} \int_{\p \rr_+ } A(u\cdot n^\bot )^2 \,ds
-2\mu \int_{\p \rr_+ } A (u^1)^2 \p_1 u^1 \,ds \\
& = -\mu \frac{d}{dt} \int_{\p \rr_+ } A |u| ^2 \,ds.
\ea\ee
Moreover, it follows from (\ref{dc1}) and (\ref{gw}) that
\be\ba\la{2bkj210}
\| \na u\|^2_{L^2} & \le C \left( \| \div u\|^2_{L^2}+\| \o \|^2_{L^2} \right) \\
& \le \frac{C}{\nu^2} \left( \| G \|^2_{L^2}+\| P-P( \tilde{\n}) \|^2_{L^2} \right) + C \| \o \|^2_{L^2} \\
& \le C \left( \mu \| \o \|^2_{L^2} + \frac{1}{\nu} \| G \|^2_{L^2} \right) + C,
\ea\ee
where in the second inequality we have used the following estimate:
\be\la{2bkj211}\ba
\| P-P( \tilde{\n}) \|^2_{L^2} \le C \| \n - \tilde{\n} \|^2_{L^2} \le C,
\ea\ee
due to (\ref{pro201}), (\ref{2bkj01}), and (\ref{qkjsn}).

Substituting (\ref{2bkj27})--(\ref{2bkj29}) into (\ref{2bkj25}) and using (\ref{2bkj210}), we obtain
\be\la{2bkj212}\ba
& \frac{d}{dt} \left( \mu \| \o \|^2_{L^2} + \frac{1}{\nu} \| G \|^2_{L^2} + \int_{\p \rr_+ } A |u| ^2 \,ds \right) + \| \sqrt{\n} \dot{u}\|^2_{L^2} \\
& \le C \left( \mu \| \o \|^2_{L^2} + \frac{1}{\nu} \| G \|^2_{L^2} \right) \| \na u\|^2_{L^2} + C\| \na u\|^2_{L^2} + C.
\ea\ee
In addition, multiplying (\ref{2bkj212}) by $t$ and using (\ref{gw}) lead to
\be\la{2bkj213}\ba
& \frac{d}{dt} \left( t \left( \mu \| \o \|^2_{L^2} + \frac{1}{\nu} \| G \|^2_{L^2} + \int_{\p \rr_+ } A |u| ^2 \,ds \right) \right) + t \| \sqrt{\n} \dot{u}\|^2_{L^2} \\
& \le C t \left( \mu \| \o \|^2_{L^2} + \frac{1}{\nu} \| G \|^2_{L^2} \right) \| \na u\|^2_{L^2} + C t \left( \| \na u\|^2_{L^2} + 1 \right) \\
& \quad + C \left( \| \na u\|^2_{L^2} + \nu \| \div u \|^2_{L^2} + 1 \right).
\ea\ee
Applying Gr\"onwall's inequality to (\ref{2bkj212}) and (\ref{2bkj213}) over $(0,\si(T))$ and using (\ref{2bkj01}), we arrive at (\ref{2bkj201}) and (\ref{2bkj202}).
This finishes the proof of Lemma \ref{2bkjl2}.
\end{proof}

\begin{lemma}\la{2bkjl3}
There exist two positive constants $\hat{\nu}_2$ and $\mathbf{C_1}$ depending only on
$A$, $\ga$, $\mu$, $E_0$, $\| \n_0 \|_{L^\infty}$, and $\tilde{\n}$ such that if $(\n,u)$ is a strong solution to \eqref{ns}--\eqref{bjtj2} satisfying \eqref{pro201}, then
\be\la{2bkj301}\ba
E(T) \le \mathbf{C_1},
\ea\ee
provided $\nu \ge \hat{\nu}_2$.
\end{lemma}
\begin{proof}
First, if $T \le 1$, we set $\mathbf{C_1} = \tilde{C}$ and Lemma \ref{2bkjl2} directly gives (\ref{2bkj301}).
Next, we assume that $T>1$.

By (\ref{2bkj202}), we have
\be\la{2bkj31}\ba
\sup_{0 \le t \le 1} t \left( \| \na u \|^2_{L^2} + \nu \| \div u \|^2_{L^2} \right)
+ \int_0^1 t \| \sqrt{\n} \dot{u}\|^2_{L^2} dt \le C.
\ea\ee
We rewrite $(\ref{ns})_2$ as
\be\la{2bkj32}\ba
\n \dot{u} - \mu \na^\bot \o - \nu \na \div u + \na(P - P(\tilde{\n})) = 0.
\ea\ee
Multiplying $(\ref{ns})_2$ by $\dot{u}$ leads to
\be\la{2bkj33}\ba
\int \n |\dot{u}|^2 dx = - \int \dot{u} \cdot \na (P-P(\tilde{\n})) dx + \mu \int \na^\bot \o \cdot \dot{u} dx + (\mu+\lam) \int \na \div u \cdot \dot{u} dx.
\ea\ee
From $(\ref{ns})_1$, we conclude that $P-P(\tilde{\n})$ satisfies
\be\la{2bkj34}\ba
(P-P(\tilde{\n}))_t + u \cdot \na (P-P(\tilde{\n})) + \ga (P-P(\tilde{\n})) \div u + \ga P(\tilde{\n}) \div u = 0.
\ea\ee
Integration by parts along with (\ref{2bkj34}) and (\ref{bkjbjds}) yields
\be\la{2bkj35}\ba
- \int \dot{u} \cdot \na (P-P(\tilde{\n})) dx
& = \int \left( (\div u)_t (P-P(\tilde{\n})) - (u \cdot \na u) \cdot \na (P-P(\tilde{\n})) \right) dx \\
& = \frac{d}{dt} \left( \int \div u (P-P(\tilde{\n})) dx \right)
+ \int u \cdot \na (P-P(\tilde{\n})) \div u dx \\
& \quad + \int \ga P (\div u)^2 dx
- \int (u \cdot \na u) \cdot \na (P-P(\tilde{\n})) dx \\
& = \frac{d}{dt} \left( \int \div u (P-P(\tilde{\n})) dx \right)
- \int \left( (P-P(\tilde{\n})) - \ga P \right) (\div u)^2 dx \\
& \quad + \int \p_i u^j \p_j u^i (P-P(\tilde{\n})) dx \\
& \le \frac{d}{dt} \left( \int \div u (P-P(\tilde{\n})) dx \right) + C \| \na u \|^2_{L^2}.
\ea\ee
Integrating by parts and using (\ref{2bkj29}), we derive
\be\la{2bkj36}\ba
\mu \int \na^\bot \o \cdot \dot{u} dx
& = \mu \int_{\p \rr_+} \o (\dot{u} \cdot n^\bot) ds
- \mu \int \o (\na^\bot \cdot \dot{u}) dx \\
& = - \frac{\mu}{2} \frac{d}{dt} \int_{\p \rr_+} A |u|^2 ds
- \mu \int \o \left( \o_t + \na^\bot_j u^i \p_i u^j + u \cdot \na \o \right) dx \\
& = - \frac{\mu}{2} \frac{d}{dt} \int_{\p \rr_+} A |u|^2 ds - \frac{\mu}{2} \frac{d}{dt} \int |\o|^2 dx
- \mu \int \o \left( \na^\bot_j u^i \p_i u^j - \frac{1}{2} \o \div u \right) dx \\
& \le - \frac{\mu}{2} \frac{d}{dt} \int_{\p \rr_+} A |u|^2 ds - \frac{\mu}{2} \frac{d}{dt} \int |\o|^2 dx + C \| \na u \|^3_{L^3}.
\ea\ee
Integration by parts also gives
\be\la{2bkj37}\ba
& (\mu+\lam) \int \na \div u \cdot \dot{u} dx \\
& = - \frac{(\mu+\lam)}{2} \left( \| \div u \|^2_{L^2} \right)_t
- (\mu+\lam) \int \div u \div (u \cdot \na u) dx \\
& = -\frac{(\mu+\lam)}{2} \left( \| \div u \|^2_{L^2} \right)_t
- (\mu+\lam) \int \left( \div u \p_i u^j \p_j u^i + \div u u \cdot \na \div u \right) dx \\
& = -\frac{(\mu+\lam)}{2} \left( \| \div u \|^2_{L^2} \right)_t
- (\mu+\lam) \int \div u \p_i u^j \p_j u^i dx
+ \frac{(\mu+\lam)}{2} \int (\div u)^3 dx \\
& = -\frac{(\mu+\lam)}{2} \left( \| \div u \|^2_{L^2} \right)_t
- \frac{(\mu+\lam)}{2} \int (\div u)^3 dx
- 2 (\mu+\lam) \int \div u \na u^1 \cdot \na^\bot u^2 dx,
\ea\ee
where in the last equality we have used the identity
\be\la{2bkj38}\ba
\p_i u^j \p_j u^i = (\div u)^2 + 2 \na u^1 \cdot \na^\bot u^2.
\ea\ee
Putting (\ref{2bkj35})--(\ref{2bkj37}) into (\ref{2bkj33}) gives
\be\la{2bkj39}\ba
\frac{d}{dt} B_1(t) + \int \n |\dot{u}|^2 dx
& \le C \| \na u \|^2_{L^2} + C \| \na u \|^3_{L^3} - \frac{(\mu+\lam)}{2} \int (\div u)^3 dx \\
& \quad - 2(\mu+\lam) \int \div u \na u^1 \cdot \na^\bot u^2 dx,
\ea\ee
where
\be\la{b1}\ba
B_1(t) \triangleq \frac{\mu}{2} \| \o \|^2_{L^2} + \frac{(\mu+\lam)}{2} \| \div u \|^2_{L^2} + \frac{\mu}{2} \int_{\p \rr_+} A |u|^2 ds - \int \div u (P-P(\tilde{\n})) dx.
\ea\ee
It follows from (\ref{gn11}), (\ref{dc1}), (\ref{gw}), (\ref{2bkj24}), and Young's inequality that
\be\la{2bkj310}\ba
\| \na u \|^3_{L^3} & \le C \left( \| \div u \|^3_{L^3} + \| \o \|^3_{L^3} \right) \\
& \le C \left( \| \div u \|^2_{L^2} + \| \div u \|^4_{L^4} \right) + C \| \o \|^2_{L^2} \| \na \o \|_{L^2} \\
& \le C \| \na u \|^2_{L^2} + \frac{C}{\nu^4} \left( \| G \|^4_{L^4} + \| P - P(\tilde{\n}) \|^4_{L^4} \right) + C \| \o \|^2_{L^2} \left( \| \sqrt{\n} \dot{u} \|_{L^2} + \| \na u \|_{L^2} \right) \\
& \le \ep \| \sqrt{\n} \dot{u} \|^2_{L^2} + \frac{C}{\nu^2} \left( 1 + \| \na u \|^2_{L^2} \right) \| \sqrt{\n} \dot{u} \|^2_{L^2} + \frac{C}{\nu^4} \| P - P(\tilde{\n}) \|^4_{L^4} \\
& \quad + C \left( \| \na u \|^2_{L^2} + \| \na u \|^4_{L^2} \right),
\ea\ee
where in the last inequality we have used the following estimate:
\be\la{2bkj311}\ba
\| G \|^4_{L^4}
& \le C \| G \|^2_{L^2}\| \na G \|^2_{L^2} \\
& \le C \nu^2 \left( \| \div u \|^2_{L^2} + \| P-P(\tilde{\n}) \|^2_{L^2} \right) \left( \| \sqrt{\n} \dot{u} \|^2_{L^2} + \| \na u \|^2_{L^2} \right) \\
& \le C \nu^2 \left( 1 + \| \div u \|^2_{L^2} \right) \| \sqrt{\n} \dot{u} \|^2_{L^2}
+ C \nu^2 \left( \| \na u \|^2_{L^2} + \| \na u \|^4_{L^2} \right),
\ea\ee
due to (\ref{gn11}), (\ref{gw}), (\ref{2bkj24}), and (\ref{2bkj211}).

Using (\ref{gw}), (\ref{2bkj311}), and Young's inequality, we have
\be\la{2bkj312}\ba
- \frac{(\mu+\lam)}{2} \int (\div u)^3 dx
& \le \frac{C}{\nu} \int \left( |G|^2 + | P-P(\tilde{\n}) |^2 \right) |\div u| dx \\
& \le \frac{C}{\nu^3} \left( \| G \|^4_{L^4} + \| P-P(\tilde{\n}) \|^4_{L^4} \right) + C \nu \| \div u \|^2_{L^2} \\
& \le \frac{C}{\nu} \left( 1 + \| \na u \|^2_{L^2} \right) \| \sqrt{\n} \dot{u} \|^2_{L^2} + \frac{C}{\nu^3} \| P-P(\tilde{\n}) \|^4_{L^4} + C \nu \| \div u \|^2_{L^2} \\
& \quad + C \left( \| \na u \|^2_{L^2} + \| \na u \|^4_{L^2} \right).
\ea\ee
By virtue of (\ref{gw}), (\ref{2bkj24}), and Young's inequality, it holds that
\be\la{2bkj313}\ba
- 2(\mu+\lam) \int \div u \na u^1 \cdot \na^\bot u^2 dx
& = - \frac{2(\mu+\lam)}{\nu} \int (G + (P-P(\tilde{\n}))) \na u^1 \cdot \na^\bot u^2 dx \\
& \le C \| \na u \|^2_{L^2} + C \| \na G \|_{L^2} \| \na u \|^2_{L^2} \\
& \le C \| \na u \|^2_{L^2} + C ( \| \sqrt{\n} \dot{u} \|_{L^2} + \| \na u \|_{L^2} ) \| \na u \|^2_{L^2} \\
& \le \frac{1}{4} \| \sqrt{\n} \dot{u} \|^2_{L^2} + C \left( \| \na u \|^2_{L^2} + \| \na u \|^4_{L^2} \right).
\ea\ee
Substituting (\ref{2bkj310}), (\ref{2bkj312}), and (\ref{2bkj313}) into (\ref{2bkj39}) and choosing $\ep$ sufficiently small, we arrive at
\be\la{2bkj314}\ba
\frac{d}{dt} B_1(t) + \frac{1}{2} \int \n |\dot{u}|^2 dx
& \le \frac{C}{\nu} \left( 1 + \| \na u \|^2_{L^2} \right) \| \sqrt{\n} \dot{u} \|^2_{L^2}
+ \frac{C_1}{\nu^3} \| P-P(\tilde{\n}) \|^4_{L^4} \\
& \quad + C \| \na u \|^4_{L^2} + C \left( \| \na u \|^2_{L^2} + \nu \| \div u \|^2_{L^2} \right).
\ea\ee
Moreover, multiplying (\ref{2bkj34}) by $3 (P-P(\tilde{\n}))^2$ and using (\ref{gw}), we derive
\be\la{2bkj315}\ba
\frac{3\ga-1}{\nu} \| P-P(\tilde{\n}) \|^4_{L^4}
& = -\left( \int (P-P(\tilde{\n}))^3 dx \right)_t
- \frac{3\ga-1}{\nu} \int (P-P(\tilde{\n}))^3 G dx \\
& \quad - 3\ga P(\tilde{\n}) \int (P-P(\tilde{\n}))^2 \div u dx \\
& \le -\left( \int (P-P(\tilde{\n}))^3 dx \right)_t
+ \frac{\ga}{\nu} \| P-P(\tilde{\n}) \|^4_{L^4} + \frac{C}{\nu} \| G \|^4_{L^4} + C \nu \| \div u \|^2_{L^2},
\ea\ee
which gives
\be\la{2bkj316}\ba
\frac{1}{\nu} \| P-P(\tilde{\n}) \|^4_{L^4}
& \le -\left( \int (P-P(\tilde{\n}))^3 dx \right)_t
+ \frac{C}{\nu} \| G \|^4_{L^4} + C \nu \| \div u \|^2_{L^2}.
\ea\ee
Multiplying both sides of (\ref{2bkj316}) by $\frac{C_1}{\nu^2}$ and adding the result to (\ref{2bkj314}), we obtain after using (\ref{2bkj311}) that
\be\la{2bkj317}\ba
\frac{d}{dt} B_2(t) + \frac{1}{2} \int \n |\dot{u}|^2 dx
& \le \frac{C}{\nu} \left( 1 + \| \na u \|^2_{L^2} \right) \| \sqrt{\n} \dot{u} \|^2_{L^2} + C \| \na u \|^4_{L^2} \\
& \quad + C \left( \| \na u \|^2_{L^2} + \nu \| \div u \|^2_{L^2} \right),
\ea\ee
where
\be\la{2bkj318}\ba
B_2(t) & \triangleq \frac{\mu}{2} \| \o \|^2_{L^2} + \frac{(\mu+\lam)}{2} \| \div u \|^2_{L^2} + \frac{\mu}{2} \int_{\p \rr_+} A |u|^2 ds \\
& \quad - \int \div u (P-P(\tilde{\n})) dx + \frac{C_1}{\nu^2} \int (P-P(\tilde{\n}))^3 dx.
\ea\ee
For any $t \in (1,T)$, integrating (\ref{2bkj317}) over $(1,t)$ and using (\ref{2bkj01}) yield
\be\la{2bkj319}\ba
& \frac{\mu}{2} \| \o(t) \|^2_{L^2} + \frac{(\mu+\lam)}{2} \| \div u(t) \|^2_{L^2} - \int (\div u (P-P(\tilde{\n}))) (t) dx
+ \frac{1}{2} \int_{1}^t \| \sqrt{\n} \dot{u} \|^2_{L^2} ds \\
& \le B_2(1) + \frac{C}{\nu} \left( 1 + E(T) \right) E(T)
+ C \int_{1}^t \| \na u(\cdot,s) \|^4_{L^2} ds + C.
\ea\ee
It follows from (\ref{pro201}), (\ref{2bkj211}), (\ref{2bkj31}), and H\"older's inequality that
\be\la{2bkj320}\ba
B_2(1) & = \frac{\mu}{2} \| \o(\cdot,1) \|^2_{L^2} + \frac{(\mu+\lam)}{2} \| \div u(\cdot,1) \|^2_{L^2} + \frac{\mu}{2} \int_{\p \rr_+} A |u(x,1)|^2 ds \\
& \quad - \int (\div u (P-P(\tilde{\n})))(x,1) dx
+ \frac{C_1}{\nu^2} \int (P-P(\tilde{\n}))^3 (x,1) dx \\
& \le C \left( \| \na u(\cdot,1) \|^2_{L^2} + \nu \| \div u(\cdot,1) \|^2_{L^2} \right) + C \| (P-P(\tilde{\n}))(\cdot,1) \|^2_{L^2}
\le C,
\ea\ee
and
\be\la{2bkj321}\ba
\int (\div u (P-P(\tilde{\n}))) (t) dx
\le \| \div u(\cdot,t) \|^2_{L^2} + C \| (P-P(\tilde{\n}))(\cdot,t) \|^2_{L^2}
\le \frac{1}{\nu} E(T) + C.
\ea\ee
The combination of (\ref{2bkj319}), (\ref{2bkj320}), and (\ref{2bkj321}) gives
\be\la{2bkj322}\ba
& \| \na u(t) \|^2_{L^2} + \nu \| \div u(t) \|^2_{L^2} + \int_1^t \| \sqrt{\n} \dot{u} \|^2_{L^2} ds \\
& \le C_2 + \frac{C_3}{\nu} \left( 1 + E(T) \right) E(T)
+ C_4 \int_1^t \| \na u(\cdot,s) \|^4_{L^2} ds.
\ea\ee
Applying Gr\"onwall's inequality to (\ref{2bkj322}), we derive for any $t \in (1,T)$,
\be\la{2bkj323}\ba
& \| \na u(t) \|^2_{L^2} + \nu \| \div u(t) \|^2_{L^2} + \int_1^t \| \sqrt{\n} \dot{u} \|^2_{L^2} ds \\
& \le \left( C_2 + \frac{C_3}{\nu} \left( 1 + E(T) \right) E(T) \right)
\exp{ \left( C_4 \int_1^t \| \na u(\cdot,s) \|^2_{L^2} ds \right) } \\
& \le \left( C_2 + \frac{C_3}{\nu} \left( 1 + E(T) \right) E(T) \right) C_5.
\ea\ee
Setting
\be\la{2bkj324}\ba
\mathbf{C_1} \triangleq C_2 C_5 + \tilde{C},
\ea\ee
and
\be\la{2bkj325}\ba
\hat{\nu}_2 \triangleq 2 \tilde{C}^{-1} C_3 C_5 \left( C_2 C_5 + \tilde{C} \right) \left( 1 + 2 C_2 C_5 + 2 \tilde{C} \right),
\ea\ee
where $\tilde{C}$ is given in Lemma \ref{2bkjl2}.

From (\ref{2bkj202}), (\ref{2bkj323}), (\ref{2bkj324}), and (\ref{2bkj325}), we conclude that (\ref{2bkj301}) holds provided $\nu \ge \hat{\nu}_2$, which completes the proof of Lemma \ref{2bkjl3}.
\end{proof}

\begin{lemma}\la{2bkjl4}
Let $(\n,u)$ be a strong solution to \eqref{ns}--\eqref{bjtj2} satisfying \eqref{pro201}.
Then there exists a positive constant $C$ depending only on
$A$, $\ga$, $\mu$, $E_0$, $\| \n_0 \|_{L^\infty}$, and $\tilde{\n}$ such that
\be\ba\la{2bkj401}
\sup_{0\le t\le T}
\si^2 \int\n|\dot u|^2dx
+\int_0^{T} \si^2 \| \na\dot u\|^2_{L^2} dt \le C.
\ea\ee
\end{lemma}
\begin{proof}
First, for any $v \in H^1(\rr_+)$, we use (\ref{2bkj211}), (\ref{gn11}), and Young's inequality to derive
\be\nonumber\ba
\int |v|^2 dx & \le C \int \n |v|^2 dx + C \int |\n - \tilde{\n}| |v|^2 dx \\
& \le C \| \sqrt{\n} v \|^2_{L^2} + C \| \n - \tilde{\n} \|_{L^2} \| v \|^2_{L^4} \\
& \le C \| \sqrt{\n} v \|^2_{L^2}
+ C \| v \|_{L^2} \| \na v \|_{L^2} \\
& \le \frac{1}{2} \| v \|^2_{L^2} + C \| \sqrt{\n} v \|^2_{L^2} + C \| \na v \|^2_{L^2}.
\ea\ee
This implies that
\be\la{pti}\ba
\|v\|_{L^2} \leq C\|\sqrt{\rho} v\|_{L^2}+C\|\nabla v\|_{L^2}.
\ea\ee

Applying the operator $\dot{u}^j[\frac{\pa}{\pa t}+\div(u\cdot)]$ to
$(\ref{2bkj21})^j$, summing with respect to $j$, and integrating by parts over $\rr_+$, we arrive at
\be\la{2bkj41}\ba
\frac{d}{dt}\left(\frac{1}{2}\int\rho|\dot{u}|^2dx \right)
&=\int \bigg( {\dot{u}}\cdot \nabla G_t +{\dot{u}}^j\mathrm {div}(u \partial _jG) \bigg) dx\\
&\quad +\mu \int \bigg( {\dot{u}}\cdot \nabla ^\bot \omega _t 
+{\dot{u}}^j\partial _k(u^k (\nabla ^\bot \omega )_j ) \bigg) dx \\
& \triangleq I_1+I_2.
\ea\ee
For $I_1$, integrating by parts, adapting the argument similar to (\ref{bkj42}), and using (\ref{2bkj24}) and Young's inequality, we obtain
\be\la{2bkj42}\ba
I_1 & \le - \int \div \dot{u} \dot{G} dx
+ C \| \na G \|_{L^2} \| \na u \|_{L^2} \| \na \dot{u} \|_{L^2} \\
& \le - \int \div \dot{u} \dot{G} dx
+ C \left( \| \sqrt{\n} \dot{u} \|_{L^2} + \| \na u \|_{L^2} \right) \| \na u \|_{L^2} \| \na \dot{u} \|_{L^2} \\
& \le - \int \div \dot{u} \dot{G} dx
+ \ep \| \na \dot{u} \|^2_{L^2}
+ C \| \sqrt{\n} \dot{u} \|^2_{L^2} \| \na u \|^2_{L^2} + C \| \na u \|^4_{L^2}.
\ea\ee
Recalling (\ref{bkj44}), we have
\be\ba\la{2bkj43}
\div \dot{u} = \frac{1}{\nu} \dot{G} - \frac{\ga}{\nu} P \div u + \p_i u^j \p_j u^i,
\ea\ee
which together with (\ref{pro201}) and Young's inequality gives
\be\la{2bkj44}\ba
- \int \div \dot{u} \dot{G} dx
& = -\frac{1}{\nu} \| \dot{G} \|^2_{L^2} + \frac{\ga}{\nu} \int \dot{G} P \div u dx
- \int \dot{G} \p_i u^j \p_j u^i dx \\
& \le -\frac{1}{2\nu} \| \dot{G} \|^2_{L^2}+\frac{C}{\nu} \| \div u \|^2_{L^2}
- \int \dot{G} \p_i u^j \p_j u^i dx.
\ea\ee
Similarly to (\ref{bkj46}), integration by parts yields
\be\la{2bkj45}\ba
- \int \dot{G} \p_i u^j \p_j u^i dx
& = - \frac{d}{dt} \left( \int G \p_i u \cdot \na u^i dx \right)
+ 2 \int G \p_i u \cdot \na \dot{u}^i dx \\
& \quad - 2 \int G \p_i u \cdot \na u \cdot \na u^i dx
+ \int G \div u \p_i u \cdot \na u^i dx.
\ea\ee
It follows from (\ref{bkj43}), (\ref{bkj408}), (\ref{2bkj24}), and Lemma \ref{hm} that
\be\la{2bkj46}\ba
& 2 \int G \p_i u \cdot \na \dot{u}^i dx - 2 \int G \p_i u \cdot \na u \cdot \na u^i dx
+ \int G \div u \p_i u \cdot \na u^i dx \\
& \le C \| \na \dot{u} \|_{L^2} \| G \|_{L^4} \| \div u \|_{L^4}
+ C \| \na G \|_{L^2} \| \na \dot{u} \|_{L^2} \| \na u \|_{L^2}
+ C \| G \|_{L^4} \| \div u \|_{L^4} \| \na u \|^2_{L^4} \\
& \le \ep \| \na \dot{u} \|^2_{L^2} + \frac{C}{\nu^2} \| G \|^4_{L^4}
+ \frac{C}{\nu^2} \| P - P(\tilde{\n}) \|^4_{L^4}
+ C \| \sqrt{\n} \dot{u} \|^2_{L^2} \| \na u \|^2_{L^2} + C \| \na u \|^4_{L^2} + C \| \na u \|^4_{L^4}.
\ea\ee
Combining (\ref{2bkj42}), (\ref{2bkj44}), (\ref{2bkj45}), and (\ref{2bkj46}) leads to
\be\la{2bkj47}\ba
I_1 & \le - \frac{d}{dt} \left( \int G \p_i u \cdot \na u^i dx \right)
- \frac{1}{2\nu} \| \dot{G} \|^2_{L^2} + 2 \ep \| \na \dot{u} \|^2_{L^2} + \frac{C}{\nu^2} \| G \|^4_{L^4} + C \| \na u \|^2_{L^2} \\
& \quad + \frac{C}{\nu^2} \| P - P(\tilde{\n}) \|^4_{L^4}
+ C \| \sqrt{\n} \dot{u} \|^2_{L^2} \| \na u \|^2_{L^2} + C \| \na u \|^4_{L^2} + C \| \na u \|^4_{L^4}.
\ea\ee

For $I_2$, integrating by parts over $\rr_+$ and using the boundary conditions (\ref{bjtj1}) and Young's inequality, we derive
\be\la{2bkj48}\ba
I_2&=\mu \int \left({\dot{u}}\cdot \nabla ^\bot \omega _t
+{\dot{u}}^j \p_k \left( u^k (\na^\bot \o)^j \right) \right) dx \\
&=\mu \int_{\partial \rr_+} ({\dot{u}}\cdot n^\bot )\omega _t ds
-\mu \int \curl \dot{u} \o_t dx-\mu \int u \cdot \na \dot{u} \cdot \na^\bot \o dx \\
&=\mu \int_{\partial \rr_+} \left( -A ({\dot{u}}\cdot n^\bot)^2 + ({\dot{u}}\cdot n^\bot) A (u \cdot \na u \cdot n^\bot) \right) ds
-\mu \int (\curl \dot{u})^2 dx \\
&\quad + \mu \int \curl \dot{u} \curl (u \cdot \na u) dx
- \mu \int_{\p \rr_+} u \cdot \na \dot{u} \cdot n^\bot \o dx
+ \mu \int \left( u \cdot \na \curl \dot{u} + \na^\bot_j u \cdot \na \dot{u}^j \right) \o dx \\
& \le \mu \int_{\partial \rr_+} \left( ({\dot{u}}\cdot n^\bot) A (u \cdot \na u \cdot n^\bot) + u \cdot \na \dot{u} \cdot n^\bot A (u \cdot n^\bot) \right) ds
-\mu \int (\curl \dot{u})^2 dx \\
&\quad + \mu \int \curl \dot{u} (\na^\bot_j u^i \p_i u^j + u \cdot \na \o) dx
+ \mu \int \left( u \cdot \na \curl \dot{u} + \na^\bot_j u \cdot \na \dot{u}^j \right) \o dx \\
& = \mu \int_{\partial \rr_+} \left( ({\dot{u}}\cdot n^\bot) A (u \cdot \na u \cdot n^\bot) + u \cdot \na \dot{u} \cdot n^\bot A (u \cdot n^\bot) \right) ds
-\mu \int (\curl \dot{u})^2 dx \\
&\quad + \mu \int \curl \dot{u} (\na^\bot_j u^i \p_i u^j - \o \div u ) dx
+ \mu \int \na^\bot_j u \cdot \na \dot{u}^j \o dx \\
& \le - \mu \| \curl \dot{u}\|^2_{L^2} + 2 \ep \| \na \dot{u} \|^2_{L^2} + C \| \na u \|^4_{L^4}
+ C \| \na u \|^2_{L^2} + C \| \na u \|^4_{L^2} + C \| \sqrt{\n} \dot{u} \|^2_{L^2},
\ea\ee
where the boundary term is treated via
\be\ba\nonumber
& \mu \int_{\partial \rr_+} \left( ({\dot{u}}\cdot n^\bot) A (u \cdot \na u \cdot n^\bot) + u \cdot \na \dot{u} \cdot n^\bot A (u \cdot n^\bot) \right) ds \\
& = \mu \int_{\partial \rr_+} \left( ({\dot{u}}\cdot n^\bot) A ( (u \cdot n^\bot) n^\bot \cdot \na u \cdot n^\bot)
+ (u \cdot n^\bot) n^\bot \cdot \na \dot{u} \cdot n^\bot A (u \cdot n^\bot) \right) ds \\
& = \mu \int \na^\bot \cdot \left( \na u \cdot n^\bot ({\dot{u}}\cdot n^\bot) A (u \cdot n^\bot)
+ \na \dot{u} \cdot n^\bot A (u \cdot n^\bot)^2 \right) dx \\
& \le C \int \left( |\na u| |\na \dot{u}| | A u| + A |\na u|^2 |\dot{u}| \right) dx \\
& \le C \| \na \dot{u} \|_{L^2} \| \na u \|_{L^2} \| u \|_{L^\infty} + C \| \dot{u} \|_{L^4} \| \na u \|^2_{L^4} \\
& \le C \| \na \dot{u} \|_{L^2} \| \na u \|_{L^2} \left( 1 + \| \na u \|_{L^2} + \| \na u \|_{L^4} \right)
+ C \left( \| \sqrt{\n} \dot{u} \|_{L^2} + \| \na \dot{u} \|_{L^2} \right) \| \na u \|^2_{L^4} \\
& \le \ep \| \na \dot{u} \|^2_{L^2} + C \| \na u \|^2_{L^2} + C \| \na u \|^4_{L^2} + C \| \na u \|^4_{L^4} + C \| \sqrt{\n} \dot{u} \|^2_{L^2},
\ea\ee
owing to (\ref{pti}), (\ref{2bkj01}), and Young's inequality.

It follows from (\ref{2bkj43}) and (\ref{pro201}) that
\be\nonumber\ba
\| \div \dot{u} \|^2_{L^2} \le \frac{C}{\nu^2} \| \dot{G} \|^2_{L^2} + C \| \na u \|^2_{L^2} + C \| \na u \|^4_{L^4},
\ea\ee
which together with (\ref{dc1}) and (\ref{bkjbjds}) yields
\be\la{2bkj49}\ba
\| \na \dot{u} \|^2_{L^2}
& \le C \left( \| \div \dot{u} \|^2_{L^2} + \| \curl \dot{u} \|^2_{L^2} \right) \\
& \le C \left( \frac{1}{2\nu} \| \dot{G} \|^2_{L^2} + \mu \| \curl \dot{u} \|^2_{L^2} \right) + C \| \na u \|^2_{L^2} + C \| \na u \|^4_{L^4}.
\ea\ee
By (\ref{gn11}), (\ref{gw}), (\ref{2bkj24}), (\ref{2bkj211}), and (\ref{dc1}), we have
\be\la{2bkj410}\ba
\| G \|^4_{L^4} & \le C \| G \|^2_{L^2}\| \na G \|^2_{L^2} \\
& \le C \left( \nu^2 \| \div u \|^2_{L^2} + \| P-P(\tilde{\n}) \|^2_{L^2} \right) \left( \| \sqrt{\n} \dot{u} \|^2_{L^2} + \| \na u \|^2_{L^2} \right) \\
& \le C \left( 1 + \nu^2 \| \div u \|^2_{L^2} \right) \left( \| \sqrt{\n} \dot{u} \|^2_{L^2} + \| \na u \|^2_{L^2} \right),
\ea\ee
and
\be\la{2bkj411}\ba
\|\na u\|^4_{L^4} & \le C \left( \|\div u\|^4_{L^4} + \|\o \|^4_{L^4} \right) \\
& \le \frac{C}{\nu^4} \left( \| G \|^4_{L^4} + \| P-P(\tilde{\n}) \|^4_{L^4} \right)
+ C \|\o \|^2_{L^2} \| \na \o \|^2_{L^2} \\
& \le C \left( 1 + \| \div u \|^2_{L^2} + \| \o \|^2_{L^2} \right) \| \sqrt{\n} \dot{u} \|^2_{L^2} + \frac{C}{\nu^4} \| P-P(\tilde{\n}) \|^4_{L^4}.
\ea\ee
Substituting (\ref{2bkj47}) and (\ref{2bkj48}) into (\ref{2bkj41}), choosing $\ep$ sufficiently small, and using (\ref{2bkj49}), (\ref{2bkj410}), (\ref{2bkj411}), and (\ref{2bkj316}), we arrive at
\be\la{2bkj412}\ba
& \frac{d}{dt}\left(\frac{1}{2} \int \rho |\dot{u}|^2 dx + \int G \p_i u \cdot \na u^i dx \right)
+ \frac{1}{4 \nu} \| \dot{G} \|^2_{L^2} + \frac{\mu}{2} \| \curl \dot{u} \|^2_{L^2} \\
& \le - \frac{C}{\nu} \left( \int (P-P(\tilde{\n}))^3 dx \right)_t
+ C \left( \| \sqrt{\n} \dot{u} \|^2_{L^2} + \| \na u \|^2_{L^2} \right) \left( 1 + \| \na u \|^2_{L^2} \right).
\ea\ee
Multiplying (\ref{2bkj412}) by $\si^2$, we obtain from (\ref{pro201}) that
\be\la{2bkj413}\ba
& \frac{d}{dt}\left(\frac{1}{2} \si^2 \int \rho |\dot{u}|^2 dx + \si^2 \int G \p_i u \cdot \na u^i dx \right)
+ \frac{1}{4 \nu} \si^2 \| \dot{G} \|^2_{L^2} + \frac{\mu}{2} \si^2 \| \curl \dot{u} \|^2_{L^2} \\
& \le 2 \si' \si \left(\int\rho|\dot{u}|^2dx + \frac{2(\nu-\mu)}{\nu} \int G_2 \p_i u \cdot \na u^i dx \right)
- \frac{C}{\nu} \left( \si^2 \int (P-P(\tilde{\n}))^3 dx \right)_t \\
& \quad + \frac{C}{\nu} \si \si' \int (P-P(\tilde{\n}))^3 dx
+ C \si \left( \| \sqrt{\n} \dot{u} \|^2_{L^2} + \| \na u \|^2_{L^2} \right).
\ea\ee
Furthermore, in view of (\ref{gn11}), (\ref{2bkj24}), (\ref{2bkj211}), and (\ref{pro201}), we derive
\be\la{2bkj414}\ba
\left| \int G \p_i u \cdot \na u^i dx \right|
& = \left| \int \left( G (\div u)^2 + 2 G \na u^1 \cdot \na^\bot u^2 \right) dx \right| \\
& \le \frac{C}{\nu^2} \| G \|^3_{L^3} + \frac{C}{\nu^2} \| P-P(\tilde{\n}) \|^3_{L^3}
+ C \| \na G \|_{L^2} \| \na u \|^2_{L^2} \\
& \le \frac{C}{\nu^2} \| G \|^2_{L^2} \| \na G \|_{L^2} + C
+ C \left( \| \sqrt{\n} \dot{u} \|_{L^2} + \| \na u \|^2_{L^2} \right) \| \na u \|^2_{L^2} \\
& \le \frac{1}{4} \| \sqrt{\n} \dot{u} \|^2_{L^2}
+\frac{C}{\nu^4} \| G \|^4_{L^2} + C \| \na u \|^4_{L^2} + C \| \na u \|^2_{L^2} + C \\
& \le \frac{1}{4} \| \sqrt{\n} \dot{u} \|^2_{L^2} + C \| \na u \|^4_{L^2} + C \| \na u \|^2_{L^2} + C.
\ea\ee
Integrating (\ref{2bkj413}) over $(0,T)$ and using (\ref{2bkj414}), (\ref{2bkj211}), (\ref{2bkj49}), and (\ref{pro201}), we get (\ref{2bkj401}) and complete the proof of Lemma \ref{2bkjl4}.
\end{proof}

\begin{lemma}\la{2bkjl5}
There exists a positive constant $\nu_2$ depending only on
$A$, $\ga$, $\mu$, $E_0$, $\| \n_0 \|_{L^\infty}$, $\| \na u_0 \|_{L^2}$, and $\tilde{\n}$ such that if $(\n,u)$ is a strong solution to \eqref{ns}--\eqref{bjtj2} satisfying \eqref{pro201}, then
\be\ba\la{2bkj501}
\sup_{0\leq t\leq T} \|\n\|_{L^\infty} 
\leq \frac{3}{2} \left( 1 + \tilde{\n} + \| \n_0 \|_{L^\infty} \right),
\ea\ee
provided $\nu \geq \nu_2$.
\end{lemma}
\begin{proof}
First, by (\ref{pti}), (\ref{2bkj23}), (\ref{2bkj24}), (\ref{dc1}), and (\ref{gw}), we derive that for any $2 \le s<\infty$,
\be\la{2bkj54}\ba
& \| \na G \|_{L^s} \\
& \le C \left( \| \n \dot{u} \|_{L^s} + \| \na u \|_{L^s} \right) \\
& \le C \left( \| \dot{u} \|_{L^2} + \| \na \dot{u} \|_{L^2}
+ \| \div u \|_{L^s} + \| \o \|_{L^s} \right) \\
& \le C \left( \| \sqrt{\n} \dot{u} \|_{L^2} + \| \na \dot{u} \|_{L^2} + \frac{1}{\nu} \| G \|_{L^s} + \frac{1}{\nu} \| P - P(\tilde{\n}) \|_{L^s}
+ \| \o \|_{L^2} + \| \na \o \|_{L^2} \right) \\
& \le C \left( \| \sqrt{\n} \dot{u} \|_{L^2} + \| \na \dot{u} \|_{L^2} + \frac{1}{\nu} \| G \|_{L^s} + \frac{1}{\nu} \| P - P(\tilde{\n}) \|_{L^s}
+ \| \na u \|_{L^2} \right).
\ea\ee
Then, we use (\ref{gw}) to rewrite $(\ref{ns})_1$ as
\be\ba\la{2bkj51}
\frac{d}{dt} \n = g(\n) + h'(t),
\ea\ee
where
\be\la{2bkj52}\ba
g(\n) = - \frac{1}{\nu} \n (\n^{\ga} - \tilde{\n}^\ga), \quad h(t) = - \frac{1}{\nu} \int_0^t \n G ds.
\ea\ee
It follows from (\ref{gw}), (\ref{gn11}), (\ref{2bkj01}), (\ref{2bkj201}), and (\ref{2bkj401}) that for any $0 \le t \le \si(T)$,
\be\ba\la{2bkj55}
|h(t)| & \le \frac{C}{\nu} \int_0^{\si(t)} \| G \|_{L^\infty} ds \\
& \le \frac{C}{\nu} \int_0^{\si(t)} \| G \|_{L^2}^{\frac38} \| \na G \|_{L^5}^{\frac58} ds \\
& \le C \nu^{-\frac58} \int_0^{\si(t)} \left( \| \sqrt{\n} \dot{u} \|_{L^2} + \| \na \dot{u} \|_{L^2} + \frac{1}{\nu} + \| \na u \|_{L^2} \right)^{\frac58} ds \\
& \le C \nu^{-\frac58} + C \nu^{-\frac58} \int_0^{\si(t)} \left( \si^2 \| \sqrt{\n} \dot u \|^{2}_{L^2} + \si^2 \| \na \dot u \|^{2}_{L^2} \right)^{\frac{5}{16}} \si^{-\frac{5}{8}} ds \\
& \le C \nu^{-\frac58} + C \nu^{-\frac58} \left( \int_0^1 \si^{-\frac{10}{11}} ds \right)^{\frac{11}{16}} \\
& \le C \nu^{-\frac58},
\ea\ee
where we have used the following estimate:
\be\nonumber\ba
\| \na G \|_{L^5}
& \le C \left( \| \sqrt{\n} \dot{u} \|_{L^2} + \| \na \dot{u} \|_{L^2} + \frac{1}{\nu} \| G \|_{L^5} + \frac{1}{\nu} \| P - P(\tilde{\n}) \|_{L^5}
+ \| \na u \|_{L^2} \right) \\
& \le C \left( \| \sqrt{\n} \dot{u} \|_{L^2} + \| \na \dot{u} \|_{L^2} + \frac{1}{\nu} \| G \|_{L^2} + \frac{1}{\nu} \| \na G \|_{L^2} + \frac{1}{\nu} + \| \na u \|_{L^2} \right) \\
& \le C \left( \| \sqrt{\n} \dot{u} \|_{L^2} + \| \na \dot{u} \|_{L^2} + \frac{1}{\nu} + \| \na u \|_{L^2} \right),
\ea\ee
due to (\ref{2bkj54}), (\ref{2bkj24}), (\ref{pro201}), and (\ref{2bkj211}).

Combining (\ref{pro201}), (\ref{2bkj51}), and (\ref{2bkj55}) leads to
\be\la{2bkj560}\ba
\sup_{0 \le t \le \si(T)} \| \n \|_{L^\infty} \le \| \n_0 \|_{L^\infty} + C \nu^{-\frac58}.
\ea\ee
In addition, using (\ref{gn11}), (\ref{2bkj54}), (\ref{2bkj24}), (\ref{2bkj311}), and (\ref{2bkj316}), we have
\be\ba\la{2bkj56}
\| G \|^4_{L^\infty}
& \le C \| G \|^2_{L^4} \| \na G \|^2_{L^4} \\
& \le C \| G \|^2_{L^4} \left( \| \sqrt{\n} \dot{u} \|^2_{L^2} + \| \na \dot{u} \|^2_{L^2} + \frac{1}{\nu^2} \| G \|^2_{L^4} + \frac{1}{\nu^2} \| P - P(\tilde{\n}) \|^2_{L^4} + \| \na u \|^2_{L^2} \right) \\
& \le C \| G \|_{L^2} \| \na G \|_{L^2} \left( \| \sqrt{\n} \dot{u} \|^2_{L^2} + \| \na \dot{u} \|^2_{L^2} + \| \na u \|^2_{L^2} \right) \\
& \quad + \frac{C}{\nu^2} \| G \|^4_{L^4} + \frac{C}{\nu^2} \| P - P(\tilde{\n}) \|^4_{L^4} \\
& \le C \left( 1 + \nu \| \div u \|_{L^2} \right)
\left( \| \sqrt{\n} \dot{u} \|_{L^2} + \| \na u \|_{L^2} \right)
\left( \| \sqrt{\n} \dot{u} \|^2_{L^2} + \| \na \dot{u} \|^2_{L^2} + \| \na u \|^2_{L^2} \right) \\
& \quad + C \left( 1 + \| \na u \|^2_{L^2} \right)
\left( \| \sqrt{\n} \dot{u} \|^2_{L^2} + \| \na u \|^2_{L^2} \right)
- \frac{C}{\nu} \left( \int (P-P(\tilde{\n}))^3 dx \right)_t.
\ea\ee
From (\ref{2bkj56}), (\ref{pro201}), (\ref{2bkj401}), and (\ref{2bkj211}), we deduce that
\be\ba\la{2bkj57}
\int_{\si(t)}^t \| G \|^4_{L^\infty} ds
& \le C \int_{\si(t)}^t \| G \|^2_{L^4} \| \na G \|^2_{L^4} ds \\
& \le C \nu^{\frac12} \int_{\si(t)}^t \left( \| \sqrt{\n} \dot{u} \|^2_{L^2} + \| \na \dot{u} \|^2_{L^2} + \| \na u \|^2_{L^2} \right) ds + C \\
& \le C \nu^{\frac12},
\ea\ee
which together with Young's inequality yields that for any $\si(T) \le t_1 \le t_2 \le T$,
\be\la{2bkj58}\ba
h(t_2) - h(t_1) & \le \frac{C}{\nu} \int_{t_1}^{t_2} \| G \|_{L^\infty} dt \\
& \le \frac{1}{\nu}(t_2 - t_1) + \frac{C}{\nu} \int_{\si(T)}^T \| G \|^4_{L^\infty} dt \\
& \le \frac{1}{\nu}(t_2 - t_1) + C \nu^{-\frac{1}{2}}.
\ea\ee
Then, we choose $N_0$, $N_1$, and $\overline{\zeta}$ in Lemma \ref{zli} as follows:
\be\la{2bkj59}\ba
N_0 = C \nu^{-\frac{1}{2}}, \quad N_1 = \frac{1}{\nu}, \quad \overline{\zeta} = 1 + \tilde{\n},
\ea\ee
which together with (\ref{2bkj52}) implies
\be\nonumber\ba
g(\zeta) = - \frac{1}{\nu} \zeta (\zeta^{\ga} - \tilde{\n}^\ga) \le - N_1 = - \frac{1}{\nu} \quad \text{ for all } \zeta \ge 1 + \tilde{\n}.
\ea\ee
This, combined with (\ref{2bkj560}), (\ref{2bkj59}), and Lemma \ref{zli}, gives
\be\la{2bkj510}\ba
\sup_{ \si(T) \le t \le T} \| \n \|_{L^\infty}
\le 1 + \tilde{\n} + \| \n_0 \|_{L^\infty} + \mathbf{M_2} \nu^{-\frac12},
\ea\ee
where $\mathbf{M_2}$ is a positive constant depending only on
$A$, $\tilde{\n}$, $\ga$, $\mu$, $\|\n_0\|_{L^\infty}$, $E_0$, and $\|\na u_0\|_{L^2}$, but is independent of $T$ and $\nu$.

Finally, set
\be\ba\la{2bkj511}
\nu_2 \triangleq \max \left\{ \hat{\nu}_2, \left( \frac{2 \mathbf{M_2}}{1 + \tilde{\n} + \| \n_0 \|_{L^\infty}} \right)^2 \right\},
\ea\ee
with $\hat{\nu}_2$ given in (\ref{2bkj325}).
Then, when $\nu\geq\nu_2$, we derive (\ref{2bkj501}) and complete the proof of Lemma \ref{2bkjl5}.
\end{proof}

\subsection{A Priori Estimates (II): Higher Order Estimates}

In this subsection, we assume that $(\n,u)$ is a strong solution to (\ref{ns})--(\ref{bjtj2}) on $\rr_+ \times (0,T]$ satisfying (\ref{pro201}) and establish the higher-order estimates, which ensure that the strong solution can be extended globally in time.

\begin{lemma}\la{2hl1}
There exists a positive constant $C$ depending only on
$T$, $\ga$, $\mu$, $\lam$, $E_0$, $\tilde{\n}$, $\| \rho_0 \|_{L^\infty}$, $\tilde{\n}$, and $\| \na u_0 \|_{L^2}$ such that
\be\la{2h01} \ba
\sup_{0\le t\le T} \si \int\n|\dot u|^2dx+\int_0^{ T} \si \|\na\dot u\|^2_{L^2}dt\le C.
\ea\ee
\end{lemma}
\begin{proof}
First, we conclude from (\ref{2bkj01}) and (\ref{pro201}) that
\be\la{2h11}\ba
\sup_{0\le t\le T} \| \na u\|_{L^2} + \int_0^T \left( \| \na u\|^2_{L^2}+ \| \sqrt{\n} \dot{u} \|^2_{L^2} \right) dt \le C.
\ea\ee
Multiplying (\ref{2bkj412}) by $\si$, integrating over $(0,T)$, and using (\ref{2h11}), (\ref{2bkj49}), and (\ref{2bkj414}), one derives (\ref{2h01}).
\end{proof}

\begin{lemma}\la{2hl2}
There exists a positive constant $C$ depending only on
$T$, $q$, $\ga$, $\mu$, $\lam$, $E_0$, $\tilde{\n}$, $\| \rho_0 \|_{L^\infty}$, $\| \na \n_0 \|_{L^2 \cap L^q}$, and $\| \na u_0 \|_{L^2}$ such that
\be\la{2h02}\ba
&\sup_{0\le t\le T} \left( \| \n - \tilde{\rho} \|_{H^1 \cap W^{1,q}} + \| u \|_{H^1} 
+ t \| \na^2 u \|^2_{L^2} + t \| \sqrt{\n} u_t \|^2_{L^2} + \| \n_t \|_{L^2} \right) \\
& + \int_0^T \left( \| u \|^{2}_{H^2}+\|\nabla^2 u\|^{(q+1)/q}_{L^q}+t \|\nabla^2 u\|_{L^q}^2 
+ \| \sqrt{\n} u_t \|^2_{L^2} + t \| u_t \|^2_{H^1} \right) dt
\le C.
\ea\ee
\end{lemma}
\begin{proof}
First, adapting the argument in Lemma \ref{1hl1} and using (\ref{pti}) and (\ref{2bkj211}), we can obtain
\be\la{2h21}\ba
&\sup_{0\le t\le T} \left( \| \n - \tilde{\rho} \|_{H^1 \cap W^{1,q}} + \| u \|_{H^1} 
+ t \| \na^2 u \|^2_{L^2} \right) \\
& + \int_0^T \left( \| u \|^{2}_{H^2}+\|\nabla^2 u\|^{(q+1)/q}_{L^q}+t \|\nabla^2 u\|_{L^q}^2 \right) dt
\le C.
\ea\ee
By $(\ref{ns})_1$, (\ref{pti}), (\ref{2h21}), and H\"older's inequality, one has
\be\la{2h22}\ba
\| \n_t\|_{L^2}\le
C\|u\|_{L^{2q/(q-2)}}\|\nabla \n \|_{L^q}+C\|\n\|_{L^\infty} \|\nabla u\|_{L^2} \le C.
\ea\ee
Using (\ref{2h21}), (\ref{gn11}), and H\"older's inequality, we arrive at
\be\la{2h23}\ba 
\int\rho|u_t|^2dx 
&\le \int\rho| \dot u |^2dx+\int \n |u\cdot\na u|^2dx \\
&\le \int\rho| \dot u |^2dx+C \| u \|_{L^4}^2 \| \na u\|_{L^4}^2 \\ 
&\le \int\rho| \dot u |^2dx+C\| \na^2 u\|_{L^2}^2,
\ea\ee
and
\be\la{2h24}\ba 
\|\nabla u_t\|_{L^2}^2 
&\le \| \nabla \dot u \|_{L^2}^2+ \| \nabla(u\cdot\nabla u)\|_{L^2}^2  \\ 
&\le \|\nabla \dot u\|_{L^2}^2+ \|u\|_{L^{2q/(q-2)}}^2\|\nabla^2u \|_{L^q}^2+ \| \nabla u \|_{L^4}^4 \\ 
&\le \|\nabla \dot u\|_{L^2}^2+C\|\nabla^2u \|_{L^q}^2+ \| \nabla u \|_{L^4}^4.
\ea\ee
Combining (\ref{pti}), (\ref{2h01}), (\ref{2h21}), (\ref{2h23}), and (\ref{2h24}) leads to
\be\la{2h25}\ba
\sup_{0\le t\le T} t \| \sqrt{\n} u_t \|^2_{L^2} 
+ \int_0^T \| \sqrt{\n} u_t \|^2_{L^2} + t\|u_t\|_{H^1}^2 dt
\le C.
\ea\ee
This, together with (\ref{2h21}), gives (\ref{2h02}) and completes the proof of Lemma \ref{2hl2}.
\end{proof}

\section{Proofs of Theorems \ref{th0}--\ref{th3}}
With the a priori estimates established in Sections 3 and 4 at hand, we prove the main results of this paper in this section.

\noindent\textbf{Proof of Theorem \ref{th1}}.
By the local existence result in Lemma \ref{lct},
there exists a $T>0$ such that the problem (\ref{ns})--(\ref{bjtj2}) with $\tilde{\n}=0$ has a unique strong solution $(\n,u)$ on $\rr_+ \times (0,T]$.
From the a priori estimates in Proposition \ref{pro1} and Lemmas \ref{1hl1}--\ref{1hl3}, we deduce that the problem (\ref{ns})--(\ref{bjtj2}) with $\tilde{\n}=0$ has a global strong solution $(\n,u)$ satisfying the properties listed in Theorem \ref{th1}, provided $\nu \ge \nu_1$.
Moreover, the proof of uniqueness of $(\n,u)$ satisfying \eqref{wsol2}, \eqref{wsol3}, \eqref{wsol30}, \eqref{wsol300}, and (\ref{ssol4}) is similar to that in \cite{LLL}.

Similarly, using Proposition \ref{pro2}, Lemmas \ref{2hl1}, and \ref{2hl2}, we can derive that when $\nu \ge \nu_2$, the problem (\ref{ns})--(\ref{bjtj2}) with $\tilde{\n}>0$ has a unique global strong solution $(\n,u)$ satisfying (\ref{2wsol2}), (\ref{wsol4}), and (\ref{ssol5}).
It remains to prove the decay estimates (\ref{cpwsol4}).

It follows from (\ref{pro201}), (\ref{2bkj201}), (\ref{2bkj311}), (\ref{2bkj316}), and (\ref{2bkj411}) that
\be\la{qkjltb1}\ba
\sup_{0 \le t < \infty} \left( \| \n \|_{L^\infty} + \| \na u \|_{L^2} \right)
+ \int_1^\infty \left( \| \n - \tilde{\n} \|^4_{L^4} + \| \na u \|^2_{L^2} + \| \na u \|^4_{L^4} \right) dt \le C.
\ea\ee
Multiplying $(\ref{ns})_1$ by $4(\n-\tilde{\n})^3$, integrating by parts over $\rr_+$, and using (\ref{qkjltb1}) and H\"older's inequality, we derive
\be\la{qkjltb2}\ba
\frac{d}{dt} \left( \| \n - \tilde{\n} \|^4_{L^4} \right)
& = \int \left( (\n - \tilde{\n})^4 \div u - 4 (\n - \tilde{\n})^3 \n \div u \right) dx \\
& \le C \| \n - \tilde{\n} \|^4_{L^4} + C \| \na u \|^2_{L^2},
\ea\ee
which implies that for all $1 \le N \le s \le N+1 \le t \le N+2$,
\be\la{qkjltb3}\ba
\| \n(\cdot,t) - \tilde{\n} \|^4_{L^4}
\le \| \n(\cdot,s) - \tilde{\n} \|^4_{L^4} + C \int_N^{N+1} \left( \| \n - \tilde{\n} \|^4_{L^4} + C \| \na u \|^2_{L^2} \right) dt.
\ea\ee
Integrating (\ref{qkjltb3}) with respect to $s$ over $[N,N+1]$ leads to
\be\la{qkjltb4}\ba
\| \n(\cdot,t) - \tilde{\n} \|^4_{L^4}
\le C \int_N^{N+1} \left( \| \n - \tilde{\n} \|^4_{L^4} + C \| \na u \|^2_{L^2} \right) dt.
\ea\ee
This, combined with (\ref{qkjltb1}), yields
\be\la{qkjltb5}\ba
\lim_{t \to \infty} \| \n(\cdot,t) - \tilde{\n} \|_{L^4} = 0.
\ea\ee
Using (\ref{qkjltb1}) and H\"older's inequality, we obtain that for any $s \in (2,\infty)$,
\be\la{qkjltb6}\ba
\lim_{t \to \infty} \| \n(\cdot,t) - \tilde{\n} \|_{L^s} = 0.
\ea\ee
Moreover, a straightforward calculation gives
\be\la{qkjltb7}\ba
\int_1^\infty \left| \frac{d}{dt}(\| \na u \|^2_{L^2}) \right| dt
& = 2 \int_1^\infty \left| \int \p_i u^j \p_i u^j_t dx \right| dt \\
& = 2 \int_1^\infty \left| \int \p_i u^j \p_i \left( \dot{u}^j - u^k \p_k u^j \right) dx \right| dt \\
& = \int_1^\infty \left| \int \left( 2 \p_i u^j \p_i \dot{u}^j - 2 \p_i u^j \p_i u^k \p_k u^j + |\na u|^2 \div u \right) dx \right| dt \\
& \le C \int_1^\infty \left( \| \na \dot{u} \|^2_{L^2} + \| \na u \|^2_{L^2} + \| \na u \|^4_{L^4} \right) dt \le C,
\ea\ee
which together with (\ref{qkjltb1}) shows
\be\la{qkjltb8}\ba
\lim_{t \to \infty} \| \na u \|_{L^2} = 0.
\ea\ee
In addition, from (\ref{2bkj23}), (\ref{2bkj211}), and (\ref{qkjltb1}), we conclude that for any $2 \le r <\infty$,
\be\la{qkjltb9}\ba
\| \na u \|_{L^r}
& \le C \left( \| \div u \|_{L^r} + \| \o \|_{L^r} \right) \\
& \le C \left( \| G \|_{L^r} + \| P-P(\tilde{\n}) \|_{L^r} + \| \o \|_{L^2} + \| \na \o \|_{L^2} \right) \\
& \le C \left( 1 + \| \na u \|_{L^2} + \| \n \dot{u} \|_{L^2} \right) \le C,
\ea\ee
which together with (\ref{qkjltb8}) implies that
for any $2 \le r <\infty$,
\be\la{qkjltb11}\ba
\lim_{t \to \infty} \| \na u \|_{L^r} = 0.
\ea\ee
In view of (\ref{qkjltb6}) and (\ref{qkjltb11}), we arrive at (\ref{cpwsol4}) and complete the proof of Theorem~\ref{th1}.

\noindent\textbf{Proof of Theorem \ref{th0}}.
By employing standard compactness arguments in \cite{F,L2,LZZ}, the proof is similar to that of Theorem \ref{th1}, and thus is omitted.

\noindent\textbf{Proof of Theorem \ref{th01}}.
For any $0<T<\infty$, when $\nu>\nu_1$, from (\ref{wsol2}), (\ref{bkj01}), (\ref{bkj022}), (\ref{bkj211}), (\ref{bkj53}), and (\ref{pt1}), we conclude that $\{\n^{\nu}\}_\nu$ is bounded in $L^\infty(0,\infty;L^1) \cap L^\infty(\rr_+ \times (0,\infty))$, and for any $0<\tau<T$ and $0<R<\infty$, $\{u^{\nu}\}_\nu$ is bounded in $L^\infty(\tau,T;H^1(B^+_R))\cap L^2(0,T;H^1(B^+_R))$, and $\{ \na u^{\nu}\}_\nu$ is bounded in $L^2(\rr_+ \times (0,\infty)) \cap L^\infty(\tau,\infty;L^2)$.

Furthermore, by (\ref{pt1}), (\ref{gw}), (\ref{bkj24}), (\ref{gn11}),
and H\"older's inequality, we derive
\be\la{ins01}\ba
\| u^{\nu}_t \|_{L^2(B^+_R)}
& \le C \left( \| \dot{u^{\nu}} \|_{L^2(B^+_R)} + \| u^{\nu} \cdot \na u^{\nu} \|_{L^2(B^+_R)} \right) \\
& \le C \left( \| \sqrt{\n^{\nu}} \dot{u^{\nu}} \|_{L^2(B^+_R)} 
+ \| \na \dot{u^{\nu}} \|_{L^2(B^+_R)}
+ \| u^{\nu} \|_{L^4(B^+_R)} \| \na u^{\nu} \|_{L^4(B^+_R)} \right) \\
& \le C \| \sqrt{\n^{\nu}} \dot{u^{\nu}} \|_{L^2} + C \| \na \dot{u^{\nu}} \|_{L^2}
+ \frac{C}{\nu} \| u^{\nu} \|_{H^1(B^+_R)} \| G^\nu \|^{1/2}_{L^2} \| \n^\nu \dot{u^{\nu}} \|^{1/2}_{L^2} \\
& \quad + C \| u^{\nu} \|_{H^1(B^+_R)} 
\left( \| \o^\nu \|^{1/2}_{L^2} \| \n^\nu \dot{u^{\nu}} \|^{1/2}_{L^2}
+\| P^\nu \|_{L^4} \right),
\ea\ee
which, together with (\ref{bkj022}) and (\ref{bkj401}), 
implies that $\{u^{\nu}\}_\nu$ is bounded in $H^1(\tau,T;L^2(B^+_R))$ 
for any fixed constant $R$.

Thus, without loss of generality, we may assume that there exists a subsequence $(\n^n,u^n)$ of $(\n^{\nu},u^{\nu})$ such that
\be\la{ins1}\ba
\begin{cases}
\n^n \rightharpoonup \n  \mbox{ weakly * in } L^\infty(\rr_+ \times (0,T)),\\
u^n \rightharpoonup u  \mbox{ weakly * in } \ L^\infty(\tau,T;H^1(B^+_R))\cap L^2(0,T;H^1(B^+_R)), \\
u^n \to u  \mbox{ strongly  in } \ L^\infty(\tau,T;L^p(B^+_R)), \\
\na	u^n \rightharpoonup \na u  \mbox{ weakly * in } \ L^\infty(\tau,T;L^2(\rr_+))\cap L^2(\rr_+ \times (0,T)),
\end{cases}
\ea\ee 
for any $1 \le p < \infty$ and $0<R<\infty$.

In addition, we set $G^n \triangleq n \div u^n - P^n$ and $\o^n \triangleq \na^\bot \cdot u^n$.
Using (\ref{bkj28}), (\ref{bkj211}), (\ref{wsol2}), and H\"older's inequality, we obtain that for any $2<r<\infty$,
\be\la{ins03}\ba
\| G^n \|_{L^r} & \le C\| \n^n \dot{u^n} \|_{ L^{\frac{2r}{r+2}} }
\le C\| \sqrt{\n^n} \|_{ L^{2r} } \| \sqrt{\n^n} \dot{u^n} \|_{L^2}
\le C\| \n^n \|^{\frac{1}{2}}_{ L^{r} } \| \sqrt{\n^n} \dot{u^n} \|_{L^2}.
\ea\ee
It follows from (\ref{ins03}), (\ref{bkj24}), and (\ref{bkj022}) that for any $0<\tau<T$, $0<R<\infty$, and $2<r<\infty$, $\{G^n\}_n$ is bounded in $L^2(\tau,T;L^r) \cap L^2(\tau,T;H^1(B^+_R))$, $\{\na G^n\}_n$ and $\{\na \o^n\}_n$ are bounded in $L^2(\rr_+ \times (\tau,T))$.
Hence, without loss of generality, we can assume that there exist $\pi \in L^2(\tau,T;L^r) \cap L^2(\tau,T;H^1(B^+_R))$
and $\na \pi \in L^2(\rr_+ \times (\tau,T)) $ such that
\be\la{ins3}\ba
G^n \rightharpoonup - \pi \quad  \mbox{ weakly in } L^2(\tau,T;H^1(B^+_R)),
\ea\ee
for any $0<R<\infty$.
Recalling (\ref{ns}), we see that $(\n^n,u^n)$ satisfies
\be\la{ins4}\ba
\begin{cases}
(\n^n)_t+\div(\n^n u^n)=0, \\
(\n^n u^n)_t+\div(\n^n u^n\otimes u^n) -\na G^n -\mu \na^{\bot} w^n = 0.
\end{cases}
\ea\ee
Using (\ref{ins1}) and (\ref{ins3}) and taking the limit as $n \to \infty$, we deduce that $(\n,u)$ satisfies
\be\la{ins6}\ba
\begin{cases}
\n_t+\div(\n u)=0,\\
(\n u)_t+\div(\n u\otimes u) -\mu \na^{\bot} w + \na \pi = 0.
\end{cases} 
\ea\ee
Moreover, (\ref{bkj01}) and (\ref{bkj022}) imply that
\be\la{ins7}\ba
\div u^n \to 0  \  \mbox{ strongly in } L^2(\rr_+ \times (0,T))\cap L^\infty( (\tau,T); L^2),
\ea\ee
which together with (\ref{ins1}) gives (\ref{isol4}) and $\div u=0$.
This, combined with the equality $\Delta u = \na \div u + \na^{\bot} w $, yields $\na^{\bot} w=\Delta u$.
Therefore, $(\n,u)$ satisfies (\ref{isol2}) and (\ref{isol3}).

Furthermore, for any $0<T<\infty$, when $\nu \ge \nu_1$ and the initial data $(\n_0,u_0)$ satisfy (\ref{ws}), we deduce from (\ref{wsol2}), (\ref{bkj01}), (\ref{bkj021}), (\ref{bkj211}), (\ref{bkj53}), and (\ref{pt1}) that the sequence $\{\n^{\nu}\}_\nu$ is bounded in $L^\infty(0,T;L^1) \cap L^\infty(\rr_+ \times (0,T))$,
and for any $0<R<\infty$, $\{u^{\nu}\}_\nu$ is bounded in $L^\infty(0,T;H^1(B^+_R))$, $\{\na u^{\nu}\}_\nu$ is bounded in $L^2(\rr_+ \times (0,\infty)) \cap L^\infty(0,\infty;L^2)$,
and $\{\sqrt{\n^\nu} \du^{\nu}\}_\nu$ is bounded in $L^2(\rr_+ \times (0,\infty))$.
Then, multiplying (\ref{bkj417}) by $\si$, integrating the resulting equation over $(0,T)$, and using (\ref{wsol2}), (\ref{bkj021}), and (\ref{bkj418}), we derive that $\{\sqrt{t} \sqrt{\n^\nu} \dot{u^{\nu}}\}_\nu$ is bounded in $L^\infty(0,T;L^2)$ and $\{\sqrt{t} \na \dot{u^{\nu}}\}_\nu$ is bounded in $L^2(\rr_+ \times (0,T))$.
This, combined with (\ref{bkj23}), (\ref{bkj501}), and (\ref{ins03}), yields that for any $2 \le s <\infty$ and $2<r<\infty$, $\{\sqrt{t} \na G^{\nu}\}_\nu$ and $\{\sqrt{t} \na \o^{\nu}\}_\nu$ 
are bounded in $L^\infty(0,T;L^2) \cap L^2(0,T;L^s)$, and $\{ \sqrt{t} G^\nu \}_{\nu}$ is bounded in $L^\infty(0,T;L^r)$.
By arguments similar to those above, we can obtain that there exists a subsequence of $(\n^{\nu},u^{\nu})$ that converges to a global solution of (\ref{isol2}) satisfying (\ref{lws1}).
This completes the proof of Theorem \ref{th01}.

\noindent\textbf{Proof of Theorem \ref{th001}}.
For any $0<T<\infty$, when $\nu>\nu_2$, by (\ref{wsol2}), (\ref{bkj01}), (\ref{pro201}), and (\ref{pti}), we obtain that $\{\n^{\nu}\}_\nu$ is bounded in $L^\infty(\rr_+ \times (0,T))$, and for any $0<\tau<T$, $\{u^{\nu}\}_\nu$ is bounded in $L^\infty(\tau,T;H^1)\cap L^2(0,T;H^1)$.
Moreover, in view of (\ref{2wsol2}), (\ref{2bkj211}), (\ref{2bkj411}), and (\ref{pti}), we arrive at
\be\nonumber\ba
\| u^{\nu}_t \|_{L^2}
& \le C \left( \| \dot{u^{\nu}} \|_{L^2} + \| u^{\nu} \cdot \na u^{\nu} \|_{L^2} \right) \\
& \le C \left( \| \sqrt{\n^{\nu}} \dot{u^{\nu}} \|_{L^2} + \| \na \dot{u^{\nu}} \|_{L^2}
+ \| u^{\nu} \|_{L^4} \| \na u^{\nu} \|_{L^4} \right) \\
& \le C \| \sqrt{\n^{\nu}} \dot{u^{\nu}} \|_{L^2} + C \| \na \dot{u^{\nu}} \|_{L^2}
+ C \left( 1 + \| u^{\nu} \|^{2}_{H^1} \right) \| \sqrt{\n^{\nu}} \dot{u^{\nu}} \|^{1/2}_{L^2} + C \| u^{\nu} \|_{H^1},
\ea\ee
which together with (\ref{2bkj01}), (\ref{pti}), and (\ref{2bkj401}) yields $\{u^{\nu}\}_\nu$ is bounded in $H^1(\tau,T;L^2)$.

Consequently, without loss of generality, we can assume that there exists a subsequence $(\n^n,u^n)$ of $(\n^{\nu},u^{\nu})$ such that
\be\la{inss1}\ba
\begin{cases}
\n^n \rightharpoonup \n  \mbox{ weakly * in } L^\infty(\rr_+ \times (0,T)), \\
u^n \rightharpoonup u  \mbox{ weakly * in } \ L^2(0,T;H^1) \cap L^\infty(\tau,T;H^1), \\
u^n \to u  \mbox{ strongly  in } \ L^\infty(\tau,T;L^p(B^+_R)),
\end{cases}
\ea\ee
for any $1 \le p < \infty$ and $0<R<\infty$.

Define $G^n \triangleq n \div u^n-(P^n-P(\tilde{\n}))$ and $\o^n \triangleq \na^\bot \cdot u^n$, which together with (\ref{pro201}) and (\ref{2bkj24}) shows that $\{\na G^n\}_n$ and $\{\o^n\}_n$ are bounded in $L^2(\rr_+ \times (\tau,T))$.
Thus, without loss of generality, we may assume that there exists $ \Psi  \in L^2(\rr_+ \times (\tau,T))$ such that
\be\la{inss3}\ba
\na G^n \rightharpoonup \Psi \quad  \mbox{ weakly in } L^2(\rr_+ \times (\tau,T)).
\ea\ee
Moreover, for any $\phi \in \left( \mathcal{D}(\rr_+) \right)^2$ with $\div \phi=0$ and $\psi \in \mathcal{D}(\tau,T)$, we have
\be\la{inss6}\ba
\int_\tau^T \int \na G^n \cdot \phi dx \ \psi \ dt = 0.
\ea\ee
Letting $n \to \infty$ and using (\ref{inss3}), we arrive at
\be\la{inss7}\ba
\int_\tau^T \int \Psi \cdot \phi dx \ \psi \ dt = 0,
\ea\ee
which yields
\be\la{inss8}\ba
\int \Psi \cdot \phi dx =0, \ \ \mathrm{a.e.\ }t \in (\tau,T).
\ea\ee

By virtue of the De Rham Theorem \cite[Proposition 1.1]{TR}, we can deduce that there exists a distribution $ \pi $ such that $ \Psi = -\na \pi$.

Letting $n \to \infty$ and using (\ref{inss1}) and (\ref{inss3}), we derive that $(\n,u)$ satisfies
\be\la{inss9}\ba
\begin{cases}
\n_t+\div(\n u)=0,\\
(\n u)_t+\div(\n u\otimes u) -\mu \na^{\bot} w + \na \pi =0.
\end{cases} 
\ea\ee
In addition, (\ref{pro201}) and (\ref{2bkj01}) ensure that (\ref{lisol4}) holds and $\div u = 0$, thereby giving $\na^{\bot} w=\Delta u$.
Hence, $(\n,u)$ satisfies (\ref{isol2}) and (\ref{lisol3}).

Moreover, for any $0<T<\infty$, when $\nu \ge \nu_2$ and the initial data $(\n_0,u_0)$ satisfy (\ref{0lws}), by an argument analogous to the case of $\tilde{\n}=0$, we deduce that $(\n^{\nu},u^{\nu})$ has a subsequence converging to a global solution of (\ref{isol2}), and $(\n,u)$ satisfies (\ref{0lws1}).
The proof of Theorem \ref{th001} is completed.

\noindent\textbf{Proof of Theorem \ref{th3}}.
The proof of Theorem \ref{th3} is similar to that of \cite[Theorem 1.2]{LX}, hence we omit it here.

\bigskip

\noindent\textbf{Data availability.} No data was used for the research described in the article.

\bigskip

\noindent\textbf{Conflict of interest.} The authors declare that they have no conflict of interest.


\begin {thebibliography} {99}



\bibitem{AJ} J. Aramaki, 
$L^p$ theory for the div-curl system,
Int. J. Math. Anal. (Ruse) {\bf 8} (2014), no.~5-8, 259--271.



\bibitem{BKM} J.~T. Beale, T. Kato and A.~J. Majda,
Remarks on the breakdown of smooth solutions for the $3$-D Euler equations,
Comm. Math. Phys. {\bf 94} (1984), no.~1, 61--66.




\bibitem{CL} G.~C. Cai and J. Li,
Existence and exponential growth of global classical solutions to the compressible Navier-Stokes equations with slip boundary conditions in 3D bounded domains,
Indiana Univ. Math. J. {\bf 72} (2023), no.~6, 2491--2546.



\bibitem{CCK} Y. Cho, H.~J. Choe and H. Kim,
Unique solvability of the initial boundary value problems for compressible viscous fluids,
J. Math. Pures Appl. (9) {\bf 83} (2004), no.~2, 243--275.



\bibitem{CK} Y. Cho and H. Kim,
On classical solutions of the compressible Navier-Stokes equations with nonnegative initial densities,
Manuscripta Math. {\bf 120} (2006), no.~1, 91--129.



\bibitem{CK2} H.~J. Choe and H. Kim,
Strong solutions of the Navier-Stokes equations for isentropic compressible fluids,
J. Differential Equations {\bf 190} (2003), no.~2, 504--523.







\bibitem{CLMS} R.~R. Coifman, Lions, P. L, Meyer, Y, Semmes, S.,
Compensated compactness and Hardy spaces,
J. Math. Pures Appl. (9) {\bf 72} (1993), no.~3, 247--286.



\bibitem{D} R. Danchin,
Global existence in critical spaces for compressible Navier-Stokes equations,
Invent. Math. {\bf 141} (2000), no.~3, 579--614.



\bibitem{DM3} R. Danchin and P.~B. Mucha,
Compressible Navier-Stokes system: large solutions and incompressible limit,
Adv. Math. {\bf 320} (2017), 904--925.



\bibitem{DM} R. Danchin and P.~B. Mucha,
Compressible Navier-Stokes equations with ripped density,
Comm. Pure Appl. Math. {\bf 76} (2023), no.~11, 3437--3492.





\bibitem{DQ} Q. Duan,
Global well-posedness of classical solutions to the compressible Navier-Stokes equations in a half-space, J. Differential Equations {\bf 253} (2012), no.~1, 167--202.






\bibitem{FC} C.~L. Fefferman,
Characterizations of bounded mean oscillation,
Bull. Amer. Math. Soc. {\bf 77} (1971), 587--588.



\bibitem{F}  E. Feireisl,
Dynamics of Viscous Compressible Fluids,
Oxford Lecture Series in Mathematics and its Applications vol. 26, Oxford University Press, Oxford, 2004.



\bibitem{FNP} E. Feireisl, A. Novotn\'y{} and H. Petzeltov\'a,
On the existence of globally defined weak solutions to the Navier-Stokes equations,
J. Math. Fluid Mech. {\bf 3} (2001), no.~4, 358--392.



\bibitem{GT}  D. Gilbarg and N.~S. Trudinger,
Elliptic partial differential equations of second order, Springer, 2001.







\bibitem{H4} D. Hoff,
Global existence for 1D, compressible, isentropic Navier-Stokes equations with large initial data,
Trans. Amer. Math. Soc. {\bf 303} (1987), no.~1, 169--181.



\bibitem{H1} D. Hoff,
Global solutions of the Navier-Stokes equations for multidimensional compressible flow with discontinuous initial data,
J. Differential Equations {\bf 120} (1995), no.~1, 215--254.



\bibitem{H2} D. Hoff,
Strong convergence to global solutions for multidimensional flows of compressible, viscous fluids with polytropic equations of state and discontinuous initial data,
Arch. Rational Mech. Anal. {\bf 132} (1995), no.~1, 1--14.







\bibitem{H3} D. Hoff,
Compressible flow in a half-space with Navier boundary conditions,
J. Math. Fluid Mech. {\bf 7} (2005), no.~3, 315--338.







\bibitem{HL} X.-D. Huang and J. Li,
Global well-posedness of classical solutions to the Cauchy problem of two-dimensional barotropic compressible Navier-Stokes system with vacuum and large initial data,
SIAM J. Math. Anal. {\bf 54} (2022), no.~3, 3192--3214.



\bibitem{HLX3} X.-D. Huang, J. Li and Z. Xin,
Blowup criterion for viscous baratropic flows with vacuum states,
Comm. Math. Phys. {\bf 301} (2011), no.~1, 23--35.



\bibitem{HLX1} X.-D. Huang, J. Li and Z. Xin,
Serrin-type criterion for the three-dimensional viscous compressible flows,
SIAM J. Math. Anal. {\bf 43} (2011), no.~4, 1872--1886.



\bibitem{HLX2} X.-D. Huang, J. Li and Z. Xin,
Global well-posedness of classical solutions with large oscillations and vacuum to the three-dimensional isentropic compressible Navier-Stokes equations,
Comm. Pure Appl. Math. {\bf 65} (2012), no.~4, 549--585.



\bibitem{JK} D.~S. Jerison and C.~E. Kenig,
The Neumann problem on Lipschitz domains,
Bull. Amer. Math. Soc. (N.S.) {\bf 4} (1981), no.~2, 203--207.







\bibitem{K} T. Kato,
Remarks on the Euler and Navier-Stokes equations in ${\bf R}^2$,
Proc. Sympos. Pure Math., {\bf 45}, (1986),1--7.



\bibitem{KS} A.~V. Kazhikhov and V.~V. Shelukhin,
Unique global solution with respect to time of initial-boundary value problems for one-dimensional equations of a viscous gas,
Prikl. Mat. Meh. {\bf 41} (1977), no.~2J. Appl. Math. Mech. {\bf 41} (1977), no.~2.



\bibitem{Lei1} Q. Lei,
Global Existence and Incompressible Limit for the Three-Dimensional Axisymmetric Compressible Navier-Stokes Equations with Large Bulk Viscosity and Large Initial Data, arXiv:2509.12614.



\bibitem{LeXi1} Q. Lei and C. Xiong,
Global Existence and Incompressible Limit for Compressible Navier-Stokes Equations with Large Bulk Viscosity Coefficient and Large Initial Data, arXiv:2507.01432.



\bibitem{LeXi2} Q. Lei and C. Xiong,
Global Existence and Incompressible Limit for Compressible Navier-Stokes Equations in Bounded Domains with Large Bulk Viscosity Coefficient and Large Initial Data, arXiv:2507.02462.



\bibitem{LeXi3} Q. Lei and C. Xiong,
Global Existence and Incompressible Limit of the Cauchy Problem for 2D Compressible Navier-Stokes Equations with Large Bulk Viscosity and Large Initial Data, arXiv:2507.02497.



\bibitem{LLL} J. Li, Z. Liang,
On local classical solutions to the Cauchy problem of the two-dimensional barotropic compressible Navier-Stokes equations with vacuum,
J. Math. Pures Appl. (9) {\bf 102} (2014), no.~4, 640--671.



\bibitem{LX} J. Li and Z. Xin,
Some uniform estimates and blowup behavior of global strong solutions to the Stokes approximation equations for two-dimensional compressible flows,
J. Differential Equations {\bf 221} (2006), no.~2, 275--308.



\bibitem{LX2} J. Li and Z. Xin,
Global well-posedness and large time asymptotic behavior of classical solutions to the compressible Navier-Stokes equations with vacuum,
Ann. PDE {\bf 5} (2019), no.~1, Paper No. 7, 37 pp.



\bibitem{LZZ} J. Li, J.~W. Zhang and J.~N. Zhao,
On the global motion of viscous compressible barotropic flows subject to large external potential forces and vacuum,
SIAM J. Math. Anal. {\bf 47} (2015), no.~2, 1121--1153.



\bibitem{LZ2} X. Liao and S.M. Zodji,
Global-in-time well-posedness of the compressible Navier-Stokes equations with striated density, arXiv:2405.11900.



\bibitem{L1}  P.L. Lions,
Mathematical Topics in Fluid Mechanics. Vol. 1: Incompressible Models,
Oxford University Press, New York, 1996.



\bibitem{L2}  P.L. Lions,
Mathematical Topics in Fluid Mechanics. Vol. 2: Compressible Models,
Oxford University Press, New York, 1998.



\bibitem{LZ} Z. Luo,
Local existence of classical solutions to the two-dimensional viscous compressible flows with vacuum,
Commun. Math. Sci. {\bf 10} (2012), no.~2, 527--554.



\bibitem{LSZ}B. L\"u, X. Shi and X. Zhong,
Global existence and large time asymptotic behavior of strong solutions to the Cauchy problem of 2D density-dependent Navier-Stokes equations with vacuum,
Nonlinearity {\bf 31} (2018), no.~6, 2617--2632.


\bibitem{MN1} A. Matsumura, T. Nishida,
The initial value problem for the equations of motion of viscous and heat-conductive gases,
J. Math. Kyoto Univ. {\bf 20}(1) (1980), 67--104.



\bibitem{MD} D.~I.~R. Mitrea,
Integral equation methods for div-curl problems for planar vector fields in nonsmooth domains,
Differential Integral Equations {\bf 18} (2005), no.~9, 1039--1054.







\bibitem{N} J. Nash,
Le probl\`{e}me de Cauchy pour les \'{e}quations diff\'{e}rentielles d'un fluide g\'{e}n\'{e}ral,
Bull. Soc. Math. France {\bf 90} (1962), 487--497 (French).



\bibitem{NI} L. Nirenberg,
On elliptic partial differential equations,
Ann. Scuola Norm. Sup. Pisa Cl. Sci. (3) {\bf 13} (1959), 115--162.









\bibitem{SS} R. Salvi and I. Stra\v skraba,
Global existence for viscous compressible fluids and their behavior as $t\to\infty$,
J. Fac. Sci. Univ. Tokyo Sect. IA Math. {\bf 40} (1993), no.~1, 17--51.



\bibitem{S1} D. Serre,
Solutions faibles globales des \'equations de Navier-Stokes pour un fluide compressible,
C. R. Acad. Sci. Paris S\'er. I Math. {\bf 303} (1986), no.~13, 639--642.



\bibitem{S2} D. Serre,
Sur l'\'equation monodimensionnelle d'un fluide visqueux, compressible et conducteur de chaleur,
C. R. Acad. Sci. Paris S\'er. I Math. {\bf 303} (1986), no.~14, 703--706.



\bibitem{S} J. Serrin,
On the uniqueness of compressible fluid motions,
Arch. Rational Mech. Anal. {\bf 3} (1959), 271--288.



\bibitem{SE} E.~M. Stein,
Singular integrals and differentiability properties of functions,
Princeton Mathematical Series, No. 30, Princeton Univ. Press, Princeton, NJ, 1970.



\bibitem{SEM} E.~M. Stein,
Harmonic analysis: real-variable methods, orthogonality, and oscillatory integrals,
Princeton Univ. Press, Princeton, NJ, 1993.




\bibitem{TR}  R.~M. Temam,
Navier-Stokes equations, revised edition, Studies in Mathematics and its Applications, 2, North-Holland, Amsterdam-New York, 1979.



\bibitem{WWV} W. von~Wahl,
Estimating $\nabla u$ by ${\rm div}\, u$ and ${\rm curl}\, u$,
Math. Methods Appl. Sci. {\bf 15} (1992), no.~2, 123--143.



\bibitem{WWZ} S. Wang, G. Wu, X. Zhong,
Global weak solutions and incompressible limit to the isentropic compressible Navier-Stokes equations in the half-plane with ripped density and large initial data, arXiv:2507.03505.



\bibitem{ZAA} A.~A. Zlotnik,
Uniform estimates and the stabilization of symmetric solutions of a system of quasilinear equations,
Differ. Equ. {\bf 36} (2000), no.~5, 701--716.



\end {thebibliography}
\end{document}